\documentclass[11pt]{amsart}

\usepackage{tikz}
\usepackage{pgfplots}
\usepackage{amsthm}
\usepackage{amsxtra}
\usepackage{amssymb}
\usepackage{graphicx}
\usepackage{comment}
\usepackage{array}
\usepackage{url}
\usepackage{color}
\usepackage{amsfonts}
\usepackage{textcomp}
\usepackage{perpage,enumerate}
\usepackage[colorlinks=true]{hyperref}
\usepackage{bbold}

\usepackage{caption}
\usepackage{subcaption}
\DeclareMathAlphabet{\mathpzc}{OT1}{pzc}{m}{it}
\usepackage{lmodern}

\usepackage{sidecap}

\DeclareMathOperator*{\oph}{Op_{h}}

\DeclareMathOperator*{\op}{Op}
\DeclareMathOperator*{\supp}{supp}

\DeclareMathOperator*{\WFh}{WF_{h}}
\DeclareMathOperator*{\WFhp}{WF_{h,\Ph{}{}}}
\DeclareMathOperator*{\WFhpi}{WF^{i}_{h,\Ph{}{}}}
\DeclareMathOperator*{\WFhpf}{WF^{f}_{h,\Ph{}{}}}
\DeclareMathOperator*{\loc}{loc}
\DeclareMathOperator*{\comp}{comp}
\DeclareMathOperator*{\Ell}{ell}
\DeclareMathOperator*{\MS}{MS_{h}}
\DeclareMathOperator*{\MSp}{MS_{h,\Ph{}{}}}

\DeclareMathOperator*{\sgn}{sgn}
\DeclareMathOperator*{\Id}{Id}

\DeclareMathOperator*{\Arg}{Arg}
\newcommand{\Hh}{H_{\operatorname{h}}}

\usepackage{chngcntr}

\graphicspath{{Figures/}}

\newtheorem{theorem}{Theorem}

\numberwithin{prop}{section}

\numberwithin{corol}{section}

\newtheorem{lemma}{Lemma}
\numberwithin{lemma}{section}

\numberwithin{conjecture}{section}

\newenvironment{remarks}[1][]{\begin{remark}\begin{trivlist}
\item[\hskip \labelsep {\bfseries #1}]\end{trivlist}\begin{itemize}}{\end{itemize}\end{remark}}

{\theoremstyle{definition}
\newtheorem{defin}{Definition}
\numberwithin{defin}{section}
}
\numberwithin{figure}{section}

\renewcommand{\Re}{\mathop{\rm Re}\nolimits}
\renewcommand{\Im}{\mathop{\rm Im}\nolimits}

\newcommand{\quotient}[2]{{\left.\raisebox{.2em}{$#1$}\middle/\raisebox{-.2em}{$#2$}\right.}}

\newcommand{\Spec}{\operatorname{Spec}}
\newcommand{\bl}{\begin{flushleft}}
\newcommand{\el}{\end{flushleft}}
\newcommand{\br}{\begin{flushright}}
\newcommand{\ert}{\end{flushright}}
\newcommand{\bc}{\begin{center}}
\newcommand{\ec}{\end{center}}

\newcommand{\recip}[1]{\frac{1}{#1}}
\newcommand{\imply}{\Rightarrow}

\newcommand{\complex}{\mathbb{C}}
\newcommand{\ints}{\mathbb{Z}}
\newcommand{\numList}{\begin{enumerate}}
\newcommand{\enumList}{\end{enumerate}}

\newcommand{\composed}{\text{\textopenbullet}}

\newcommand{\e}{\epsilon}

\newcommand{\re}{\mathbb{R}}

\newcommand{\nn}{\nonumber\\}
\newcommand{\la}{\langle}
\newcommand{\ra}{\rangle}

\newcommand{\mc}[1]{\mathcal{#1}}
\newcommand{\pO}{\partial\Omega}
\theoremstyle{remark}
\newtheorem{remark}{Remark}

\newcommand{\Deltad}[1]{\Delta_{#1,\delta}}

\newcommand{\asec}{\operatorname{arcsec}}
\newcommand{\Fh}{\mc{F}_h}

\newcommand{\dDl}{\partial_{\nu}\mc{D}\ell}
\newcommand{\Dl}{N}
\newcommand{\Sl}{\mc{S}\ell}
\newcommand{\D}{\mc{D}\ell}
\renewcommand{\S}{\mc{S}\ell}
\newcommand{\Cc}{C_c^\infty}
\renewcommand{\O}[1]{\mathpzc{O}_{#1}}
\renewcommand{\o}[1]{\mathpzc{o}_{#1}}
\newcommand{\opht}[1]{\operatorname{Op_{h,#1}}}
\newcommand{\wt}[1]{\widetilde{#1}}
\newcommand{\No}{N_2}
\newcommand{\Ph}[2]{\Psi^{#1}_{#2}}
\newcommand{\Do}{D_{\Omega}}
\newcommand{\m}{$$}
\usepackage{multicol}
\usepackage{enumitem}
\title[Quantum Sabine Law for Resonances]{The Quantum Sabine Law for Resonances in Transmission Problems}
\author{Jeffrey Galkowski}
\email{jeffrey.galkowski@stanford.edu}
\address{ Mathematics Department, Stanford University, Stanford, CA USA}

\begin{document}
\begin{abstract}
We prove a quantum version of the Sabine law from acoustics describing the location of resonances in transmission problems. This work extends the work of the author to a broader class of systems. Our main applications are to scattering by transparent obstacles, scattering by highly frequency dependent delta potentials, and boundary stabilized wave equations. We give a sharp characterization of the resonance free regions in terms of dynamical quantities. In particular, we relate the imaginary part of resonances or generalized eigenvalues to the chord lengths and reflectivity coefficients for the ray dynamics, thus proving a quantum version of the Sabine law.
\end{abstract}

\maketitle
\setcounter{tocdepth}{1} 
\tableofcontents
\section{Introduction}
In this paper we study scattering in systems where the metric or potential has a singularity along an interface. Metric examples include scattering in media having sharp changes of index of refraction \cite{Card, Card2,popVod}, in dielectric microcavities \cite{dielectric} and in fiber optic cables \cite{elliott2002fiber}. Schr\"odinger operators with a distributional potential along a hypersurface can be used to model quantum corrals, concert halls, and other thin barriers  \cite{Heller,Crommie}. Such potentials are also used to understand leaky quantum graphs \cite{Exner}.

Mathematically, an abrupt change in the index of refraction corresponds to a discontinuity in the metric along a hypersurface. Scattering in such situations has been studied in \cite{mourad,Card, Card2,popVod,popovNear} while scattering by certain distributional potentials has been studied in \cite{Galk,GalkCircle,GS}. These types of problems have also been studied from the point of view of propagation of singularities \cite{MelTayl, Miller,weiss1985reflection} and quantum chaos \cite{SafRaySplit}.

For a Schr\"odinger operator, $P$, on $L^2(\re^d)$ ($d$ odd) it is often possible to prove that solutions, $u$, to 
$$(\partial_t^2+P)u=0$$
have expansions roughly of the form 
\begin{equation}
\label{eqn:expand} u\sim \sum_{\lambda \in \text{Res}}e^{-it\lambda}u_\lambda
\end{equation}
where $\text{Res}$ is the set of \emph{scattering resonances} of $P$. Thus, the real and (negative) imaginary part of a scattering resonance correspond respectively to the frequency and decay rate of the associated resonance state, $e^{-it\lambda}u_\lambda$. This expression is similar to the expansion in terms of eigenvalues that one obtains when solving the wave equation on a compact manifold. Hence, for leaky systems, scattering resonances play the role of eigenvalues in the closed setting. 

To get a quantitative heuristic for the decay of waves (the imaginary part of resonances), we imagine that the interface for our problem occurs at $\pO$ for some $\Omega\Subset \re^d$. We then think of solving the wave equation 
$$(\partial_t^2+P)u=0\,,\quad u|_{t=0}=u_0,\quad u_t|_{t=0}=0$$
with initial data $u_0$ a wave packet (that is a function localized in frequency and space up to the scale allowed by the uncertainty principle) localized at position $x_0\in \Omega$ and frequency $\xi_0\in S^{d-1}$. We also assume that $P$ creates waves with speed $c$. The solution, $u$, then propagates along the billiard flow starting from $(x_0,\xi_0)$. At each intersection of the billiard flow with the boundary, the amplitude inside of $\Omega$ will decay by a factor, $R$, depending on the point and direction of intersection. Suppose that the billiard flow from $(x_0,\xi_0)$ intersects the boundary at $(x_n,\xi_n)\in\pO\times S^{d-1}$, $n>0$. Let $l_n=|x_{n+1}-x_n|$ be the distance between two consecutive intersections with the boundary (see Figure \ref{fig:reflectPic}). Then the amplitude of the wave decays by a factor $\prod_{i=1}^nR_i$ in time $\sum_{i=1}^nc^{-1}l_i$ where $R_i=R(x_i,\xi_i)$. The energy scales as amplitude squared and since the imaginary part of a resonance gives the exponential decay rate of $L^2$ norm, this leads us to the heuristic that resonances should occur at 
\begin{equation} \label{eqn:heurRes}\Im \lambda=\left.\overline{\log |R|^2}\right/(2c^{-1}\bar{l})\end{equation}
where the map $\bar{\cdot}$ is defined by $\bar{f}=\frac{1}{N}\sum_{i=1}^Nf_i$. 
In the early 1900s, Sabine \cite{Sabine} postulated that the decay rate of acoustic waves in a region with leaky walls is determined by the average decay over billiards trajectories. The expression \eqref{eqn:heurRes} provides a precise statement of Sabine's idea and, because resonances are spectral quantity, we refer to such an expression as a quantum Sabine law. We will show in Theorem \ref{thm:mainGeneral} that such a Sabine law holds for many different types of transmission problem.

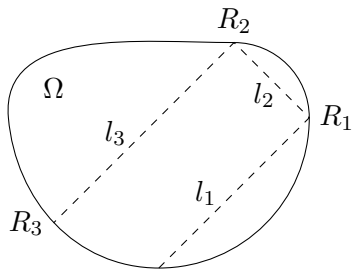
\begin{figure}
\centering
\begin{tikzpicture}
\draw (0,0) to [out=0,in=-90](2,2) to[out=90, in=0](1,3)to[out=180, in =95](-2,2)to [out=275,in=180](0,0);
\draw [dashed](0,0)to(2,2)node[right]{$R_1$}to (1,3)node[above]{$R_2$}to (1-2.4,3-2.4)node[left]{$R_3$};
\draw (0.9,1)node[left]{$l_1$} (1.4,2.6)node[below]{$l_2$}(1-1.3,3-1.2)node[left]{$l_3$};
\draw (-1.4,2.4)node{$\Omega$};
\end{tikzpicture}
\caption[Schematic of a wave packet inside a thin barrier]{\label{fig:reflectPic}The figure shows the path of a wave packet along with the lengths between each intersection ($l_i$) and the reflection coefficient at each point of intersection with the boundary ($R_i$). After each reflection with the boundary, the amplitude of the wave packet inside $\Omega$ decays by a factor of $R_i$. If the speed of the wave is $c$, the time between reflections is given by $c^{-1}l_i$.}
\end{figure}

Although the appearance of scattering resonances in \eqref{eqn:expand} is intuitive, a more mathematically useful definition of a scattering resonance is as a pole of the meromorphic continuation of 
$$(P-\lambda^2)^{-1}$$
from $\Im \lambda\gg 1$. This description allows us to show that the existence of a scattering resonance at $\lambda$ corresponds to the existence of a nonzero $\lambda-$outgoing solution to 
$$(P-\lambda^2)u=0.$$
By $\lambda$-outgoing we mean that there exists $g\in L^2_{\comp}(\re^d)$ and $M\geq 0$ such that 
$$u(x)=(R_0(\lambda)g)(x),\quad\quad |x|\geq M.$$
Here, $R_0(\lambda)$ is the meromorphic continuation of $(-\Delta-\lambda^2)^{-1}$ from $\Im \lambda\gg1$ as an operator $R_0(\lambda):L^2_{\comp}(\re^d)\to L^2_{\loc}(\re^d)$. (For a more complete description of mathematical scattering and further references, see~\cite{ZwScat})

We start by considering a few applications of our main theorem (see Theorem \ref{thm:mainGeneral}).

\subsection{Transparent Obstacles}
\label{sec:Omega}
Our first application is to scattering by a transparent obstacle. That is, an obstacle with different refractive index than the ambient medium. In particular, let $\Omega \Subset \re^d$ be strictly convex with smooth boundary,  $c\in \re_+\setminus\{1\}$ be the speed of light in $\Omega$, and $\aleph>0$ be a coupling parameter. In \cite{Card}, Cardoso, Popov, and Vodev show that the set of scattering resonances in this setting is given by $\lambda$ such that there is a non-zero solution to 
\begin{equation}
\label{intro-e:transparent}
\begin{cases}
(-c^2\Delta -\lambda^2)u_1=0&\text{ in }\Omega\\
(-\Delta -\lambda^2) u_2=0&\text{ in }\re^d\setminus\overline{\Omega}\\
u_1=u_2&\text{ on }\pO\\
\partial_\nu u_1-\aleph\partial_\nu u_2=0&\text{ on } \pO\\
u_2\text{ is $\lambda$-outgoing.}
\end{cases}
\end{equation}
We denote the set of such $\lambda$ by $\Lambda$. Here, $\nu$ denotes the outward unit normal to $\pO$. 

Let $T^* \partial\Omega$ be the cotangent bundle to $\partial\Omega$ and $B^*\partial\Omega$ denote the coball bundle of $\partial \Omega$. Let $\pi_x:T^*\pO\to \pO$ be the projection to the base. Then define $r,l_N,r_N\in C^\infty(B^*\pO)$ and 
$$l\in C^\infty(T^*\pO\times T^*\pO\setminus \{(x,\xi',x,\eta')\in T^*\pO\times T^*\pO\})\cap C(T^*\pO\times T^*\pO)$$
 by
\begin{equation}
\label{eqn:definitionsTransmision}
\begin{aligned}
r(x',\xi')&:=\frac{\sqrt{1-|\xi'|_g^2}-\aleph\sqrt{c^{2}-|\xi'|_g^2}}{\aleph \sqrt{c^2-|\xi'|_g^2}+\sqrt{1-|\xi'|_g^2}}& r_N(q)&:=\frac{\sum_{j=1}^N\log |r(\beta^j(q))|^2}{N}\\
l(q_1,q_2)&:=|\pi_x(q_1)-\pi_x(q_2)|& l_N(q)&:=\frac{\sum_{j=1}^{N}l(\beta^{j-1}(q),\beta^j(q))}{N}
\end{aligned}
\end{equation}
where $\beta:B^*\pO\to B^*\pO$ denotes the billiard ball map (see section \ref{sec:billiard}) and $|\xi'|_g$ denotes the norm induced on the fibers of $T^*\pO$ by the metric on $\re^d$. Then $r$ is the reflectivity for the transparent obstacle problem. Note that we take the branch of the square root so that $\sqrt{-1}=i$ and place the branch cut on the negative imaginary axis.
\begin{remarks}
\item We will use $\xi'$ to denote coordinates in the fiber of $T^*\pO$ and $q$ to denote points in $T^*\pO$ throughout this paper.
\item Note that the $\log$ in the definition of $r_N$ appears because we measure exponential rates of decay and the reflection coefficient acts by multiplication.
\end{remarks}
\begin{theorem}
\label{thm:mainTransparent}
Let $\Omega\Subset \re^d$ be strictly convex with smooth boundary and suppose that $0<c\neq1$, $\aleph>0$. Then for all $M,\,\e>0$ there exists $\lambda_0>0$  such that for $\lambda\in \Lambda $ with $\Re \lambda \geq \lambda_0$ and $\Im \lambda \geq -M\log \Re \lambda$,
\begin{equation}
\label{e:transRange}\sup_{N>0}\inf_{|\xi'|_g\leq 1}\frac{r_N}{2c^{-1}l_N}-\e\leq \Im \lambda \leq \inf_{N>0}\sup_{|\xi'|_g\leq 1}\frac{r_N}{2c^{-1}l_N}+\e.\nonumber
\end{equation}
Moreover, for every $\aleph$, $c$ as above, and $K>0$, this bound is sharp in the region $\Im \lambda \geq -K$ when $\Omega=B(0,1)\subset \re^2$.
 
\end{theorem}
\begin{remarks}
\item The lower bound in Theorem \ref{thm:mainTransparent} is nontrivial, i.e. $|r(x',\xi')|>0$, if either $c<1$ and $\aleph<c^{-1}$, or $c>1$ and $\aleph>c^{-1}$. This corresponds to \emph{transverse electric waves} (TE). The opposite case, when there is no lower bound, corresponds to \emph{transverse magnetic waves} (TM). In the TM case, the angle at which $r(x',\xi')=0$ is called the \emph{Brewster angle} (\cite[Chapter 13]{Brewster}). At this angle, there is complete transmission of the wave in the ray dynamics picture.
\item The upper bound in Theorem \ref{thm:mainTransparent} is nontrivial if $c>1$. When $c<1$, Popov and Vodev \cite{popovNear} show that the presence of total internal reflection (see Figure \ref{fig:refraction}) produces resonances $\{\lambda_n\}_{n=1}^\infty$ with $\Re \lambda_n\to \infty$ and $\Im \lambda_n=\O{}((\Re \lambda_n)^{-\infty}).$ 
\item The bounds for resonances given in Theorem \ref{thm:mainTransparent} match our prediction \eqref{eqn:heurRes}.
\end{remarks}

\begin{figure}
\includegraphics[scale=.75]{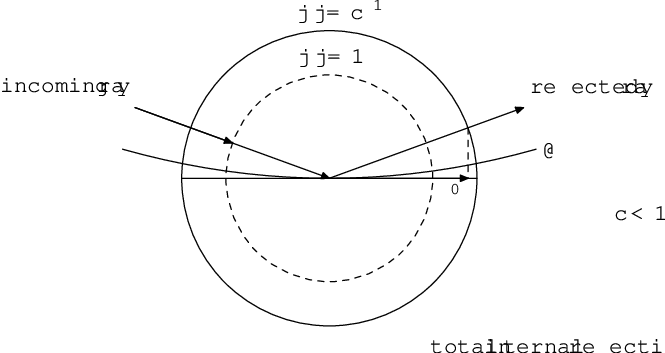}
\includegraphics[scale=.75]{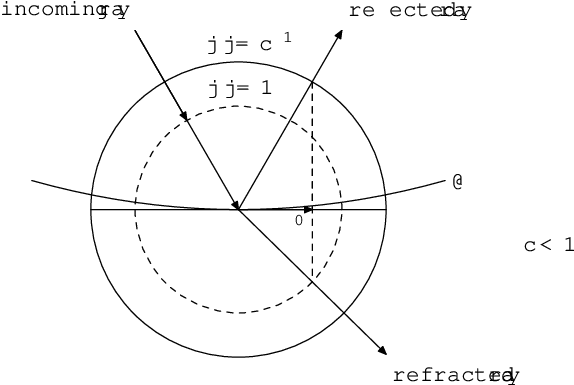}\\[20pt]
\includegraphics[scale=.75]{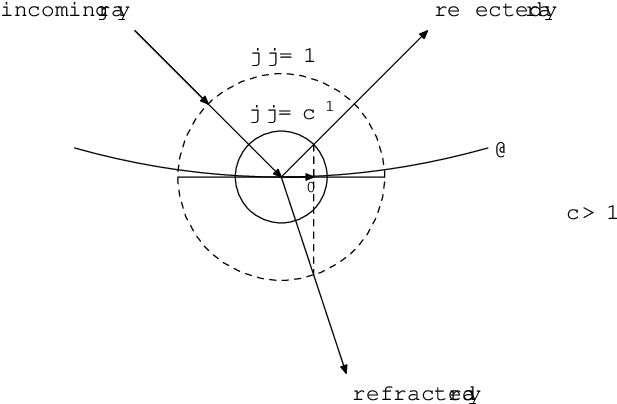}
\caption{\label{fig:refraction} The figure shows the geometry of reflection and refraction at the boundary of an interface between a medium with speed of light $c$ and one with speed of light $1$. Total internal reflection occurs when the incoming ray does not project onto the ball of radius 1 in the $\xi'$ variable.}
\end{figure}

Theorem \ref{thm:mainTransparent} improves upon the results of Cardoso--Popov--Vodev \cite{Card,Card2} by giving sharp estimates on the sizes of the resonance free regions as well as expanding the range of parameters, $\aleph$, for which we have only a band of resonances. 

\begin{figure}
\centering
\begin{subfigure}[b]{\textwidth}
\includegraphics[width=\textwidth]{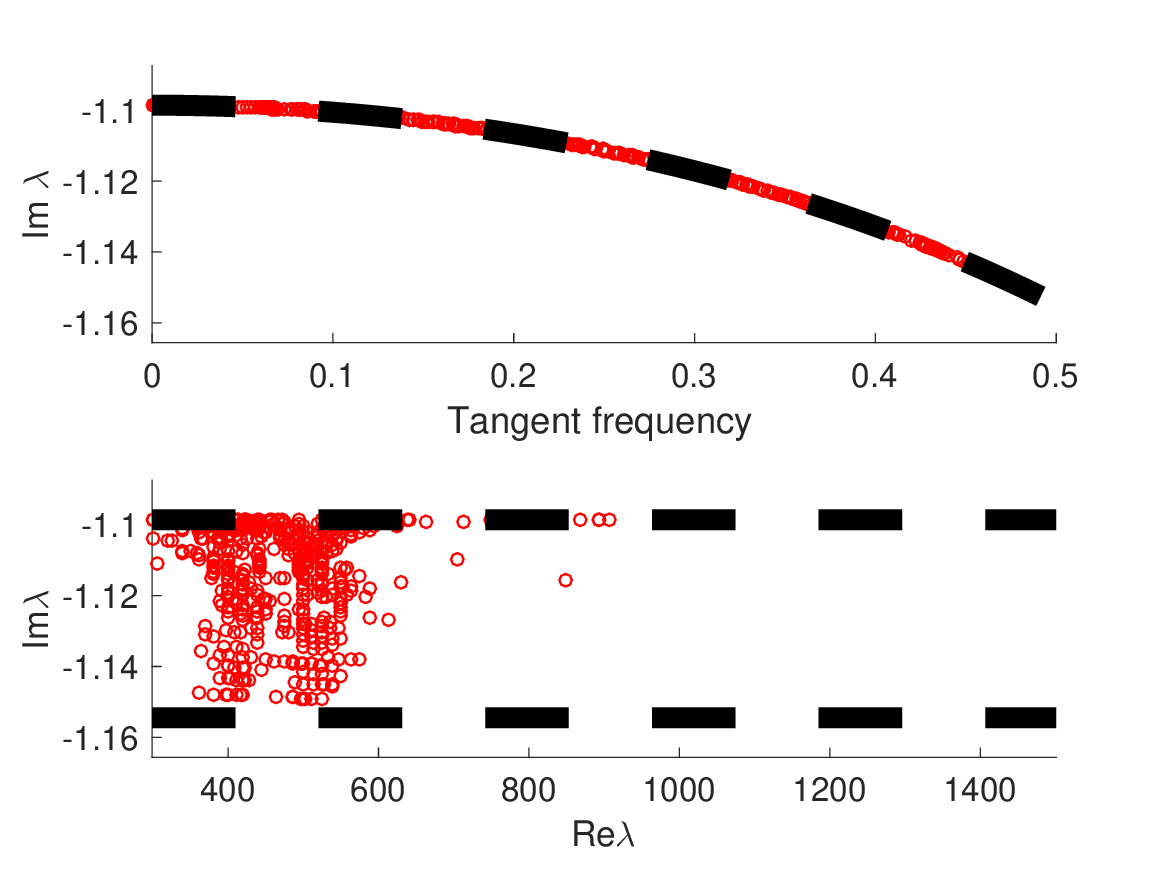}
\end{subfigure}\\
\vspace{.4cm}
\begin{subfigure}[b]{\textwidth}
\centering
\begin{tabular}{|>{\centering\arraybackslash} m{4cm}|>{\centering\arraybackslash} m{2cm}|>{\centering\arraybackslash} m{4cm}|}
\hline
&Decay Rate&Tangential Frequency\\
\hline
Quantum Quantity&$\Im \lambda $&\raisebox{-.2cm}{$\displaystyle\frac{n}{\Re \lambda}$}\\[3ex]
\hline
Classical Quantity&$\displaystyle\frac{cr_1}{2l_1}$&\raisebox{-.2cm}{$c\xi'$}\\[3ex]
\hline
\end{tabular}
\end{subfigure}
\caption{We show numerically computed resonances for the transparent obstacle problem with $c=2$ and $\aleph=1$ when $\Omega=B(0,1)\subset \re^2$. (See Figure \ref{fig:circleLots} for other values of $c$ and $\aleph$.) In this case, we expand the solutions to \eqref{intro-e:transparent} as $u_i(r,\theta)=\sum_{n}u_{i,n}(r)e^{in\theta}$ and solve for some of the resonances with $\Re \lambda \sim 500$. In the lower graph, the red circles show $\Im \lambda$ vs. $\Re \lambda$. The dashed black lines show the upper and lower bounds for $\Im \lambda$ (since $\aleph$ is in the TE range with have both an upper and lower bound) from Theorem \ref{thm:mainTransparent}. Notice that by orthogonality of $e^{in\theta}$ and $e^{im\theta}$ for $m\neq n$, the pair $(u_{1,n}e^{in\theta}, u_{2,n}e^{in\theta})$ satisfies \eqref{intro-e:transparent}. In the top graph, the red circles show $\Im \lambda$ vs. $n/\Re \lambda$ for such pairs.  The dashed curve shows a plot of $c\frac{r_1}{2l_1}(c\xi')$, the decay rate predicted for a billiards trajectory traveling with scaled tangent frequency $c\xi'$. See the table for the relationship between the points $(\Im \lambda, n/\Re \lambda)$ and $(cr_1/2l_1(c\xi'),c\xi')$ predicted by the quantum Sabine law.\label{fig:circle}}
\end{figure}

\subsection{Highly Frequency Dependent Delta Potentials}
\label{sec:potential}
Let $\Ph{\infty}{}(\pO)$ denote the set of semiclassical pseudodifferential operators of all orders whose seminorms are bounded by a constant independent of $h$ so that $h^{-N}\Ph{\infty}{}(\pO)$ denotes those whose seminorms are bounded by $h^{-N}$ (see section \ref{sec:semiclassicalPreliminaries} for more details).

We next consider operators of the form 
\begin{equation}-h^2\Delta +h(h\delta_{\pO}\otimes V)=:-h^2\Deltad{\pO}.\label{eqn:delta}\end{equation}
where $h\in (0,1]$, is a semiclassical parameter that should be thought of as the wavenumber (i.e. the inverse of the frequency), $V\in h^{-N}\Ph{\infty}{}(\pO)$, and for $u,w\in \Cc(\re^d)$
\begin{equation}
\label{e:tensorAction}\la (\delta_{\pO}\otimes V)u,w\ra:=\int_{\pO}(Vu)(x)w(x)d\sigma(x)
\end{equation}
and $\sigma$ is the surface measure of $\pO$.
(See \cite[Section 2.1]{GS} for the formal definition of this operator.)
These operators are used as models for quantum corrals \cite{Heller,Crommie} as well as concert halls, leaky quantum graphs \cite{Exner} and other thin barriers. 

In a typical physical system, the interaction between a potential and wave depends on the frequency of the interacting wave. Therefore, we are motivated to consider $h$-dependent potentials $V$. Moreover, if one considers the delta interaction in 1 dimension
$$-\Delta+\delta(x_1)\otimes 1$$
and rescales to $y=hx$, we obtain
\begin{equation}
\label{eqn:quantumPoint}
-h^2\Delta_y^2+\delta(y_1/h)\otimes 1=-h^2\partial_y^2+h\delta(y_1)\otimes 1
\end{equation}
which corresponds to $V=h^{-1}$ in \eqref{eqn:delta}. The operator \eqref{eqn:quantumPoint} describes the \emph{quantum point} interaction \cite{Miller}.

Another motivation for highly frequency dependent delta potentials is the following wave equation
$$\left\{\begin{gathered}(\partial_t^2-\Delta +i(\delta_{\pO}\otimes(\la a(x),\partial_x\ra +a_0(x)\partial_t))u=F\text{ in }\re^d\\
F\in L^2_{\comp}((0,\infty)_t\times \re^d),\quad u=0\text{ on }t<0
\end{gathered}\right.$$
where $a, a_0\in C^\infty(\pO;\re)$, and the tensor product acts as in \eqref{e:tensorAction}.
Then, taking the time Fourier transform 
$$\mc{F}_{t\to \lambda}u(x,\lambda):=\int_0^\infty e^{it\lambda}u(x,t)dt,$$
gives with $\lambda=z/h$,
$$(-h^2\Delta -z^2+z(h\delta_{\pO}\otimes(\la z^{-1}a, hD_x\ra + a_0))\mc{F}_{t\to z/h}u=\mc{F}_{t\to z/h}F.$$
\begin{remark} Note that we have switched the usual convention for the Fourier transform in our definition of $\mc{F}_{t\to \lambda}$ so that the integral converges absolutely for $\Im \lambda >0$.
\end{remark}

In \cite{GS}, Smith and the author show that the set of scattering resonances, $\Lambda(h)$, is equal to the set of $z$ such that there is a non-zero solution to 
$$\begin{cases}(-h^2\Delta-z^2)u_1=0&\text{ in }\Omega\\
(-h^2\Delta -z^2)u_2=0&\text{ in }\re^d\setminus \overline{\Omega}\\
u_1=u_2&\text{ on }\pO\\
\partial_\nu u_1-\partial_\nu u_2+Vu_1=0&\text{ on }\pO\\
u_2\text{ is }z/h\text{ outgoing}.
\end{cases}$$
Denote by 
\begin{equation}
\label{e:lLog}\Lambda_{\log}(h):=\{z\in \Lambda(h)\,|\, z\in [1-Ch,1+Ch]+i[-Mh\log h^{-1}, 0]\}.
\end{equation}
For $V\in h^{-N}\Ph{\infty}{}(\pO)$ with real valued symbol, $\sigma(V)$, the reflectivity, $r\in C^\infty(B^*\pO)$, is given by
$$r(x',\xi'):=\frac{h\sigma(V)}{2i\sqrt{1-|\xi'|_g^2}-h\sigma(V)}$$
with $r_N(q)$ and $l_N(q)$ as in \eqref{eqn:definitionsTransmision}.
For a more general definition of $r$ see \eqref{eqn:reflect} and for $r_N$ see \eqref{eqn:defineAverageReflection}.

Let $\Ph{m}{}(\pO)$ denote the set of semiclassical pseudodifferential operators of order $m$ (see Section \ref{sec:semiclassicalPreliminaries}) and $\Lambda_{\log}(h)$ be as in \eqref{e:lLog}. Next, let
\begin{equation}
\label{e:Airy1} 
\begin{gathered} Ai(s):=\frac{1}{2\pi}\int e^{i(st+t^3/3)}dt,\quad A_-(s):=Ai(e^{2\pi i/3}s),\quad \Phi_-(s):=A_-'(s)/A_-(s),\\
0>\zeta_1>\zeta_2>\dots \text{ be the zeros of }Ai(s).
\end{gathered}
\end{equation}
Finally, let $Q(x,\xi')\in C^\infty(T^*\pO)$ be the symbol of the second fundamental form to $\pO$. Then we have:
\begin{theorem}
\label{thm:deltaMain}
Let $\Omega \Subset \re^d$ be strictly convex with smooth boundary, $\alpha \geq -1$, and suppose that $V\in h^\alpha\Ph{1}{}(\pO)$ is self adjoint with $\sigma(V)\geq 0$ and 
$\sigma(V)>c>0$ in a neighborhood of $\{|\xi'|_g=1\}$.
\begin{enumerate} 
\item Suppose that $\alpha>-5/6$. Then for all $\e,\,N_1>0$ there exist $\e_1>0$, $h_0>0$ such that for $0<h<h_0$ 
$$\Lambda_{\log}(h)\subset \left\{\frac{\Im z}{h}\leq \inf_{N\leq N_1}\sup_{|\xi'|<1-\e_1}\frac{r_N}{2l_N}+\e\right\}.$$
\item Suppose that $-5/6\geq \alpha \geq -1$. Then for all $\e>0$, $M>0$, there exists $h_0>0$  such that for $0<h<h_0$  
$$\Lambda_{\log}(h)\subset \bigcup_{j=1}^M\left\{ B_{\min{}} -\e\leq \frac{h^{2/3}\Im z}{\Im \Phi_-(\zeta_j)} \leq B_{\max{}}+\e\right\}\bigcup \left\{\frac{h^{2/3}\Im z}{\Im \Phi_{-}(\zeta_{M+1})}\geq B_{\min{}}-\e\right\}$$
where 
$$B_{\max}:=\sup_{|\xi'|_g=1}\frac{2^{1/3}Q(x,\xi')^{4/3}}{|\sigma(V)(x,\xi')|^2},\quad\quad B_{\min{}}:=\inf_{|\xi'|_g=1}\frac{2^{1/3}Q(x,\xi')^{4/3}}{|\sigma(V)(x,\xi')|^2}.$$
\end{enumerate}
Moreover, these estimates are sharp in the case of $\Omega=B(0,1)\subset \re^2$ with $V\equiv 1.$
\end{theorem}

\begin{figure}
\includegraphics{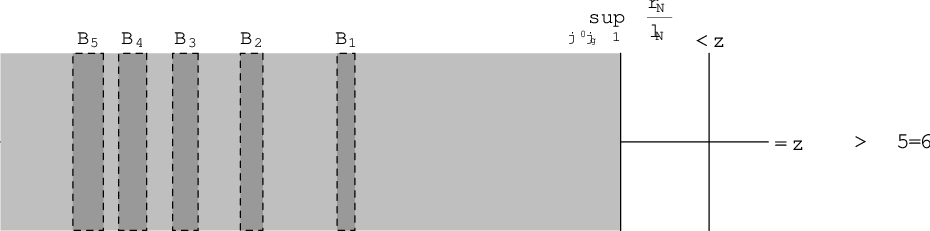}\\[20pt]
\includegraphics{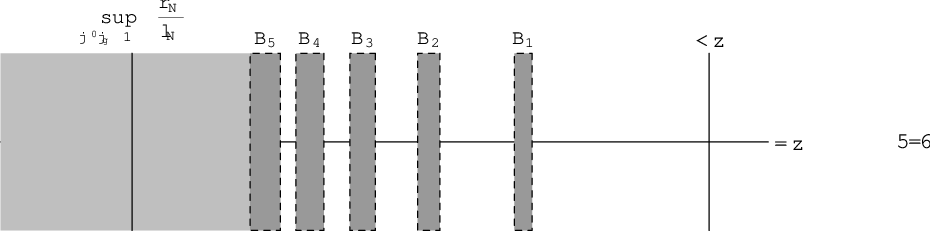}
\caption{This figure shows a schematic representation of the resonance free regions from Theorem \ref{thm:deltaMain} for $\alpha>-5/6$ on the top and $\alpha\leq -5/6$ on the bottom. Resonances lie in the dark grey bands, 
$\displaystyle\mc{B}_j:=\left\{ B_{\min{}} -\e\leq \frac{h^{2/3}\Im z}{\Im \Phi_-(\zeta_j)} \leq B_{\max{}}+\e\right\},$
 or the light gray shaded region, but not in the white regions. Note that the bands start to group closer together as they go deeper into the complex plane. Thus, there will be only a finite number of bands if 
$\frac{B_{\max{}}}{B_{\min{}}}\neq 1.$ See also Figures \ref{fig:numRes2} and \ref{fig:resBands} for numerically computed resonances in the case of the disk where $\frac{B_{\max{}}}{B_{\min{}}}=1$ when $V\equiv h^\alpha$.}
\end{figure}

\begin{figure}
\includegraphics[width=\textwidth]{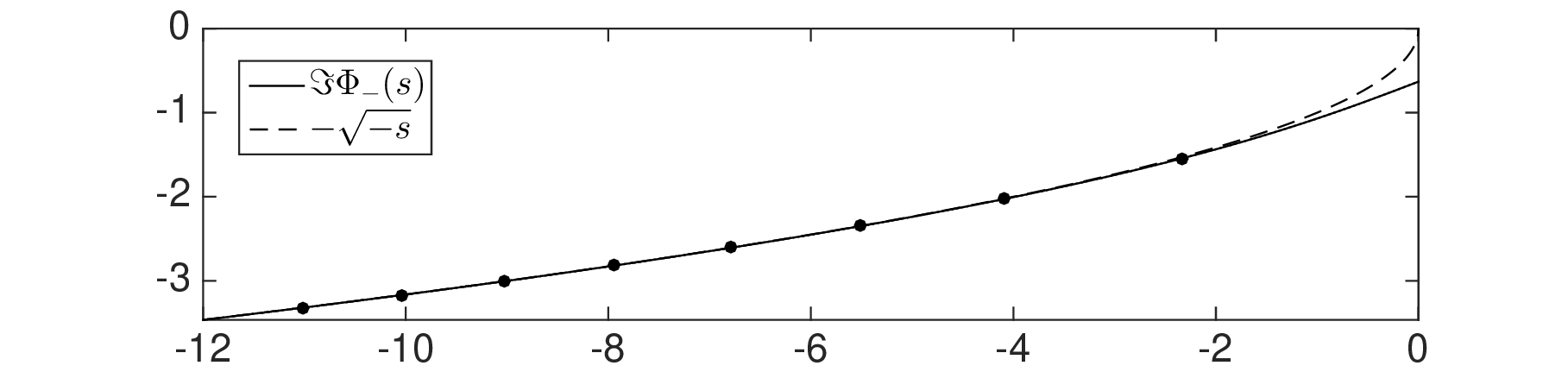}
\caption{\label{fig:phim}This figure shows $\Im \Phi_-$ in the solid line and $-\sqrt{-s}$ in the dashed line. The black dots are placed at $(\zeta_j,\Im \Phi_-(\zeta_j)).$}
\end{figure}

\begin{figure}
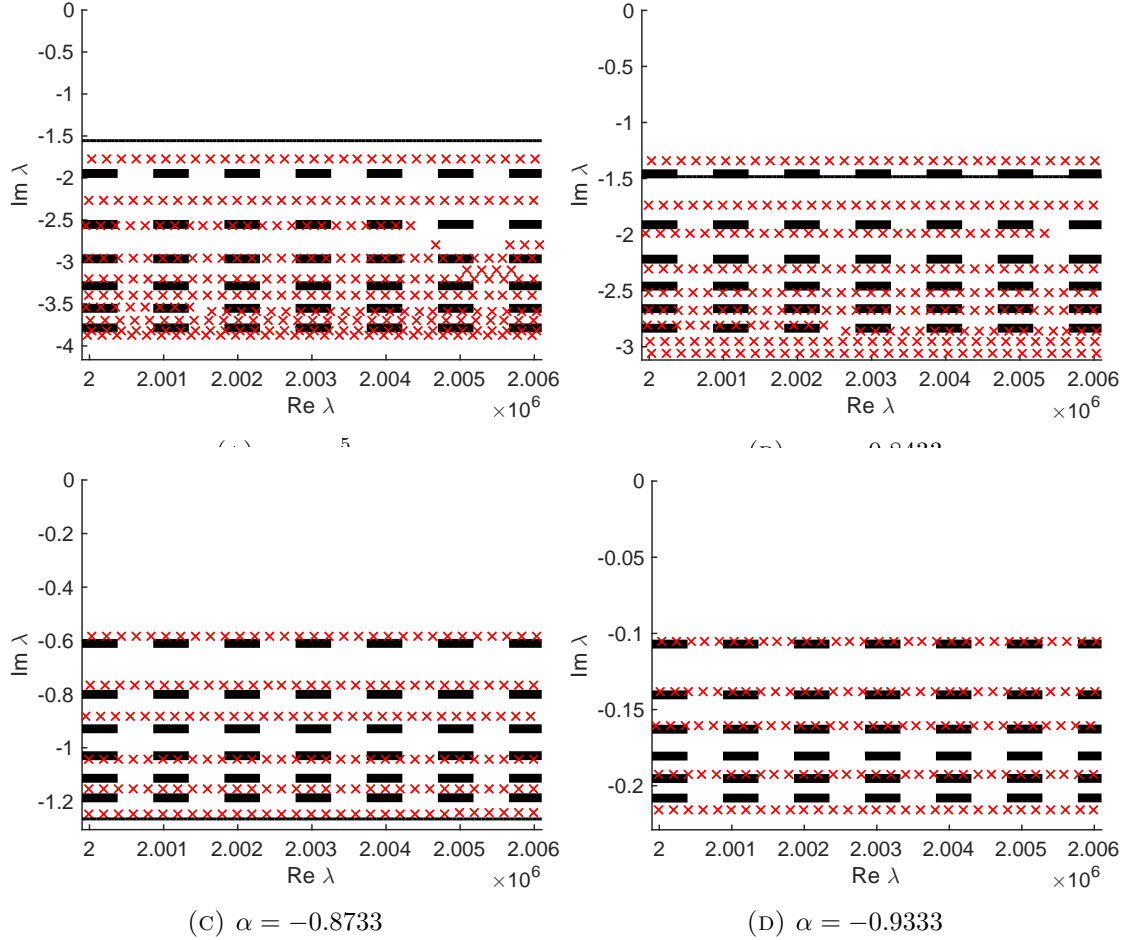

\centering
\begin{subfigure}[b]{.45\textwidth}
\includegraphics[width=\textwidth]{5-6.eps}
\caption{$\alpha=-\frac{5}{6}$}
\end{subfigure}
\begin{subfigure}[b]{.45\textwidth}
\includegraphics[width=\textwidth]{0-8433.eps}
\caption{$\alpha=-0.8433$}
\end{subfigure}
\begin{subfigure}[b]{.45\textwidth}
\includegraphics[width=\textwidth]{0-8733.eps}
\caption{$\alpha=-0.8733$}
\end{subfigure}
\begin{subfigure}[b]{.45\textwidth}
\includegraphics[width=\textwidth]{0-9333.eps}
\caption{$\alpha=-0.9333$}
\end{subfigure}
\caption[Numerical computation of resonances for the disk with highly frequency dependent potential large $\Re \lambda$]{We show resonances for the delta potential on the circle with $\Re \lambda \sim 10^{6}$, $V\equiv (\Re \lambda)^{-\alpha}$ and several $\alpha$. The plots show $\Im \lambda $ vs. $\Re \lambda $ in each case. The solid black line shows the (logarithmic) bound for resonances coming from non-glancing trajectories and the dashed black lines show the first few (polynomial) bands of resonances from near glancing trajectories. Since the solid black line is above the dashed black lines at $\alpha=-5/6$, it is necessary to go to still larger $\Re \lambda$ to see the transition to resonances with fixed size imaginary parts. However, at $\alpha<-5/6$, we start to see better agreement with the bands of resonances predicted in Theorem \ref{thm:deltaMain}.}
\label{fig:numRes2}
\end{figure}
\begin{figure}
\centering
\includegraphics[width=.75\textwidth]{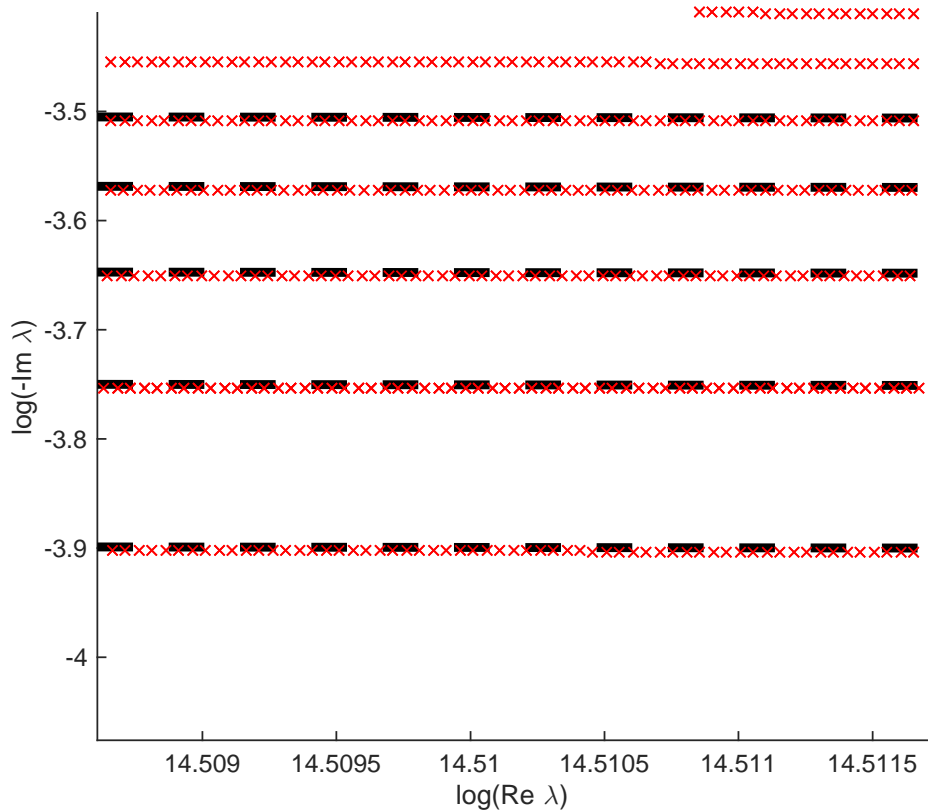}
\caption[Numerical computation of resonance bands for the disk (log-log plot)]{We show a plot of resonances for the delta potential on the disk with $V\equiv \Re \lambda$. In particular, we show $\log(\Re \lambda)$ vs. $\log(-\Im \lambda)$ for $\Re \lambda \sim 10^6$. The bands predicted by Theorem \ref{thm:deltaMain} are shown by the black dashed lines. \label{fig:resBands}}
\end{figure}
Theorem \ref{thm:deltaMain} verifies several conjectures from \cite{Galk} and generalizes the results from \cite{GalkCircle} to arbitrary convex domains. It also provides a second general class of examples that may have resonances with $-\Im z/h\sim ch^\gamma$ for some $\gamma>0$. That is, resonances converging to the real axis at a fixed polynomial rate, but no faster. Compared to the work in \cite[Theorem 5.4]{Galk}, Theorem \ref{thm:deltaMain} allows for potentials that depend more strongly on frequency. When the dependence is strong enough ($\alpha\leq -5/6$), the new phenomenon of a band structure appears. 

\begin{remarks}
\item Under the pinching condition, 
$$\frac{B_{\min{}}}{B_{\max{}}}>\frac{\Im \Phi_{-}(\zeta_j)}{\Im \Phi_-(\zeta_{j+1})}$$
there is a gap between the $j^{\text{th}}$ and $(j+1)^{\text{th}}$ band of resonances given by Theorem \ref{thm:deltaMain} for $\alpha\leq -5/6$. For a plot of $\Im \Phi_-(s)$ see Figure \ref{fig:phim}.
\item To see that the resonance bands in Theorem \ref{thm:deltaMain} for $\alpha\leq -5/6$ agree with those in \cite{GalkCircle}, observe that 
$$\Im \frac{A_-'(\zeta_j)}{A_-(\zeta_j)}=\Im \frac{2\pi Ai'(\zeta_j)A_-'(\zeta_j)}{e^{5\pi i/6}}=-\frac{2\pi Ai'(\zeta_j)Ai'(\zeta_j)}{2|A_-(\zeta_j)Ai'(\zeta_j)|^3(2\pi)^3}=-\frac{1}{8\pi ^2|A_-(\zeta_j)^3Ai'(\zeta_j)|}.$$
\end{remarks}

\subsection{Boundary stabilization problem}
Our final application of Theorem \ref{thm:mainGeneral} is to a boundary stabilized wave equation
\begin{equation}
\label{eqn:boundStable}\left\{\begin{aligned}(\partial_t^2-\Delta)u&=F\text{ in }\Omega\\
\partial_\nu u+a(x)\partial_t u&=0\text{ on }\pO\\
F&\in L^2_{\comp}((0,\infty)_t\times \Omega)\\
u&\equiv 0\text{ on }t<-1
\end{aligned}\right.
\end{equation}
with $0\leq  a(x)\in C^\infty(\pO;\re)$. It is not hard to see that the energy
$$E(t):=\frac{1}{2}\left(\|\partial_t u\|^2+\|\nabla u\|^2\right)$$
for the corresponding initial value problem is nonincreasing. The study of \eqref{eqn:boundStable} has a long history, see \cite{BardLebeau} and the references therein. In \cite{BardLebeau}, Bardos, Lebeau, and Rauch give nearly sharp conditions on $a$ to guarantee exponential decay of the energy. 

Here, we impose the strongly dissipative condition $0<a_0\leq a$ and study the asymptotic ($|\Re \lambda|\gg 1$) spectral gap for the corresponding stationary problem. That is, taking the Fourier transform in time, we study
\begin{equation}\label{eqn:stationaryBoundStable}\left\{\begin{aligned}(-\Delta -\lambda^2)\mc{F}_{t\to \lambda}u&=\mc{F}_{t\to \lambda}F&\text{ in }\Omega\\
(\partial_\nu  -i\lambda a(x))\mc{F}_{t\to \lambda}u&=0&\text{ on }\pO.
\end{aligned}\right.
\end{equation} 

In \cite{CardVod}, the authors show the existence of a spectral gap in a much more general, but still strongly dissipative, situation. Here, we give estimates on the size of the gap. Let $\Lambda$ denote the set of $\lambda$ so that \eqref{eqn:stationaryBoundStable} has a nonzero solution. The reflectivity, $r\in C^\infty(B^*\pO)$, for this problem is given by
$$r(x',\xi'):=\frac{\sqrt{1-|\xi'|_g^2}-a(x')}{a(x')+\sqrt{1-|\xi'|_g^2}}$$
and $l_N,\,r_N$ as in \eqref{eqn:definitionsTransmision}.
\begin{theorem}
\label{thm:boundaryStable}
Let $\Omega\Subset \re^d$ be strictly convex with smooth boundary and $a(x)\geq a_0>0$. Then for all $\e, M>0$ there exist $\lambda_0>0$ such that for $\lambda\in \Lambda$ with $|\Re \lambda|>\lambda_0$ and $\Im \lambda\geq -M\log |\Re \lambda|$,
\begin{equation}
\label{eqn:bstable}
\sup_{N>0 }\inf_{|\xi'|_g\leq 1}\frac{r_N}{2l_N}-\e\leq \Im \lambda \leq \inf_{N>0}\sup_{|\xi'|_g\leq 1}\frac{r_N}{2l_N}+\e.
\end{equation}
\end{theorem}
Note that Theorem \ref{thm:boundaryStable} can also be obtained from the results of Koch--Tataru \cite{KochTataruStable}. Indeed, the result contained there actually implies a stronger estimate than \eqref{eqn:bstable} in the case of \eqref{eqn:stationaryBoundStable}. We include this application to give a new proof of those results in this special case and to show that our analysis may be applied even to non-transmission problems. Moreover, note that the operator $a\partial_t$ can be replaced by a much more general pseudodifferential operator and our methods still apply.

\subsection{The general setup - a generalized boundary damped wave equation}
\label{sec:mainThm}
Theorems \ref{thm:mainTransparent}, \ref{thm:deltaMain}, and \ref{thm:boundaryStable} are a consequence of analysis of the boundary damped problem 
\begin{equation} 
\label{eqn:mainGeneral}
\left\{\begin{aligned} (-h^2\Delta -z^2)u=w&\text{ in }\Omega\\
h\partial_\nu u+Bu=hv&\text{ on }\pO
\end{aligned}\right.
\end{equation} 
with $\Re z \sim 1$. Here, the operator $B$ plays the role of damping waves upon interaction with the boundary and encodes the interaction with the exterior of $\Omega$ in the case of scattering problems. 

Let $N_2(z/h)$ denote the outgoing Dirichlet to Neumann map for $\re^d\setminus \overline\Omega$. That is, the map given by $C^\infty(\Omega)\ni f\mapsto -\partial_\nu u$ where $u$ solves
$$\begin{cases}(-h^2\Delta -z^2)u=0&\text{ in }\re^d\setminus\overline{\Omega}\\
u|_{\pO}=f\\
u\text{ is $z/h$ outgoing.}
\end{cases}
$$
We assume that $B=hN_2(z/h)+hV(z)$ where $V$ is in a certain second microlocal class of pseudodifferential operators which we specify later.  
\begin{remark}
By replacing $\tilde{h}=hE$ and $B(h)=EB(\tilde{h}/E)$, $\tilde{z}=Ez$ we may work with  $\Re z\sim E$. Notice that $z/h=\tilde{z}/\tilde{h}$ so operators that are functions of $z/h$ do not change under this rescaling.
\end{remark}

We first introduce some notation.. Let 
$$D_{M}(h):=[1-h,1+h]+i[-Mh\log h^{-1},Mh\log h^{-1}].$$
Let $\gamma:H^s(\re^d)\to H^{s-1/2}(\pO)$, $s>1/2$ be the restriction operator. Then the \emph{single layer operator} is given by
$$G(z/h):=\gamma R_0(z/h)\gamma^*.$$
Recall that $R_0(\lambda)$ is the meromorphic continuation of $(-\Delta-\lambda^2)^{-1}$.  
From \cite[Lemma 4.25]{Galk} \cite[Proposition 4.1]{HaZe} (see also Lemma \ref{lem:decompose}), we have that 
$$G(z/h)=G_{\Delta}(z/h)+G_B(z/h)+G_g(z/h)+\O{\mc{D}'(\pO)\to C^\infty(\pO)}(h^\infty)$$
where $G_{\Delta}$ is pseudodifferential, $G_B$ is a semiclassical Fourier integral operator associated to the billiard ball map (see section \ref{sec:semiclassicalPreliminaries} for the definition of semiclassical Fourier integral operators), and $G_g$ is microlocalized near $|\xi'|_g=1$. Let $m\geq 0$ and $\Ph{0,m}{2/3}(|\xi'|_g=E')$ denote the set of pseudodifferential operators that are second microlocalized near $|\xi'|_g=E$ (see section \ref{sec:secondMicrolocal}).

We now introduce assumptions on $V$.  For $a_1\in \re$, $\alpha\geq -1$, $E'\in \re\setminus \{1\}$, $\delta>0$, $M,M_1,M_2>0$, $0<\e<1/2$. 
Let $\la \cdot \ra \in C^\infty(T^*\pO)$ be given by $\la \xi'\ra:=(1+|\xi'|_g^2)^{1/2}$. We assume that 

\begin{gather} 
V= a_1N_2(z/h)+V_1,\quad\quad  V_1\in  h^\alpha \Ph{0,m}{2/3}(|\xi'|_g=E'),\quad\quad
 V \text{ is elliptic on }\left||\xi'|_g-1\right|<\delta,\label{a1}\\
\begin{aligned}
\left|1+\frac{h\sigma(V)}{2\sqrt{|\xi'|_g^2-1}}\right|&\geq \delta\left( \left\la \frac{h^{1+\alpha}}{\sqrt{|\xi'|_g^2-1}}\right\ra+\la \xi'\ra^{m-1}\right)\quad\quad&&|\xi'|_g>1+M_1h^{2/3},\\
\left|1+\frac{hi\sigma(V)}{2\sqrt{1-|\xi'|_g^2}}\right|&\geq \delta \left\la \frac{h^{1+\alpha}}{\sqrt{1-|\xi'|_g^2}}\right\ra\quad\quad&&|\xi'|_g\leq 1-h^{\e},
\end{aligned}\label{a2}\\
  V(z)\text{ is an analytic family of operators for }z\in D_{M}(h),\label{a3}\\
\log \left(1+\frac{h\sigma(V)}{\sqrt{|\xi'|_g^2-1}}\right)\text{ exists and is smooth on }T^*\pO\setminus\{|\xi'|_g\leq M_2\}.\label{a4}
\end{gather}
We say that $\mc{AV}(a_1,\alpha, E', m, \delta,M,M_1,M_2,\e)$ holds when \eqref{a1}-\eqref{a4} hold.

We now give a heuristic understanding of \eqref{a1}-\eqref{a4}. The assumption \eqref{a1} describes the structure of the operator $V$ in particular, allowing us to include copies of $N_2(z/(hE'))$ which encodes the exterior behavior of waves at speed $\sqrt{E}^{-1}$. We assume that $V$ is elliptic on $|\xi'|_g=1$, the glancing set for the problem inside $\Omega$, to simplify some of our analysis and guarantee that glancing effects play a nontrivial role in the analysis. Notice in particular that if $\WFh(V)\cap\{ |\xi'|_g=1\}=\emptyset$, then waves near glancing escape $\Omega$ essentially without reflection. This ellipticity assumption is not necessary for our analysis, but since the main advantage of the present paper over \cite{Galk} is the analysis near glancing, we include it to simplify our presentation. 

Next, \eqref{a2} guarantees that the problem is locally elliptic in the sense that if a singularity emerges from $(x',\xi')\in T^*\partial\Omega$, then there must be a singularity coming in to $(x',\xi')$. That is, the boundary cannot produce singularities spontaneously. Furthermore, this guarantees that there are no solutions microlocalized in the elliptic region $|\xi'|_g>1$. 

Finally, \eqref{a3} and \eqref{a4} are used to guarantee that the resolvent operator corresponding to \eqref{eqn:mainGeneral} is meromorphic and hence that it makes sense to discuss its poles.

\begin{remarks}
\item For the definition of ellipticity of $V$, see sections \ref{s:ell1} and \ref{s:ell2}.
\item 
These are not quite the most general assumptions we can make on $V$, but in practice all situations we have in mind fall into this category. For the most general assumptions, see \eqref{resFree-e:ellipticAssumption} and for the statement of the Theorem in that case, see Theorem \ref{thm:fullGenerality}.
\item We make the assumption that $V$ is elliptic near glancing so there is no rapid loss of energy near glancing. We could remove this assumption, but there would be no new phenomena and the analysis near glancing would be more complicated.
\item The final assumption \eqref{a4} (used to prove that the underlying problem is Fredholm) is satisfied for example when $m<1$, or when $m\geq 1$ and for some $\theta_0$ fixed and $\sigma(\tilde{V})$ real valued
$$\sigma(V)=e^{i\theta_0}\sigma(\tilde{V}),\quad \quad |\xi'|_g\geq M_2.$$
\end{remarks} 
Lt $\chi\in \Cc(\re)$ with $\chi \equiv 1$ near 0, and define 
\begin{gather}\label{eqn:reflect}\begin{gathered}\chi_\e\in C^\infty(T^*\pO),\quad\quad \chi_\e(x,\xi'):=1-\chi\left(\frac{1-|\xi'|_g}{h^\e}\right),\\
 R:=-(I+G_{\Delta}^{1/2}VG_{\Delta}^{1/2})^{-1}G_\Delta^{1/2}VG_{\Delta}^{1/2}\oph(\chi_\e),\end{gathered}\\
 T(z):=G_{\Delta}^{-1/2}(z)G_B(z)G_{\Delta}^{-1/2}(z)\label{eqn:transition}
 \end{gather}
  where $G_B$ is the Fourier integral operator component of  $G(z)$. (See section \ref{sec:semiclassicalPreliminaries} for an explanation of the quantization procedure $\oph$.) Note also that the inverse $(I+G_{\Delta}^{1/2}VG_{\Delta}^{1/2})^{-1}$ makes sense microlocally on $\supp \chi_\e$ by \eqref{a2}.

Let $\tilde{\sigma}$ denote the compressed shymbol (see \cite[Section 2.3]{Galk} or Section \ref{sec:shymbol}). Then let $l_N,\,r_N(z): B^*\partial\Omega\to \re$ (recall that $B^*\pO$ is the coball bundle of the boundary and $\beta$ is the billiard ball map) be
\begin{gather}
\label{eqn:defineLength}
l(q,q'):=|\pi_x(q)-\pi_x(q')|,\quad\quad\quad l_N(q):=\frac{1}{N}\sum_{k=1}^Nl(\beta^{k-1}(q),\beta^k(q))\\
\label{eqn:defineAverageReflection}
r_N(q):=\frac{\Im z}{h}l_N(q)+\frac{1}{2N}\log\tilde{\sigma}(((RT(z))^*)^N(RT(z))^N)(q).
\end{gather}
The term $\tfrac{\Im z}{h}l_N$ in \eqref{eqn:defineAverageReflection} serves to cancel the growth of $T(z)$ in the definition of $r_N$.
\begin{remark}
Note that we use the notion of the compressed shymbol instead of a variable order symbol since we do not wish to make any apriori assumption on how the symbol of $V$ varies from point to point. Moreover, the order of the symbol will vary also as a function of $\Im z.$
\end{remark}

 In fact,
for $0<N$ independent of $h$ we have
\begin{equation}
\label{def:averageReflection}
r_N(q)=\frac{1}{2N}\sum_{n=1}^N\log\left|\left(\tilde{\sigma}(R)\composed \beta^n(q) +\O{}(h^{I_{R}(q)+1-2\e})\right)\right|^2
\end{equation}
where $I_{R}(q)$ is the local order of $R$ at $q$ (see \cite[Section 2.3]{Galk} or Section \ref{sec:shymbol}]). The expression \eqref{def:averageReflection} illustrates that $r_N$ is the logarithmic average reflectivity over $N$ iterations of the billiard ball map. 

Let 
\begin{equation}
\label{e:Pintro}\mc{P}(z):=\begin{pmatrix}-h^2\Delta-z^2\\\partial_\nu +B\end{pmatrix}:\Hh^{s+2}(\Omega)\to \Hh^s(\Omega)\oplus H^{s+1/2-\max(m-1,0)}_h(\pO)
\end{equation}
where $\Hh^m$ denotes the semiclassical Sobolev space with norm 
\begin{equation} \label{e:semiSob} \|u\|_{\Hh^m}:=\|\la hD\ra ^mu\|_{L^2}.\end{equation}
(See \cite[Section 7.1]{EZB} for a more precise definition.)
Let $\Phi_-(s)$ and $\zeta_j$ be as in \eqref{e:Airy1} and $Q\in C^\infty(T^*\pO)$ be the symbol of the second fundamental form to $\pO$ (as in Theorem \ref{thm:deltaMain}), and define $f_j(\cdot;h)\in C^\infty(T^*\pO)$ for $j=1,2,\dots$ by
$$f_j(q;h):=\frac{Q(q)((2hQ(q))^{1/3}(1+a_1)\Im\Phi_-(\zeta_j)+\sigma(h\Im V_1)(q))}{|\sigma(hV)(q)|^2}.$$
Let $S^*\pO$ denote the cosphere bundle of $\pO$ and 
\begin{multline*}
\mc{B}_{j,\pm}(\e,C;h):=\\
\left\{z\in D_M(h)\,:\,\inf_{S^*\pO}\left(f_j(q;h)-Ch^{-\alpha}\right)(1\mp\e)\leq \frac{\Im z}{h} 
\leq \sup_{S^*\pO}\left(f_j(q;h)+Ch^{-\alpha}\right)(1\pm\e)\right\}.\end{multline*}
Then
Theorems \ref{thm:mainTransparent}, \ref{thm:deltaMain}, and \ref{thm:boundaryStable} are a consequence of the following:
\begin{theorem}
\label{thm:mainGeneral}
Let $\Omega\Subset\re^d$ be strictly convex with smooth boundary. Fix $\e>0$, $M>0$, $N_1, N_2>0$, $m\geq 0$ and suppose that $\mc{AV}(a_1,\alpha, E', m,\delta_0,M,M_1,M_2,\e_0)$ holds.  Then there exist $h_0>0, C,c,N>0$, so that if $0<h<h_0$, $z\in D_M(h)$,
\begin{equation}
\label{eqn:hyperbolicSabine} \frac{\Im z}{h}\leq\sup_{N\leq N_1}\inf_{|\xi'|_g\leq 1-h^{\e_0}}\frac{r_N}{2l_N}(1-\e)\quad \text{or}\quad\frac{\Im z}{h}\geq \inf_{N\leq N_1}\sup_{|\xi'|_g\leq 1-h^{\e_0}}\frac{r_N}{2l_N}(1+\e),
\end{equation}
$\pm \Im z\geq 0$, $z\notin \cup_{j=1}^{N_2}\mc{B}_{j,\pm}(\e, C;h)$, and 
\begin{equation*}\frac{\Im z}{h}\geq \sup_{S^*\pO}\left(\begin{gathered}f_{N_2+1}(q;h)+Ch^{-\alpha}\end{gathered}\right)(1\pm \e).
\end{equation*}

\noindent then $\mc{P}(z)$ is invertible and moreover if 
$$\mc{P}(z)u=\begin{pmatrix}0\\v\end{pmatrix},$$
then
\begin{equation} \label{eqn:estpsi1}\|u|_{\partial\Omega}\|_{\Hh^m}\leq ch^{-N}\|v\|_{L^2}.\end{equation}
\end{theorem}
Observe that Theorem \ref{thm:mainGeneral} (in particular, \eqref{eqn:hyperbolicSabine}) takes the same form as \eqref{eqn:heurRes}. Thus, the poles of $\mc{P}(z)^{-1}$ are controlled by the average reflectivity in the hyperbolic region. To see that this continues up to the glancing set and hence that Theorem \ref{thm:mainGeneral} is a quantum version of the Sabine law, observe that \eqref{eqn:heurRes} matches \eqref{eqn:hyperbolicSabine}. Moreover, using Lemma \ref{lem:dynStrictlyConvex}, that $V$ is elliptic near $|\xi'|_g=1$ and
$$\sigma(R)=\frac{\sigma(hV)}{2i\sqrt{1-|\xi'|^2}-hV},$$
we have that for $q=(x,\xi')\in B^*\pO$ with $\sqrt{1-|\xi'|^2_g}\ll h^{1+\alpha}$ 
\begin{equation}
\label{eqn:nearGlanceSabine}\frac{\log \left|R(\beta(q))\right|^2}{2l(q,\beta(q))}=\frac{-Q(x,\xi')(\sqrt{1-|\xi'|^2}-\Im hV)}{|\sigma(hV)|^2}+\O{}(h^{-\alpha-1}\sqrt{1-|\xi'|^2}),
\end{equation}
where, as above, $Q(x,\xi')$ is the symbol of the second fundamental form to $\pO$. 
Now,
$$\Im \Phi_-(s)\sim -\sqrt{-s},\quad\quad s\to -\infty\quad\quad(\text{see Figure \ref{fig:phim}}).$$
Therefore \eqref{eqn:nearGlanceSabine} matches the bounds in Theorem \ref{thm:mainGeneral} modulo:
\begin{enumerate}
\item modes cannot concentrate closer than $h^{2/3}$ to $\{|\xi'|_g=1\}$ (the glancing set)
\item a quantization involving the zeros of the Airy function happens at scale $h^{2/3}$ near glancing
\item replacing $-\sqrt{-s}$ by $\Im \Phi_-(s)$.
\end{enumerate}
\subsection{Outline of the Proof}
Proving Theorem \ref{thm:mainGeneral} amounts to understanding the location of resonances, which correspond to $z$ so that $\mc{P}(z)$ is not invertible. We proceed by proving the estimate \eqref{eqn:estpsi1} on solutions to \eqref{eqn:mainGeneral} which implies an estimate on $\mc{P}(z)^{-1}.$ 

 To avoid analyzing the microlocally complicated interior Dirichlet to Neumann map, we change the boundary condition. In particular, we have
\begin{equation}
\label{eqn:newBV}(I+GV)\psi=Gv.
\end{equation}
We then proceed similarly to \cite{Galk} and decompose the boundary microlocally into the hyperbolic, glancing, and elliptic regions given respectively by
\begin{align*}  
\mc{H}&=\{(x,\xi')\in T^*\pO\,|\,|\xi'|_g\leq 1-h^\e\},\\
\mc{G}&=\{(x,\xi')\in T^*\pO\,|\, ||\xi'|_g-1|\leq h^\e\},\\
\mc{E}&=\{(x,\xi')\in T^*\pO\,|\,|\xi'|_g\geq 1+h^\e\}.
\end{align*}
Then, letting $\mathbb{1}_U$ be an operator microlocally equal to the identity on $U$ and $U'$ be a slight enlargement of $U$, we have 
$$(I-\mathbb{1}_{U'})G\mathbb{1}_{U}=\O{\Ph{-\infty}{}}(h^\infty)$$
where $U$ is any of $\mc{H},$ $\mc{G}$, or $\mc{E}.$ This allows us to work with each region separately. 

For notational convenience, let $\psi=u|_{\partial\Omega}$ and recall that 
\begin{equation}\label{e:uSolve}\mc{P}(z)u=\begin{pmatrix}0\\v\end{pmatrix}
\end{equation}
where $\mc{P}(z)$ is as in \eqref{e:Pintro}.
 We first consider $\mc{E}$. Here, $G$ is a pseudodifferential operator and our assumptions on $V$ allow us to prove estimates on $\mathbb{1}_\mc{E}\psi$ in terms of $v$. We then consider the hyperbolic region, $\mc{H}$. Here the situation is more complicated because $G$ consists of two pieces: $G_B$, a Fourier integral operator (FIO) associated to the billiard ball map, and $G_\Delta$, a pseudodifferential operator. Using the calculus of FIO's, we are able to reduce estimating solutions to \eqref{eqn:newBV} microlocally in $\mc{H}$ to estimating solutions to 
$$(I-(RT)^N)u=Av$$
for some $A$. Then, again using the calculus of FIOs, we see that $I-(RT)^N$ is microlocally invertible under the conditions given in \eqref{eqn:hyperbolicSabine}. 

Up to this point, the analysis in the present paper requires only minor changes from that in \cite{Galk}. However, the analysis near glancing is substantially different and heavily uses the microlocal model for $G$ and $\Sl:=1_{\Omega}R_0(z/h)\gamma^*$ near glancing given in \cite[Section 4.5]{Galk}. The analysis in \cite[Chapter 5]{Galk} uses only the microlocal model for $G$ and does so simply to obtain a norm bound on $G$ near glancing. Here we use the precise microlocal properties of $G$ and $\Sl$ near glancing. 

 We start by analyzing $I+GV$ as a second microlocal paseudodifferential operator on 
$$\mc{G}_+:=\{(x,\xi')\in T^*\pO\,|\,1-Mh^{2/3}\leq |\xi'|_g\leq 1+h^\e\}$$
which is the microlocal region closest to glancing.
When $\alpha$ is sufficiently small, ($\alpha <-2/3$) we see that $I+GV$ is elliptic on $\mc{G}_+$ outside of a union of $h^{2/3}$ thickened hypersurfaces given by
\begin{gather*} 
\mc{G}_N:=\bigcup_{j=1}^N\mc{G}_j,\quad
\mc{G}_j:=\left\{\left.(x,\xi')\in T^*\pO\,\right|\, \left|\frac{|\xi'|_g^2-1}{(2Q(x,\xi'))^{2/3}}-h^{2/3}\zeta_j\right|\leq \delta h^{2/3}\right\}.
\end{gather*}
Since we have microlocal invertibilty on $\mc{G}_+$ off of $\mc{G}_N$, resonance states must concentrate on $\mc{G}_N$. This is the quantization condition which occurs at scale $h^{2/3}$. 

To get this quantization property, we have used the microlocal structure of $G$. To obtain estimates the remaining part of $\psi$, i,e, on $\psi_g:=(\mathbb{1}_{\mc{G}_N}+\mathbb{1}_{\mc{G}_-})\psi$ where 
$$\mc{G}_-:= \{(x,\xi')\in T^*\pO\,|\, |\xi'|_g\leq 1-Mh^{2/3}\}$$
we will use the microlocal structure of $\Sl$. 

We have that $u$ solves \eqref{e:uSolve}. Integrating by parts in $\Omega$, we have
\begin{equation}
\label{eqn:greensIntro} 
\left(\frac{2\Re z\Im z}{h}\|u\|_{L^2}^2-\Im \la B\psi,\psi\ra\right)= -\Im\la h v,\psi\ra .
\end{equation}
Then, letting $\D$ denote the double layer potential and using a classical boundary layer formula together with the boundary condition from \eqref{eqn:mainGeneral}, we have
\begin{equation*}
u=h^{-1}\S h\partial_\nu u-\D u=-(h^{-1}\S B +\D )\psi+\S v=-\Sl V\psi+\Sl v.
\end{equation*}
So, we can write $u$ in terms of $\psi$ via the boundary layer potential, $\Sl$. Another technical innovation in our proof is to use the model for $\Sl$ near glancing to identify $\Sl^*\Sl$ as a second microlocal pseudodifferential operator on $\mc{G}$. We are then able to apply the sharp G$\mathring{\text{a}}$rding inequality to obtain upper and lower bounds on 
$$\left(\frac{2\Re z\Im z}{h}\|u_g\|_{L^2}^2-\Im \la B\psi_g,\psi_g\ra\right)$$
where $u_g=-h^{-1}\Sl V\psi_g$. Together with \eqref{eqn:greensIntro}, this allows us to estimate $\psi_g$ in terms of $v$.

Combining the estimates on $\mc{E}$, $\mc{G}$, and $\mc{H}$, we are able to estimate $\psi$ in terms of $v$.  In order to prove that condition (3) of Theorem \ref{thm:mainGeneral} together with \eqref{eqn:hyperbolicSabine} implies \eqref{eqn:estpsi1}, we refine our estimates on $\mc{G}$ when $|\Im z|\geq ch^N$ for some $N>0$. 

Because we have polynomial bounds on the interior Dirichlet to Neumann map, $N_1(z/h)$, in this region and 
$$(N_1+N_2)G=I=G(N_1+N_2),$$ 
we are able to show that if 
$$(I+GV)\tilde{\psi}=w,$$
then there exists $v=(N_1+N_2)w$ such that $(I+GV)\tilde{\psi}=Gv$ and hence there exists $\tilde{u}$ solving \eqref{eqn:mainGeneral} with $v$ replaced by $(N_1+N_2)w=\O{L^2\to L^2}(h^{-N})w$. 

Returning to the original problem, $(I+GV)\psi =Gv$, we see that for $\delta$ small enough,  $\mc{G}_j$ are separated by $\delta h^{2/3}$. Hence, we can find $\psi_j$ microlocalized $\delta h^{2/3}$ close to $\mc{G}_j$ so that 
$$(I+GV)\psi_j=w_j,\quad\quad \|w_j\|\leq Ch^{-M}\|v\|.$$
Therefore, we can find $u_j$ solving \eqref{eqn:mainGeneral} with $u_j|_{\pO}=\psi_j$ and $v=v_j=h(N_1+N_2)w_j$ and, repeating the analysis above using boundary layer operators, we can obtain estimates on $\psi_j$. Together with knowledge of the symbol of $N_2$ and that of $\Sl^*\Sl$, this finishes the proof of Theorem \ref{thm:mainGeneral}.

\subsection{Organization of the paper}
The paper is organized as follows. We start by introducing the necessary standard semiclassical tools as well as the shymbol from \cite{Galk} in Sections \ref{sec:semiclassicalPreliminaries} and \ref{sec:shymbol}. Then in Section \ref{sec:secondMicrolocal}, we introduce the second microlocal calculus from \cite{SjoZwDist}. We conclude the preliminary material with Section \ref{sec:billiard} where we introduce the billiard ball flow and map. 

As a guide for the general case, Section \ref{sec:Friedlander} analyzes the single and double layer potentials, respectively 
\begin{equation}
\label{e:layerDef}
\begin{aligned} 
\Sl(\lambda)f(x)&:=\int_{\pO}R_0(\lambda)(x,y)f(y)dS(y)&x\in \Omega\\
\D(\lambda)f(x)&:=\int_{\pO}\partial_{\nu_y}R_0(\lambda)(x,y)f(y)dS(y)&x\in \Omega
\end{aligned} 
\end{equation}
and operators, respectively
\begin{align*}
G(\lambda)f(x)&:=\int_{\pO}R_0(\lambda)(x,y)f(y)dS(y)&x\in \pO\\
\Dl(\lambda)f(x)&:=\int_{\pO}\partial_{\nu_y}R_0(\lambda)(x,y)f(y)dS(y)&x\in \pO
\end{align*}
in the special case of the Friedlander model. Section \ref{sec:Layer} contains the analysis of the boundary layer potentials and operators in the general strictly convex case. Next, Section \ref{sec:genBStable} gives the proof of Theorem \ref{thm:mainGeneral} including the Fredholm property and meromorphy of the resolvent for $\mc{P}$. Sections \ref{sec:appTrans}, \ref{sec:appDelta}, and \ref{sec:appStable} respectively contain the necessary material to deduce Theorems \ref{thm:mainTransparent}, \ref{thm:deltaMain}, and \ref{thm:boundaryStable} from Theorem \ref{thm:mainGeneral}. Finally, Section \ref{sec:optimalTrans} gives the proof that Theorem \ref{thm:mainTransparent} is sharp in the case of the unit disk.

\noindent
{\sc Acknowledgemnts.} The author would like to thank Maciej Zworski for  encouragement and many valuable suggestions. Thanks also to St\'{e}phane Nonnenmacher, John Toth, Dimitry Jakobson, Richard Melrose, Jeremy Marzuola, and Semyon Dyatlov for stimulating conversations about the project. Thanks also to the anonymous referee for many helpful comments. The author is grateful to the Erwin Schr\"{o}dinger Institue for support during the program on the Modern Theory of Wave Equations, to the National Science Foundation for support under the Mathematical Sciences Postdoctoral Research Fellowship  DMS-1502661 and to the Centre de Recherches Mathematique for support under the CRM postdoctoral fellowship.



\section{Semiclassical preliminaries}

\label{sec:semiclassicalPreliminaries}
In this section, we review the methods of semiclassical analysis which are needed throughout the rest of our work. The theories of pseudodifferential operators, wavefront sets, and the local theory of Fourier integral operators are standard and our treatment follows that in \cite{DyGui} and \cite{EZB}.  We introduce the notion \emph{shymbol} from \cite{Galk} which is a notion of sheaf-valued symbol that is sensitive to local changes in the semiclassical order of a symbol. 

\subsection{Notation}
We review the relevant notation from semiclassical analysis in this section. For more details, see \cite{d-s} or \cite{EZB}.
\subsubsection{Big $\O{}$ notation}
The $\O{}(\cdot)$ and $\o{}(\cdot)$ notations are used in this paper in the following ways:
we write $u=\O{\mc X}(F)$ if the norm of $u$ in the functional space $\mc X$ is bounded
by the expression $F$ times a constant. We write $u=\o{\mc{X}}(F)$ if the norm of $u$  has
$$\lim_{s\to s_0}\frac{\|u(s)\|_{\mc{X}}}{F(s)}=0$$
where $s$ is the relevant parameter. If no space $\mc{X}$ is specified, then $u=\O{}(F)$ and $u=\o{}(F)$ mean 
\begin{equation}
\label{e:oDef}|u(s)|\leq C|F(s)|\quad \text{and} \quad \lim_{s\to s_0}\frac{|u(s)|}{F(s)}=0
\end{equation}
respectively. 
\subsubsection{Phase space}
Let $M$ be a $d$-dimensional manifold without boundary. Then we denote an element of the cotangent bundle to $M$, $(x,\xi)$ where $\xi\in T_x^*M$.
\subsection{Symbols and Quantization}
We start by defining the exotic symbol class $f(h)S^m_\delta(M)$.
\begin{defin}
\label{def:symbol}
Let $a(x,\xi; h)\in C^\infty (T^*M\times [0,h_0) )$, $f\in C^\infty((0,h_0))$, $m\in \re$, and $\delta\in[0,1/2)$. Then, we say that $a\in f(h)S^m_{\delta}(T^*M)$ if for every $K\Subset M$ and ${\varsigma},\,{\varpi}$ multiindeces, there exists $C_{{\varsigma}{\varpi} K}$ such that
\begin{equation} |\partial_x^{\varsigma}\partial_\xi^{\varpi} a(x,\xi; h)|\leq C_{{\varsigma}{\varpi} K}f(h)h^{-\delta(|{\varsigma}|+|{\varpi}|)}\la \xi\ra^{m-|{\varpi}|}\}\label{eqn:defSymbol}\end{equation}
We denote $S_\delta^\infty:=\cup_m S_\delta^m$, $S_\delta^{-\infty}:=\cap_mS_\delta^m$ and when one of the parameters $\delta$ or $m$ is 0, we suppress it in the notation. 

We say that $a(x,\xi;h)\in S_\delta^{\comp}(M)$  if $a\in S_{\delta}(M)$ and $a$ is supported in some $h$-independent compact set.
\end{defin}
This definition of a symbol is invariant under changes of variables (see for example \cite[Theorem 9.4]{EZB} or more precisely, the arguments therein). 

\subsection{Pseudodifferential operators}
We follow \cite[Section 14.2]{EZB} to define the algebra $\Ph{m}{\delta}(M)$ of pseudodifferential operators with symbols in $S^m_\delta(M)$. (For the details of the construction of these operators, see for example \cite[Sections 4.4, 14.12]{EZB}. See also \cite[Chapter 18]{HOV3} or \cite[Chapter 3]{gr-s}.) Since we have made no assumption on the behavior of our symbols as $x\to \infty$, we do not have control over the behavior of $\Ph{k}{\delta}(M)$ near infinity in $M$. However, we do require that all operators $A\in \Ph{m}{\delta}(M)$ are \emph{properly supported}. That is, the restriction of each projection map $\pi_x,\pi_{x'}:M\times M\to M$ to the support of $K_A(x,x';h)$, the Schwartz kernel of $A$, is a proper map. For the construction of such a quantization procedure, see for example \cite[Proposition 18.1.22]{HOV3}. An element in $A\in \Ph{m}{\delta}(M)$ acts $H^s_{h,\loc}(M)\to H^{s-m}_{h,\loc}(M)$ where $H^s_{h,\loc}(M)$ denotes the space of distributions locally in the semiclassical Sobolev space $\Hh^s(M)$. The definition of these spaces can be found for example in \cite[Section 7.1]{EZB}. Finally, we say that a properly supported operator, $A$, with
$$A:\mc{D}'(M)\to C^\infty(M)$$ 
and each seminorm $\O{}(h^\infty)$ is $\O{\Ph{-\infty}{}}(h^\infty)$. We include operators that are $\O{\Ph{-\infty}{}}(h^\infty)$  in all pseuodifferential classes. 

With this definition, we have the semiclassical principal symbol map
\begin{equation}\label{eqn:princSymbol}\sigma:\Psi^m_{\delta}(M)\to \quotient{S^m_{\delta}(M)}{h^{1-2\delta}S^{m-1}_\delta(M)}\end{equation}
and a non-canonical quantization map
$$\oph:S^m_\delta(M)\to \Ph{m}{\delta}(M)$$
with the property that $\sigma\composed \oph$ is the natural projection map onto  
$$\quotient{S^m_{\delta}(M)}{h^{1-2\delta}S^{m-1}_\delta(M)}.$$

Henceforward, we will take $\sigma(A)$ to be any representative of the corresponding equivalence class in the right-hand side of \eqref{eqn:princSymbol}. We do not include the sub-principal symbol because then the calculus of pseudodifferential operators would be more complicated. With this in mind, the standard calculus of pseudodifferential operators with symbols in $S^m_\delta$ gives for $A\in \Ph{m_1}{\delta}(M)$ and $B\in \Ph{m_2}{\delta}(M)$, 
\begin{gather*}
\sigma(A^*)=\overline{\sigma(A)}+\O{S^{m_1-1}_\delta(M)}(h^{1-2\delta})\\
\sigma(AB)=\sigma(A)\sigma(B)+\O{S^{m_1+m_2-1}_{\delta}(M)}(h^{1-2\delta})\\
\sigma([A,B])=-ih\{\sigma(A),\sigma(B)\}+\O{S^{m_1+m_2-2}(M)}(h^{2(1-2\delta)}).
\end{gather*}
Here $\{\cdot,\cdot\}$ denotes the Poisson bracket and we take adjoints with respect to $L^2(M)$. 

\subsubsection{Wavefront sets and microsupport of pseudodifferential operators}
In order to define a notion of wavefront set that captures both $h$-microlocal and $C^\infty$ behavior, we define the \emph{fiber radially compactified cotangent bundle}, $\overline{T}^*M$, by $\overline T
^*M=T^*M\sqcup S^*M$ where 
$$S^*M:=\quotient{\left(T^*M\setminus\{M\times 0\}\right)}{\re_+}$$
and the $\re_+$ action is given by $(t,(x,\xi))\mapsto(x,t\xi).$ Let $|\cdot|_g$ denote the norm induced on $T^*M$ by the Riemannian metric $g$. Then a neighborhood of a point $(x_0,\xi_0)\in S^*M$ is given by $V\cap \{|\xi|_g\geq K\}$ where $V$ is an open conic neighborhood of $(x_0,\xi_0)\in T^*M$. 

For each $A\in \Ph{m}{\delta}(M)$ there exists $a\in S^m_\delta(M)$ with $A=\oph(a)+\O{\Ph{-\infty}{}}(h^\infty).$ Then the \emph{semiclassical wavefront set of $A$}, $\WFhp(A)\subset \overline T^*M$, is defined as follows. A point $(x,\xi)\in \overline T^*M$ does not lie in $\WFhp(A)$ if there exists a neighborhood $U$ of $(x,\xi)$ such that each $(x,\xi)$ derivative of $a$ is $\O{}(h^\infty \la \xi\ra^{-\infty})$ in $U$. As in \cite{Alexandrova}, we write 
$$\WFhp(A)=:\WFhpf(A)\sqcup\WFhpi(A)$$
where $\WFhpf(A)=\WFh(A)\cap T^*M$ and $\WFhpi(A)=\WFh(A)\cap S^*M.$

Operators with compact wavefront sets in $T^*M$ are called \emph{compactly microlocalized}. These are operators of the form 
$$\oph(a)+\O{\Ph{-\infty}{}}(h^\infty)$$ for some $a\in S^{\comp}_\delta(M).$ The class of all compactly microlocalized operators in $\Ph{m}{\delta}(M)$ are denoted by $\Ph{\comp}{\delta}(M)$. 

We will also need a finer notion of microsupport on $h$-dependent sets. 

\begin{defin}\label{d-microlocal-vanishing}
An operator $A\in \Ph{\comp}{\delta}(M)$ is \emph{microsupported} on an $h$-dependent family of sets $V(h)\subset T^*M$ if we can write $A=\oph(a)+\O{\Ph{-\infty}{}}(h^\infty)$, where for each compact set $K\subset T^*M$, each differential operator $\partial^{\varsigma}$ on $T^*M$, and each $N$, there exists a constant $C_{{\varsigma} N K}$ such that for $h$ small enough,
$$\sup_{(x,\xi)\in K\setminus V(h)}|\partial^{\varsigma} a(x,\xi;h)|\leq C_{{\varsigma} NK}h^N.$$
We then write $$\MSp(A)\subset V(h).$$
\end{defin}

%
The change of variables formula for the full symbol of a
pseudodifferential operator~\cite[Theorem~9.10]{EZB} contains an asymptotic expansion in powers
of $h$ consisting of derivatives of the original symbol. Thus
definition~\ref{d-microlocal-vanishing} does not depend on the choice
of the quantization procedure $\oph$.  Moreover, since we take $\delta<1/2$, if
$A\in\Ph{\comp}{\delta}$ is microsupported inside some $V(h)$ and
$B\in\Ph{m}{\delta}$, then $AB$, $BA$, and $A^*$ are also microsupported
inside $V(h)$. This implies the following.
\begin{lemma}
Suppose that $A,B\in \Ph{\comp}{\delta}$ and $\MSp(A)\cap\MSp(B)=\emptyset.$ Then 
$$\WFhp(AB)=\emptyset.$$
\end{lemma}

For $A\in \Ph{\comp}{\delta}(M)$, $(x,\xi)\notin \WFh(A)$ if and only if there exists an $h$-independent neighborhood $u$ of $(x,\xi)$ such that $A$ is microsupported on the complement of $U$. However, $A$ need only be microsupported on any $h$-independent neighborhood of $\WFhp(A)$, not on $\WFhp(A)$ itself. Also, notice that by Taylor's formula if $A\in \Ph{\comp}{\delta}(M)$ is microsupported in $V(h)$ and $\delta'>\delta$, then $A$ is also microsupported on the set of all points in $V(h)$ which are at least $h^{\delta'}$ away from the complement of $V(h)$.

\begin{remark} Notice that since we are working with $A\in \Ph{\comp}{\delta}(M)$ for $0\leq\delta<1/2$ we have $a\in S_\delta^{\comp}(T^*M)$ and $a$ can only vary on a scale $\sim h^{-\delta}$. This implies that the set $\MSp(A)$ will respect the uncertainty principle.
\end{remark}

\subsubsection{Ellipticity and $L^2$ operator norm}
\label{s:ell1}
For $A\in\Ph{m}{\delta}(M)$, define its \emph{elliptic set}
$\Ell(A)\subset T^*M$ as follows: $(x,\xi)\in\Ell(A)$ if and
only if there exists a neighborhood $U$ of $(x,\xi)$ in $\overline
T^*M$ and a constant $C$ such that $|\sigma(A)|\geq
C^{-1}\langle\xi\rangle^m$ in $U\cap T^*M$.  The following statement
is the standard semiclassical elliptic estimate;
see~\cite[Theorem~18.1.24']{HOV3} for the closely related microlocal
case and for example~\cite[Section~2.2]{zeeman} for the semiclassical
case.
%
%
\begin{lemma}
\label{lem:microlocalElliptic}
Suppose that $P\in\Ph{m}{\delta}(M)$ and $A\in\Ph{m'}{\delta}(M)$ with $\WFhp(A)\subset\Ell(P)$. Then for each $\chi\in\Cc(M)$,  there exist $Q_i\in \Ph{m'-m}{\delta}(M)$ such that 
$$\chi A=\chi Q_1P+\O{\Ph{-\infty}{\delta}}(h^\infty)=\chi PQ_2+\O{\Ph{-\infty}{}}(h^\infty).$$
In particular, for each $s\in \re$ and $u\in \Hh^{s+m'}$ there exists $C>0$ such that for all $N>0$, and $\chi_1\in C^\infty(M)$ with $\chi_1 \equiv 1 $ on $\supp \chi$, 
$$\|\chi Au\|_{\Hh^s}\leq C\|\chi Pu\|_{\Hh^{s+m'-m}}+\O{}(h^\infty)\|\chi_1 u\|_{\Hh^{-N}}.$$
\end{lemma}

We also recall the estimate for the $L^2\to L^2$ norm of a pseudodifferential operator (see for example \cite[Chapter 13]{EZB}).
\begin{lemma}
\label{lem:goodL2Bound}
Suppose that $A\in \Ph{}{\delta}(M)$. Then there exists $C>0$ such that 
$$\|A\|_{L^2\to L^2}\leq \sup_{T^*M} |\sigma(A)|+Ch^{1-2\delta}.$$
\end{lemma}

\subsection{Semiclassical microlocalization of distributions and operators}
\subsubsection{Semiclassical wavefront sets and microsupport for distributions}
An $h$-dependent
family $u(h):(0,h_0)\to \mc D'(M)$ is called \emph{h-tempered} if
for each open $U\Subset M$, there exist constants
$C$ and $N$ such that 
\begin{equation}
  \label{tempered}
\|u(h)\|_{H^{-N}_{h}(U)}\leq Ch^{-N}.
\end{equation}
For a tempered distribution $u$, we say that $(x_0,\xi_0)\in \overline T^*M$ does not lie in the wavefront set $\WFh(u)$, if there exists a
neighborhood $V$ of $(x_0,\xi_0)$ such that for each
$A\in\Ph{}{}(M)$ with $\WFhp(A)\subset V$, we have $Au=\O{C^\infty}(h^\infty)$. As above, we write 
$$\WFh(u)={\WFh}^f(u)\sqcup{\WFh}^i(u)$$
where ${\WFh}^i(u)=\WFh(u)\cap S^*M$.
By Lemma~\ref{lem:microlocalElliptic},
$(x_0,\xi_0)\not\in\WFh(u)$ if and only if there exists compactly
supported $A\in\Ph{}{}(M)$ elliptic at $(x_0,\xi_0)$ such that
$Au=\O{C^\infty}(h^\infty)$.  The wavefront set of $u$ is a
closed subset of $\overline T^*M$. It is empty if and only if
$u=\O{C^\infty(M)}(h^\infty)$. We can also verify that for
$u$ tempered and $A\in\Ph{m}{\delta}(M)$,
$\WFh(Au)\subset\WFhp(A)\cap\WFh(u)$.

\begin{defin}
\label{d:wavefront}
A tempered distribution $u$ is said to be \emph{microsupported} on an $h-$dependent family of sets $V(h)\subset T^*M$ if for $\delta\in[0,1/2)$, $A\in \Ph{}{\delta}(M)$, and $\MSp(A)\cap V=\emptyset$, $\WFh(Au)=\emptyset.$
\end{defin}

\subsubsection{Semiclassical wavefront sets of tempered operators}
An $h$- dependent family of operators $A(h):\mc{S}(M)\to \mc{S}'(M')$ is called \emph{h-tempered} if for each $U\Subset M$, there exists $N\geq 0$ and $k\in \ints^+$, such that
\begin{equation}
\label{eqn:temperedOp}
\|A(h)\|_{\Hh^k(U)\to H_{h,\loc}^{-k}(M')}\leq Ch^{-N}
\end{equation}

For an $h$-tempered family of operators, we write that the wavefront set of $A$ is given by
$${\WFh}'(A):=\{(x,\xi,y,\eta)\,|,(x,\xi,y,-\eta)\in \WFh(K_A)\}$$
where $K_A$ is the Schwartz kernel of $A$. 
\begin{defin}
\label{d:microlocal-vanishingOps}
A tempered operator $A$ is said to be
\emph{microsupported} on an $h$-dependent family of sets
$V(h)\subset T^*M\times T^*M'$, if for all $\delta\in [0,1/2)$ and each $B_1\in \Ph{}{\delta}(M')$ and $B_2\in \Ph{}{\delta}(M)$ with $(\MSp(B_1)\times\MSp(B_2))\cap V=\emptyset$, we have $\WFh(B_1AB_2)=\emptyset.$ We then write 
$${\MS}'(A)\subset V(h).$$
\end{defin}

\begin{remark}
With the definitions above, we have for $A\in \Ph{m}{\delta}(M)$, $${\WFh}'(A)=\{(x,\xi,x,\xi)\,:\, (x,\xi)\in \WFhp(A)\}.$$ 
In addition, we have that if $A\in \Ph{\comp}{\delta}$, then $\MSp(A)\subset V(h)$ if and only if 
$${\MS}'(A)\subset \{(x,\xi,x,\xi)\,:\, (x,\xi)\in V(h)\}.$$
Since there is a simple relationship between $\WFhp$ and $\WFh$, as well as $\MSp$ and $\MS$, we will only use the notation without $\Psi$ from this point forward and the correct object will be understood from context.
\end{remark}

\subsection{Semiclassical Lagrangian distributions}
  \label{s:prelim.lagrangian}

In this subsection, we review some facts from the theory of
semiclassical Lagrangian distributions.  See~\cite[Chapter~6]{g-s}
or~\cite[Section~2.3]{svn} for a detailed account,
and~\cite[Section~25.1]{HOV4} or~\cite[Chapter~11]{gr-s} for the microlocal case. We do not attempt to define the principal symbol as a globally invariant object. Indeed, it is not always possible to do so in the semiclassical setting. When it is possible to do so, i.e. when the Lagrangian is exact, we define the symbol modulo the Maslov bundle. Taking symbols modulo the Maslov bundle makes the theory considerably simpler. We can make this simplification since for all of our symbolic computations, we work only in a single coordinate chart and, moreover, we always work with exact Lagrangians.

\subsubsection{Phase functions}
Let $M$ be a manifold without boundary.  We denote its dimension
by $d$.  Let $\varphi(x,\theta)$ be a smooth real-valued function on
some open subset $U_\varphi$ of $M\times \mathbb R^L$, for some $L$;
we call $x$ the \emph{base variable} and $\theta$ the \emph{oscillatory
variable}. As in \cite[Section 21.2]{HOV3}, we say that $\varphi$ is a \emph{phase
function} if the differentials
$(\partial_{\theta_1}\varphi),\dots,d(\partial_{\theta_L}\varphi)$ on the \emph{critical set}
\begin{equation}
  \label{e:c-varphi}
C_\varphi:=\{(x,\theta)\mid \partial_\theta \varphi=0\}\subset U_\varphi
\end{equation}
are independent
Note that
\[
\Lambda_\varphi:=\{(x,\partial_x\varphi(x,\theta))\mid (x,\theta)\in C_\varphi\}\subset T^*M
\]
is an immersed Lagrangian submanifold (we will shrink the domain of $\varphi$ to make it
embedded). 

\subsubsection{Symbols}
Let $\delta\in [0,1/2)$. A smooth function $a(x,\theta;h)$ is called a
compactly supported symbol of type $\delta$ on $U_\varphi$, if it is supported in some compact $h$-independent subset of $U_\varphi$, and for each differential operator $\partial^{\varsigma}$ on $M\times \mathbb R^L$, there exists a constant $C_{\varsigma}$ such that
$$
\sup_{U_\varphi}|\partial^{\varsigma} a|\leq C_{\varsigma} h^{-\delta|{\varsigma}|}.
$$
As above, we write $a\in
S^{\comp}_\delta(U_\varphi)$ and denote $S^{\comp}:=S^{\comp}_0$. 

\subsubsection{Lagrangian distributions}
Given a phase function $\varphi$ and a symbol $a\in S_\delta^{\comp}(U_{\varphi})$, consider the $h$-dependent family of
functions
\begin{equation}
  \label{e:lagrangian-basic}
u(x;h)=(2\pi h)^{-(d+2L)/4}\int_{\mathbb R^L} e^{i\varphi(x,\theta)/h}a(x,\theta;h)\,d\theta.
\end{equation}
We call $u$ a \emph{Lagrangian distribution} of type $\delta$ generated by $\varphi$ and denote this by $u\in I^{\comp}_{\delta}(\Lambda_{\varphi})$.

By the method of non-stationary phase, if
$\supp a$ is contained in some $h$-dependent compact
set $K(h)\subset U_\varphi$, then
\begin{equation}
  \label{e:lag-wf}
\MS(u)\subset\{(x,\partial_x\varphi(x,\theta))\mid (x,\theta)\in C_\varphi\cap K(h)\}\subset\Lambda_\varphi.
\end{equation}
\begin{remark} We are using the fact that $a\in S_\delta(U_\varphi)$ for some $\delta<1/2$ here.
\end{remark}

\subsubsection{Principal Symbols}
We define the principal symbol of a Lagrangian distribution independently of the choice of $\varphi$. To do this, we will need to use half-densities on $\Lambda_{\varphi}$ (see, for example \cite[Chapter 9]{EZB} for a definition).

Following \cite[Section 25.1]{HOV4}, letting 
$$\Phi=\left(\begin{array}{cc}\varphi_{xx}''&\varphi_{x\theta}''\\ 
\varphi_{\theta x}''&\varphi_{\theta\theta}''\end{array}\right),$$ 
\begin{lemma}
\label{l:lagrangian-basic}
Modulo Maslov factors, and a factor $e^{iA/h}$ for some constant $A\in\re$ depending on $\varphi$, the \emph{principal symbol} 
$$\sigma(u)\in \quotient{S_{\delta}^{\comp}(\Lambda_\varphi;\Omega^{1/2})}{h^{1-2\delta}S_{\delta}^{\comp}(\Lambda_{\varphi};\Omega^{1/2})}$$ is a half density given by 
$$\sigma(u)(x,\xi)=|d\xi|^{1/2}a(x,\theta) e^{i\pi/4\sgn\Phi}|\det \Phi|^{-1/2}.$$
\end{lemma}

\begin{remark}
In the case that $\Lambda_{\varphi}$ is exact the factor $e^{iA/h}$ can be removed.
\end{remark}

\begin{defin}
  \label{d:lagrangian}
Let $\Lambda\subset T^*M$ be an embedded Lagrangian submanifold.
We say that an $h$-dependent family of functions
$u(x;h)\in \Cc(M)$ is a (compactly supported
and compactly microlocalized) 
Lagrangian distribution of type $\delta$ associated to $\Lambda$, if
it can be written as a sum of finitely many functions
of the form~\eqref{e:lagrangian-basic}, for different phase functions
$\varphi$ parametrizing open subsets of $\Lambda$, plus an
$\O{\Cc}(h^\infty)$ remainder. 
Denote by $I^{\comp}_\delta(\Lambda)$ the space of all such distributions,
and put $I^{\comp}(\Lambda):=I^{\comp}_0(\Lambda)$.
\end{defin}

The action of a pseudodifferential operator on a Lagrangian
distribution is given by the following Lemma, following from
 the method of stationary
phase:
%
%
\begin{lemma}\label{l:lagrangian-mul}
Let $u\in I^{\comp}_\delta(\Lambda)$ and $P\in \Ph{m}{\delta}(M)$.
Then $Pu\in I^{\comp}_\delta(\Lambda)$ and
$$\sigma(Pu)=\sigma(P)|_{\Lambda}\cdot\sigma(u)
+\O{}(h^{1-2\delta})_{S^{\comp}_\delta(\Lambda)}.
$$
\end{lemma}

\subsection{Fourier integral operators}
\label{sec:semiFIO}

A special case of Lagrangian distributions are Fourier integral
operators associated to canonical graphs.  Let $M$ be a
manifolds of dimension $d$. Consider a Lagrangian submanifold $\Lambda\subset T^*M\times T^*M$ given by 
$$\Lambda=\{(\kappa(y,\eta), y, -\eta)\}$$
where $\kappa$ is a symplectomorphism.

A compactly supported operator $U:\mc D'(M')\to \Cc(M)$ is
called a (semiclassical) \emph{Fourier integral operator} of type
$\delta$ associated to $\kappa$ 
if its Schwartz kernel $K_U(x,x')$ lies
in $I^{\comp}_\delta(\Lambda)$. We write $U\in
I^{\comp}_\delta(C)$ where 
$$C=\{(x,\xi,y,\eta)\,|\, (x,\xi,y,-\eta)\in \Lambda\}.$$ 
 The numerology $h^{-(d+2L)/4}$ in \eqref{e:lagrangian-basic} is explained by the fact that
the normalization for Fourier integral operators is chosen so that
$$\|U\|_{L^2(M)\to L^2(M)}\sim 1$$ when $C$ is the generated by a symplectomorphism.

We will need the following lemma from the calculus of Fourier integral operators 
\begin{lemma}
\label{lem:FIOcomp}
Let $A\in I_{\delta}^{\comp}(M\times M, C)$ and $P\in \Ph{\comp}{\delta}(M)$. Then, $A^*PA\in \Ph{\comp}{\delta}(M)$ and 
$$\sigma(A^*PA)(q)=|\sigma(A)(q,\kappa(q))|^2\sigma(P)(\kappa(q)).$$
\end{lemma}

\section{The shymbol}
\label{sec:shymbol}
It will be useful to calculate symbols of operators whose semiclassical order may vary from point to point in $T^*M$.
One can often handle this type of behavior by using weights to compensate for the growth. However, this requires some a priori knowledge of how the order changes and limits the allowable size in the change of order. In this section, we will develop a notion of a sheaf valued symbol, the \emph{shymbol}, that can be used to work in this setting without such a priori knowledge. 

Let $M$ be a compact manifold. Let $\mc{T}(T^*M)$ be the topology on $T^*M$. For $s\in \re$, denote the symbol map 
$$ \sigma_s:h^{s}\Psi^{\comp}_\delta\to h^sS^{\comp}_\delta/h^{s+1-2\delta}S^{\comp}_\delta.\,\,$$
Suppose that for some $N>0$ and $\delta \in [0,1/2)$, $A\in h^{-N}\Psi^{\comp}_\delta(M)$. We define a finer notion of symbol for such a pseudodifferential operator. Fix $0<\e\ll 1-2\delta$. For each open set $U\in \mc{T}(T^*M)$, define \emph{the $\e$-order of $A$ on $U$} 
$$ I_A^\e(U):=\sup_{s\in \mc{S}_\e}s+1-2\delta\,\,$$
where
$$\mc{S}_\e:=\left\{s\in \e\mathbb{Z} \,\left|\,\begin{gathered}\text{ there exists }\chi\in \Cc(T^*M),\,\chi|_U=1,\\\sigma_s(\oph(\chi) A\oph(\chi))|_U\equiv 0\end{gathered}\right.\right\}.$$
Then it is clear that for any $V\Subset U$ there exists $\chi\in \Cc(U)$ with $\chi=1$ on $V$ such that $\oph(\chi) A\oph(\chi)\in h^{I^\e_A(U)}\Ph{\comp}{\delta}(M).$

Give $\mc{T}(T^*M)$ the ordering that $U\leq V$ if $V\subset U$ with morphisms $U\to V$ if $U\leq V$. Notice that $U\leq V$ implies $I_A^\e(U)\leq I_A^\e(V).$ Then define the functor $F_A^\e:\mc{T}(T^*M)\to \textbf{Comm} $ (the category of commutative rings)
by 
$$F_A^\e(U)=
\begin{cases}
h^{I_A^\e(U)}S^{\comp}_\delta(M)|_U\,/\,h^{I_A^\e(U)+1-2\delta}S^{\comp}_\delta(M)|_U  & I_A^\e(U)\neq \infty\\
\{0\}  &I_A^\e(U)=\infty 
\end{cases},
$$
$$F_A^\e(U\to V)=
\begin{cases}
h^{I_A^\e(V)-I_A^\e(U)}|_V&I_A^\e(V)\neq \infty \\ 
0 &I_A^\e(V)=\infty
\end{cases}.$$

Then $F_A^\e$ is a presheaf on $T^*M$. We sheafify $F_A^\e$, still denoting the resulting sheaf by $F_A^\e$, and say that {\emph{$A$ is of $\e$-class $F_A^\e.$}} We define the \emph{stalk} of the sheaf at $q$ by $F_A^\e(q):=\varinjlim_{q\in U}F_A^\e(U).$

Now, for every $U\subset \mc{T}(T^*M)$, $I_A^\e(U)\neq \infty$, there exists $\chi_U\in \Cc(T^*M)$ with $\chi_U\equiv 1$ on $U$ such that  
$ \sigma_{I_A^\e(U)}(\oph(\chi_U) A\oph(\chi_U))|_U\neq 0.$
Then we define the \emph{$\e$-shymbol of $A$} to be the section of $F_A^\e$, $\tilde\sigma^\e_{(\cdot)}(A):\mc{T}(T^*M)\to F_A^\e(\cdot)$, given by
$$\tilde{\sigma}^\e_U(A) :=\begin{cases}\sigma_{I_A^\e(U)}(\oph(\chi_U) A\oph(\chi_U))|_U&I_A^\e(U)\neq \infty\\0&I_A^\e(U)=\infty\end{cases}.$$
Define also the \emph{$\e$-stalk shymbol}, $\tilde{\sigma}^\e(A)_q$ to be the germ of $\tilde{\sigma}^\e(A)$ at $q$ as a section of $F_A^\e.$ 

Now, define 
$$I_A^\e(q):=\sup\{I_A^\e(U)\mid q\in U\}.$$  We then define the simpler \emph{compressed shymbol}. Let $U_n\downarrow \{q\}$ be a sequence of open sets.
\begin{equation}
\label{def:symbFunc}
\begin{gathered}
\tilde{\sigma}^\e(A):T^*M\to \bigsqcup_q\quotient{h^{I_A^\e(q)}\complex}{h^{I_A^\e(q)+1-2\delta}\complex}\quad \text{by}\\
\tilde{\sigma}^\e(A)(q):=\begin{cases} 0 &I_A^\e(q)=\infty\\
\lim\limits_{n}^{}\tilde{\sigma}^\e_{U_n}(A)(q)&I_A^\e(q)<\infty\end{cases}\end{gathered}
\end{equation}

The limit in \eqref{def:symbFunc} exists since if $I_A^\e(q)<\infty$, then there exists $U\ni q$ such that for all $V\subset U$, $I_A^\e(V)=I_A^\e(U).$ This also shows that the limit is independent of the choice of sequence of $U_n\downarrow q.$
It is easy to see from standard composition formulae that the compressed shymbol has
$$\tilde{\sigma}^\e(AB)(q)=\tilde{\sigma}^\e(A)(q)\tilde{\sigma}(B)(q),\quad A\in h^{-N}\Psi_{\delta}^{\comp}\text{ and }B\in h^{-M}\Psi_{\delta}^{\comp}.$$
Moreover, 
\m \tilde{\sigma}^\e([A,B])(q)=-ih\left\{\tilde{\sigma}^\e(A)(q),\tilde{\sigma}^\e(B)(q)\right\}.\m

The following lemma follows from standard formulas for the composition of FIOs combined with the definitions above:
\begin{lemma}
\label{lem:EgorovSheaf}
Suppose that $A\in \Psi^{\comp}_\delta$ and let $T$ be a semiclassical FIO associated to the symplectomorphism $\kappa$ with elliptic symbol $t\in S_\delta$. Then for $N>0$ independent of $h$ 
$(AT)_N:=(T^*A^*)^N\left(AT\right)^N$ has $$ \tilde{\sigma}^\e((AT)_N)(q)=\prod\limits_{i=1}^N\left(|\tilde{\sigma}^\e(A)t|^2\composed \kappa^i(q)+\O{}\left(h^{I_{A_i}^\e(\beta^k(q))+1-2\delta}\right)\right).$$
\end{lemma}
\begin{proof}
Fix $q\in T^*M$. Let $\chi_k\in \Cc(T^*M)$ have $\chi_k= 1$ on $B\left(q,\recip{k}\right)$, the open ball of radius $k^{-1}$ around $q$, and $\supp \chi_k\subset B\left(q,\frac{2}{k}\right).$ Then let $D:=\oph(\chi_k)(AT)_N\oph(\chi_k).$ We have that 
\m D=\oph(\chi_k)(A_NTA_{N-1}T\dots A_{1}T)^*(A_NTA_{N-1}T\dots A_{1}T)\oph(\chi_k)+\O{\Psi^{\comp}_{\delta}}(h^\infty)\,\,\m
where $A_i=\oph(\psi_{k,i})A\oph(\psi_{k,i})$ with $C_c^\infty(T^*M)\ni \psi_{k,i}=1$ in some neighborhood of $\beta^i(q)$ and is supported inside a neighborhood $U_{k,i}$ of $\beta^i(q)$ such that $U_{k,i}\downarrow q.$ Then the result follows from standard composition formulae in Lemma \ref{lem:FIOcomp}.
\end{proof}
Now, since $\e>0$ is arbitrary, we define the \emph{semiclassical order of $A$ at $q$} by $I_A(q):=\sup_{\e>0}I_A^\e(q)$ with the understanding that $f=\O{}(h^{I_A(q)})$ means that for any $\e>0$, 
$$|f(q)|\leq C_\e h^{I_A(q)-\e}.$$
Furthermore, we suppress the $\e$ in the notation $\tilde{\sigma}^\e(A)(q)$ and denote the \emph{compressed shymbol}, $\tilde{\sigma}(A)(q)$, again with the understanding that for any $\e>0$,
$$\tilde{\sigma}(A)(q)\in \quotient{h^{I_A(q)-\e}\mathbb{C}}{h^{I_A(q)+1-2\delta-\e}\mathbb{C}}.$$


\section{A second microlocal calculus}
\label{sec:secondMicrolocal}

In the present work, it will be necessary to localize $h^{2/3}$ near the glancing submanifold in $T^*\pO$. In order to do this, we present the second microlocal calculus from \cite{SjoZwDist}. 

\subsection{The local model}
We start by considering the model case of $\Sigma_0=\{\xi_1=0\}\subset T^*\re^d$. Suppose that $U$ is a neighborhood of $(0,0)$ and $a\in \Cc(U)$. In that case, we write $a=a(x,\xi,\lambda;h)$ with $\lambda=h^{-\delta}\xi_1$. Suppose that $\e<\min(1/2,\delta)$, and $\e+\delta\leq 1$. We say that $a\in S^{k_1}_{\delta,\e}(\Sigma_0)$ if and only if 
\begin{equation} 
\label{semi-d:localSymb}
\partial_x^{\varsigma} \partial_\xi^{\varpi}\partial_\lambda^ka(x,\xi,\lambda;h)=\O{}(h^{-\e(|{\varsigma}|+|{\varpi}|)}\la h^\e\lambda\ra^{k_1-k}).
\end{equation}
We will write 
$$a=\wt{\O{\e}}(\la h^\e\lambda \ra^{k_1})\quad\text{ if and only if }\quad \text{\eqref{semi-d:localSymb} holds}.$$
For such $a$, we define the exact quantization 
$$\wt{\oph}(a)u=\frac{1}{(2\pi h)^d}\int a\left(\frac{x+y}{2},\xi,h^{-\delta}\xi_1;h\right)e^{\frac{i}{h}\la x-y,\xi\ra}u(y)dyd\xi.$$
Then, 
\begin{lemma}
Suppose that $a=\wt{\O{\e}}(\la h^\e\lambda\ra^{k_1})$ and $b=\wt{\O{\e}}(\la h^\e\lambda\ra^{k_2}).$ Then,
$$\wt{\oph}(a)\composed \wt{\oph}(b)=\wt{\oph}(a\sharp b).$$
where 
$$a\sharp b=\left.e^{ihA(D)}(a|_{\lambda=h^{-\delta}\xi_1}b|_{\mu=h^{-\delta}\eta_1})\right|_{\substack{y=x\\\xi=\eta}}=\wt{\O{\e}}(\la h^\e\lambda\ra^{k_1+k_2})$$
where 
$$A(D)=\frac{1}{2}\sigma((D_x,D_\xi),(D_y,D_\eta))=:\frac{1}{2}\la QD,D\ra.$$
Moreover if $\e+\delta<1$, 
$$a\sharp b=\left.\sum_{k=0}^\infty \frac{i^kh^k}{k!}A(D)^k(a|_{\lambda=h^{-\delta}\xi_1}b|_{\mu=h^{-\delta}\eta_1})\right|_{\substack{y=x\\\eta=\xi}}\mod h^\infty \Ph{-\infty}{}.$$
\end{lemma}

We say that $a(x,\xi,y,\eta,\lambda,\mu)=\wt{\O{\e}}(\la h^\e\lambda\ra^{k_1}\la h^\e \mu\ra^{k_2})$ if 
$$|\partial_x^{{\varsigma}_1}\partial_y^{{\varsigma}_2} \partial_\xi^{{\varpi}_1}\partial_\eta^{{\varpi}_2}\partial_\lambda^{m_1}\partial_\mu^{m_2}a|\leq C_{{\varsigma}{\varpi} m}h^{-\e(|{\varsigma}|+|{\varpi}|)}\la h^\e\lambda\ra^{k_1-m_1}\la h^\e \mu\ra^{k_2-m_2}.$$ 
The only part of this lemma that is non-standard is the following. The rest follows from applying stationary phase.
\begin{lemma}
$$e^{ihA(D)}:\wt{\O{\e}}(\la h^\e\lambda\ra^{k_1}\la h^\e \mu\ra^{k_2})\to \wt{\O{\e}}(\la h^\e\lambda\ra^{k_1}\la h^\e \mu\ra^{k_2}).$$
\end{lemma}
\begin{proof}
We start by considering the case of one dimension. Let $w_1=(x'_1,\xi'_1,y_1',\eta'_1)$ and 
$$\varphi_1(w_1)=\frac{1}{2}(\la \xi_1',y_1'\ra-\la \eta_1',x_1'\ra).$$
Then, with $z=(x_1,\xi_1,y_1,\eta_1)$,
$$ c:=(e^{ihA(D)}a)(z,\mu)=
Ch^{-2}\int e^{-\frac{i}{h}\varphi_1(w)}a(w-z,\lambda-h^{-\delta}\xi_1,\mu-h^{-\delta}\eta_1)dw
$$
Then, rescale $(x_1',y_1')=(\tilde{x}_1,\tilde{y}_1)h^{-(1-\delta)}$, and $(\xi_1',\eta_1')=(\tilde{\xi}_1,\tilde{\eta}_1)h^{-\delta}.$ 
We have that with $\tilde{w}=(\tilde{x}_1,\tilde{\xi}_1,\tilde{y}_1,\tilde{\eta}_1)$, 
\begin{align*} 
c&=C\int e^{-i\varphi_1(\tilde{w})}\left(\chi(\tilde{w}) a(x_1-h^{1-\delta}\tilde{x}_1,\xi_1-h^{\delta}\tilde{\xi}_1,y_1-h^{1-\delta}\tilde{y}_1,\eta_1-h^{\delta}\tilde{\eta}_1,\lambda-\tilde{\xi}_1,\mu-\tilde{\eta}_1)\right.\\
&\quad\quad+ (1-\chi(\tilde{w}))a(x_1-h^{1-\delta}\tilde{x}_1,\xi_1-h^{\delta}\tilde{\xi}_1,y_1-h^{1-\delta}\tilde{y}_1,\eta_1-h^{\delta}\tilde{\eta}_1,\lambda-\tilde{\xi}_1,\mu-\tilde{\eta}_1)d\tilde{w}\\
&=:A+B
\end{align*} 
where $\chi \in \Cc(\re^{4})$ has $\chi\equiv 1$ on $B(0,1)$ and $\supp \chi\subset B(0,2)$. 
$$|\partial^{\varsigma} A(z,\lambda,\mu)|\leq C\sup_{|\tilde{w}|\leq 2}|\partial^{\varsigma} a(x_1-h^{1-\delta}\tilde{x}_1,\xi_1-h^{\delta}\tilde{\xi}_1,y_1-h^{1-\delta}\tilde{y}_1,\eta_1-h^{\delta}\tilde{\eta}_1,\lambda-\tilde{\xi}_1,\mu-\tilde{\eta}_1)|$$
and hence $A=\wt{\O{\e}}(\la h^\e\lambda \ra^{k_1}\la h^\e \mu\ra^{k_2}).$ Letting 
$$L:=\frac{-\la \partial\varphi(\tilde{w}), D_{\tilde{w}}\ra}{|\partial\varphi(\tilde{w})|^2}$$ 
and integrating by parts sufficiently many times shows also that 
$B=\wt{\O{\e}}(\la h^\e\lambda \ra^{k_1}\la h^\e \mu\ra^{k_2}).$

To obtain the general case, we simply observe that 
$$e^{ihA(D_x,D_\xi,D_y,D_\eta)}=e^{ihA(D_{x'},D_{\xi'},D_{y'},D_{\eta'})}e^{ihA(D_{x_1},D_{\xi_1},D_{y_1}D_{\eta_1})}$$
and use that 
$$e^{ihA(D_{x'},D_{\xi'},D_{y'},D_{\eta'})}:S_\e\to S_\e.$$
\end{proof}

Now, rewriting the asymptotic expansion, and assuming that $|\xi_1|\leq C$ so that 
$$h^{1-2\e}\leq Ch^{1-\delta-\e}\la h^\e \lambda\ra^{-1}$$
 we have if $\e+\delta<1$, taking $p_1>\frac{1-2\e}{1-\delta-\e}$
\begin{align*} 
a\sharp b(x,\xi,\lambda;h)&=\sum_{k=0}^\infty \left.\frac{i^kh^k}{2^kk!}(\sigma(D_x,D_{\xi_1}+h^{-\delta}D_\lambda,D_{\xi'},D_y,D_{\eta_1}+h^{-\delta}D_\mu,D_{\eta'}))^kab\right|_{\substack{y=x,\,\eta=\xi\\\lambda=\mu}}\\
&=ab+\frac{1}{2i}h^{1-\delta}(\partial_\lambda b\partial_{x_1}a-\partial_\lambda a\partial_{x_1}b)+\frac{h}{2i}\{a,b\}\\
&\quad\quad +\sum_{k=2}^{p_1}\left.\frac{i^kh^{k(1-\delta)}}{2^kk!}(\sigma(D_{x_1},D_\lambda,D_{y_1},D_\mu))^kab\right.|_{\substack{y=x,\,\eta=\xi\\\lambda=\mu}}+\wt{\O{\e}}\left(h^{2-3\e-\delta}\la h^\e\lambda\ra^{k_1+k_2-1}\right)
\end{align*}

\subsubsection{Ellipticity and Boundedness  in the local model}
\label{s:ell2}
We now present the analogs of microlocal elliptic estimates and the sharp G$\mathring{\text{a}}$rding inequalities in the second microlocal setting. 
Suppose that $\e+\delta<1$ and $a=\wt{\O{\e}}(\la h^\e\lambda\ra^{k_1})$. We define the \emph{elliptic set} of $a$, $\Ell(a)$ by 
$(x,\xi,\lambda)\in \Ell(a)$ if there exists a neighborhood, $U$ of $(x,\xi,\lambda)$ and $c>0$ so that $|a|>c\la h^\e\lambda\ra^{k_1}$ on $U$. 
\begin{lemma}
\label{semi-l:elliptic}
Suppose that $p=\wt{\O{\e}}(\la h^\e\lambda\ra^{k_1})$, $b=\wt{\O{\e}}(\la h^\e\lambda\ra^{k_2})$ and that $\supp b\subset \Ell(p)$. Then there exists $a_i=\O{\e}(\la h^\e\lambda\ra^{k_2-k_1})$, $i=1,2$ so that 
$$\wt{\oph}(a_1)\wt{\oph}(p)=\wt{\oph}(p)\wt{\oph}(a_2)+\O{\Ph{-\infty}{}}(h^\infty)=\wt{\oph}(b)+\O{\Ph{-\infty}{}}(h^\infty).$$
\end{lemma}
\begin{proof}
By elementary analysis, one sees that
$$\partial^{\varsigma} p^{-1}=p^{-1}\sum_{k=1}^{|{\varsigma}|}\sum_{\substack{{\varsigma}={\varpi}^1+\dots +{\varpi}^k\\|{\varpi}^j|\geq 1}}C_{{\varpi}^1,\dots,{\varpi}^k}\prod_{j=1}^k(p^{-1}\partial^{{\varpi}_j}p)$$
 (see for example the proof of \cite[Theorem 4.32]{EZB}). Thus, since $|p|\geq c \la h^\e\lambda\ra^{k_1}$ on $\supp b$, 
 $$q_0:bp^{-1}=\wt{\O{\e}}(\la h^\e\lambda\ra^{k_2-k_1}).$$
So, 
$$\wt{\oph}(q_0)\wt{\oph}(p)=\wt{\oph}(b)+h^{1-\delta-\e}\wt{\oph}(e_1)+\O{\Ph{-\infty}{}}(h^\infty)$$
where $e_1=\wt{\O{\e}}(\la h^\e\lambda\ra^{k_2-1})$ with $\supp e_1\subset \Ell(p)$. Thus, setting $r_1=h^{1-\delta-\e}e_1$ and letting $q_1=-r_1p^{-1}=h^{1-\delta-\e}\wt{\O{\e}}(\la h^\e\lambda\ra^{k_2-k_1-1}).$ Continuing in this way, we obtain 
$$q_n= h^{n(1-\delta-\e)}\wt{\O{\e}}(\la h^\e\lambda\ra^{k_2-k_1-n})$$
 so that with 
$a_1\sim \sum q_i$, 
$$\wt{\oph}(a_1)\wt{\oph}(p)=\wt{\oph}(b)+\O{\Ph{-\infty}{}}(h^\infty).$$
A similar argument, yields $a_2$. 
\end{proof}

\begin{lemma}
Suppose that $a=\wt{\O{\e}}(\la h^\e\lambda\ra^{k_1}).$ Then
$$\|\wt{\oph}(a)\|_{L^2\to L^2}\leq C\la h^{(\e-\delta) k_1}\ra .$$ 
\end{lemma}
\begin{proof}
The proof for $h=1$ follows from for example \cite[Theorem 4.23]{EZB}. Suppose that $u(x)\in \mc{S}$. The proof follows that in \cite[Theorem 5.1]{EZB}. We have that 
$$\|\wt{\op}(a)\|_{L^2\to L^2}\leq C\sup_{|{\varsigma}|\leq Md}|\partial^{\varsigma} a|.$$
So, we rescale $\tilde{\xi}=h^{-\frac{(1+\delta)}{2}}\xi,$ $\tilde{x}=h^{-\frac{1-\delta}{2}}x$ and $\tilde{u}(\tilde{x})=h^{\frac{(1-\delta )d}{4}}u(h^{\frac{1-\delta}{2}}\tilde{x}).$ Then,
$$\wt{\oph}(a)u(x)=h^{-\frac{d(1-\delta)}{4}}\wt{\op}(a_h)\tilde{u}(\tilde{x}).$$
where 
$$a_h(\tilde{x},\tilde{\xi}):=a(h^{\frac{1-\delta}{2}}\tilde{x},h^{\frac{1+\delta}{2}}\tilde{\xi},h^{\frac{1-\delta}{2}}\tilde{\xi}_1).$$
Therefore, 
\begin{align*} 
\|\wt{\oph}(a)u\|_{L^2_x}&=\|\wt{\op}(a_h)\tilde{u}(\tilde{x})\|_{L^2_{\tilde{x}}}\leq \|\wt{\op}(a_h)\|_{L^2\to L^2}\|\tilde{u}\|_{L_{\tilde{x}}^2}\\
&\leq C\sup_{|{\varsigma}|\leq Md}|\partial^{\varsigma} a_h|\|u\|_{L_x^2}\\
&\leq C\sup_{|({\varsigma},{\varpi},k)|\leq Md}h^{(|{\varsigma}|+|{\varpi}|+k)\frac{1-\delta}{2}}|\partial_x^{\varsigma}\partial_\xi^{\varpi}\partial_\lambda^ka|\\
&\leq C\sup_{|({\varsigma},{\varpi},k)|\leq Md}h^{(|{\varsigma}|+|{\varpi}|+k)\frac{1-\delta}{2}}(1+\la h^{\e-\delta}\ra^{k_1-k})
\end{align*}
\end{proof}

We now prove an analog of the Sharp G$\mathring{\text{a}}$rding inequality for the second microlocal operators. 
\begin{lemma}
\label{semi-l:garding}
Suppose that $a=\wt{\O{0}}(\la \lambda\ra^{0})$ and $a\geq 0$. Then
$$\la \wt{\oph}(a)u,u\ra \geq -Ch^{1-\delta}\|u\|_{L^2}^2.$$
\end{lemma}
\begin{proof}
We again follow the proof in the classical case. (See for example \cite[Theorem 4.32]{EZB}). Fix $\tilde{h}$ sufficiently small and let $\gamma=h^{\e}/\tilde{h}$. We will show that 
$q=(a+\gamma)^{-1}$ satisfies
\begin{equation} 
\label{semi-e:sharpGardGoal} \partial_x^{\varsigma}\partial_\xi^{\varpi}\partial_\lambda^k q=\O{}(h^{-\e}\tilde{h}(\tilde{h} h)^{-\frac{\e}{2}(|{\varsigma}|+|{\varpi}|+k)}\la \lambda\ra^{-k}).
\end{equation}
That is $q\in h^{-\e}\tilde{h}S^{0}_{\delta+\e/2,\e/2}(\Sigma_0)$. We will then be able to invert $a+\gamma$ when $\e\leq 1-\delta$. 

First, since $a\geq 0$ and $a=\wt{\O{0}}(\la \lambda\ra^{0})$, $|\partial_\lambda a|\leq C\la \lambda\ra^{-1}a^{1/2}$. (see for example \cite[Lemma 4.31]{EZB}) Moreover, $|\partial_x a|+|\partial_\xi a|\leq Ca^{1/2}$. Then recall that 
\begin{equation} 
\label{semi-e:inverse}\partial^{\varsigma} (a+\gamma)^{-1}=(a+\gamma)^{-1}\sum_{k=1}^{|{\varsigma}|}\sum_{\substack{{\varsigma}={\varpi}^1+\dots +{\varpi}^k\\|{\varpi}^j|\geq 1}}C_{{\varpi}^1,\dots,{\varpi}^k}\prod_{j=1}^k((a+\gamma)^{-1}\partial_x^{{\varpi}_{j,1}}\partial_\xi^{{\varpi}_{j,2}}\partial_\lambda^{{\varpi}_{j,3}}a).
\end{equation}
Now,
$$|\partial_\lambda a|(a+\gamma)^{-1}\leq C\gamma^{-1/2}\la \lambda\ra^{-1}$$
and for $|{\varpi}|=1$
$$(|\partial_x^{\varpi} a|+|\partial_\xi^{\varpi} a|)(a+\gamma)^{-1}\leq C\gamma^{-1/2}.$$
Moreover, for $|({\varsigma},{\varpi}, k)|\geq 2$, 
$$|\partial_x^{\varsigma} \partial_\xi^{\varpi}\partial_\lambda^ka|(a+\gamma)^{-1}\leq C\gamma^{-1}\la \lambda\ra^{-k}.$$
So, 
$$\left|\prod_{j=1}^k(a+\gamma)^{-1}\partial_x^{{\varpi}_{j,1}}\partial_{\xi}^{{\varpi}_{j,2}}\partial_{\lambda}^{{\varpi}_{j,3}}a\right|\leq C\prod_{|{\varpi}|\geq 2}\gamma^{-1}\la \lambda\ra^{-{\varpi}_{j,3}}\prod_{|{\varpi}|=1}\gamma^{-1/2}\la \lambda\ra^{-{\varpi}_{j,3}}\leq C\la \lambda\ra^{-{\varsigma}_3}\gamma^{-|{\varsigma}|/2}$$
Plugging this into \eqref{semi-e:inverse} gives \eqref{semi-e:sharpGardGoal}.

We now choose $\e=1-\delta$. So, $a+\gamma\in S^0_{\delta,0}(\Sigma_0)\subset S^0_{\delta+\e/2,\e/2}(\Sigma_0).$ Then, write 
$a_1(x,\xi,\lambda_1)$ for the function so that
$$\wt{\oph_{\delta}}(a)=\wt{\oph_{\delta+\e/2}}(a_1).$$
Write also $q_1=(a_1+\gamma)^{-1}.$
 So we can define 
$(a_1+\gamma)\sharp q_1$ Then, using Taylor's formula and letting $w=(x,\xi)$, $z=(y,\eta)$, 
\begin{multline*} 
(a_1+\gamma)\sharp q_1=\left.e^{ihA(D)}(a_1+\gamma)|_{\lambda=h^{-\delta-\e/2}\xi_1}q_1|_{\mu=h^{-\delta-\e/2}\xi_1}\right|_{w=z}\\
=1+\int_0^1(1-t)\left.e^{ithA(D)}(ihA(D))^2(a_1(w,h^{-\delta-\e/2}\xi_1)q_1(z,h^{-\delta-\e/2}\eta_1))\right|_{w=z}dt\\
=:1+r(z).
\end{multline*}
Note that we have used that $\{a_1+\gamma, (a_1+\gamma)^{-1}\}=0$. Now, 
$(ihA(D))^2(a_1+\gamma)\sharp q_1\in \tilde{h}S^0_{\delta+\e/2,\e/2}(\Sigma_0)$. So, 
$$\|\wt{\oph\limits_{\delta+\e/2}}(r)\|_{L^2\to L^2}\leq C\tilde{h}\leq \frac{1}{2}$$ 
for $\tilde{h}$ small enough. Thus, $\wt{\oph\limits_{\delta+\e/2}}(q)$ is an approximate right (and similarly left) inverse for $\wt{\oph\limits_{\delta}}(a)+\gamma$.
This implies that $(\wt{\oph\limits_{\delta}}(a)+\gamma+\gamma_1)^{-1}$ exists for any $\gamma_1\geq0$ Therefore, 
$$\Spec(\wt{\oph_{\delta}}(a))\subset [-\gamma,\infty).$$
Thus, by \cite[Theorem C.8]{EZB} 
$$\la \wt{\oph}(a)u,u\ra \geq -\gamma \|u\|_{L^2}^2.$$
\end{proof}

Using the Sharp G$\mathring{\text{a}}$rding inequality, it is not hard to prove that
\begin{lemma}
\label{lem:upperBound}
Suppose $a=\widetilde{O}(\la \lambda\ra^0)$. Then,
$$\la \wt{\oph}(a)^*\wt{\oph}(a)u,u\ra\leq (\sup |a|+Ch^{1-\delta})\|u\|_{L^2}^2.$$
\end{lemma}

\subsection{The global second microlocal calculus}
Let $\Sigma\subset T^*M$ be a smooth compact hypersurface. Let $V_i$ denote vector fields tangent to $\Sigma$ and $W_i$ denote any vector fields.  Let $0\leq \delta <1$. We define the symbol class $S^{k_1,k_2}_{\delta}(M;\Sigma)$ by 
$a\in S^{k_1,k_2}_{\delta}(M;\Sigma)$ if and only if 
\begin{equation}
\label{semi-d:secondDef}\left\{\begin{gathered} \text{ near } \Sigma: V_1\dots V_{l_1}W_1\dots W_{l_2}a=\O{}(h^{-\delta l_2}\la h^{-\delta}d(\Sigma, \cdot)\ra^{k_1}),\\
\text{ away from }\Sigma: \partial_x^{\varsigma}\partial_\xi^{\varpi} a(x,\xi; h)=\O{}(h^{-\delta k_1}\la \xi\ra^{k_2-|{\varpi}|}).
\end{gathered}\right.
\end{equation}
where $d(\Sigma,\cdot)$ denotes the absolute value of any defining function of $\Sigma$ that behaves like $\la\xi\ra$ near fiber infinity. Then we have the following
\begin{lemma}
For $0\leq \delta <1$, there exists a class of operators, $\Ph{k_1,k_2}{\delta}(M;\Sigma)$, acting on $C^\infty(M)$ and maps
\begin{align*} 
\opht{\Sigma}&:S^{k_1,k_2}_{\delta}(T^*M;\Sigma)\to \Ph{k_1,k_2}{\delta}(M;\Sigma)\\
\sigma_{\Sigma}&:\Ph{k_1,k_2}{\delta}(M;\Sigma)\to \quotient{S^{k_1,k_2}_{\delta}(T^*M;\Sigma)}{h^{1-\delta}S^{k_1-1,k_2-1}_{\delta}(T^*M;\Sigma)}
\end{align*} 
such that
\begin{gather*}\sigma_{\Sigma}(A\composed B)=\sigma_{\Sigma}(A)\sigma_{\Sigma}(B),\\
\begin{aligned}
0&\to h^{1-\delta}\Ph{k_1-1,k_2-1}{\delta}(M;\Sigma)\to \Ph{k_1,k_2}{\delta}(M;\Sigma)\overset{\sigma_{\Sigma}}{\to}\quotient{S^{k_1,k_2}_{\delta}(T^*M;\Sigma)}{h^{1-\delta}S^{k_1-1,k_2-1}_{\delta}(T^*M;\Sigma)}
\to 0
\end{aligned}
\end{gather*}
is a short exact sequence, and 
$$\sigma_{\Sigma}\composed \opht{\Sigma}:S^{k_1,k_2}_{\delta}(T^*M;\Sigma)\to \quotient{S^{k_1,k_2}_{\delta}(T^*X;\Sigma)}{h^{1-\delta}S^{k_1-1,k_2-1}_{\delta}(T^*M;\Sigma)}$$
is the natural projection map. 
\end{lemma}
As in, \cite{SjoZwDist} near $\Sigma$ it is possible to reduce all computations to the case where $\Sigma=\Sigma_0:=\{\xi_1=0\}$. We then have analogs of all the properties from the model case for the global calculus. We sometimes suppress $M$ and $T^*M$ in our notation, writing only $S^{k_1,k_2}_{\delta}(\Sigma)$ and $\Ph{k_1,k_2}{\delta}(\Sigma)$. We also sometimes suppress the $\Sigma$ in $\opht{\Sigma}$ to simplify notation.


\section{The billiard ball flow and map}
\label{sec:billiard}
Recall that $\Omega\Subset \re^d$ is an open set with smooth boundary $\partial\Omega$. 
We need notation for the billiard ball flow and billiard ball map. Write $\nu$ for the outward pointing unit normal to $\pO$. Then
\m S^*\re^d|_{\partial\Omega}=\partial\Omega_+\sqcup\partial\Omega_-\sqcup\partial\Omega_0\,\,\m
where $(x,\xi)\in \partial\Omega_+$ if $\xi$ is pointing out of $\Omega$ (i.e. $\nu(\xi)>0$), $(x,\xi)\in\partial\Omega_-$ if it points inward (i.e $\nu(\xi)<0$), and $(x,\xi)\in\partial\Omega_0$ if $(x,\xi)\in S^*\partial\Omega$. The points $(x,\xi)\in \partial\Omega_0$ are called \emph{glancing} points. Let $B^*\partial\Omega$ be the unit coball bundle of $\partial\Omega$ and denote by $\pi_{\pm}:\partial\Omega_{\pm}\to B^*\partial\Omega$ and $\pi:S^*\re^d|_{\partial\Omega}\to \overline{B^*\partial\Omega}$ the canonical projections onto $\overline{B^*\partial\Omega}$. Then the maps $\pi_{\pm}$ are invertible. Finally, write 
\m t_0(x,\xi)=\inf\{t>0:\exp_t(x,\xi)\in T^*\re^d|_{\partial\Omega}\}\,\,\m
where $\exp_t(x,\xi)$ denotes the lift of the geodesic flow to the cotangent bundle. That is, $t_0$ is the first positive time at which the geodesic starting at $(x,\xi)$ intersects $\partial\Omega$.

\begin{figure}
\centering
\begin{tikzpicture}
\begin{scope}[shift={(4,3)}, scale =1.5]
\draw[ultra thin] (0,0) circle [radius=1];
\draw[dashed, thick, ->] (0,0) to (.8,.6);
\draw[ultra thick,->] (0,0)to (.373,.799);
\draw[thin, dashed] (0.373,0.799)to (0.8,0.6);
\draw (.7,-.9)node[right]{$S_{\pi_x(\beta(q))}^*\re^d$};
\end{scope}
\begin{scope}[scale=1.5]
\draw[ultra thin] (0,0) circle [radius=1];
\draw[ultra thick,->] (0,0)node[below]{$x$} to (.8,0)node[below]{$\xi$};
\draw[thin, dashed](.8,0) to (.8,.6);
\draw[dashed, thick, ->] (0,0) to (.8,.6);
\draw (-1,0)node[left]{$S_x^*\re^d$};
\end{scope}
\draw [thin](0,0) to [out=0, in =-115] (4,3) to [out=65, in =-85] (4.25,4.5);
\draw[ dashed, thin] (0,0) to (4,3);
\end{tikzpicture}
\caption[Billiard ball map]{\label{fig:billiardBallmap} The figure shows how the billiard ball map is constructed. Let $q=(x,\xi)\in B^*\partial\Omega$. The solid black arrow on the left denotes the covector $\xi\in B_x^*\partial\Omega$ and that on the right $\xi(\beta(q))\in B_{\pi_x(\beta(q))}^*\partial\Omega.$ The center of the left circle is $x$ and that of the right is $\pi_x(\beta(q)).$ }
\end{figure}
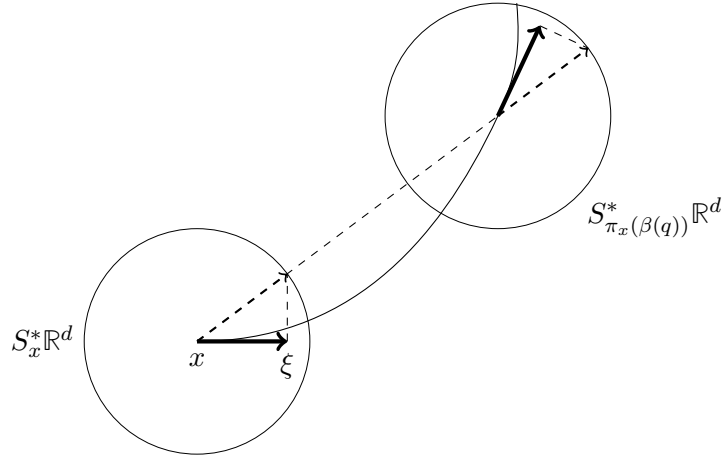

We define the broken geodesic flow as in \cite[Appendix A]{DyZw}. Without loss of generality, we assume $t_0> 0$. Fix $(x,\xi)\in S^*\re^d$ and denote $t_0=t_0(x,\xi)$. If $\exp_{t_0}(x,\xi)\in \partial\Omega_0$, then the billiard flow cannot be continued past $t_0$. Otherwise there are two cases: $\exp_{t_0}(x,\xi)\in \partial\Omega_+$ or $\exp_{t_0}(x,\xi)\in \partial\Omega_-$. We let 
$$(x_0,\xi_0)=\begin{cases}
\pi_-^{-1}(\pi_+(\exp_{t_0}(x,\xi)))\in \partial\Omega_-\,,&\text{if }\exp_{t_0}(x,\xi)\in \partial\Omega_+\\
\pi_+^{-1}(\pi_-(\exp_{t_0}(x,\xi)))\in \partial\Omega_+\,,&\text{if }\exp_{t_0}(x,\xi)\in \partial\Omega_-
\end{cases}.$$
We then define $\varphi_t(x,\xi)$, the \emph{broken geodesic flow}, inductively by putting 
$$\varphi_t(x,\xi)=\begin{cases}\exp_t(x,\xi)&0\leq t<t_0\\
\varphi_{t-t_0}(x_0,\xi_0)&t\geq t_0
\end{cases}.$$

We introduce notation from \cite{Saf1} for the billiard flow. Let $K$ be the set of ternary fractions of the form $0.k_1k_2,\dots$, where $k_j=0$ or $1$ and $S$ denote the left shift operator 
\m S(0.k_1k_2\dots)=0.k_2k_3\dots.\m

\noindent For $k\in K$, we define the billiard flow of type $k$, $G_k^t:S^*\re^d\to S^*\re^d$ as follows. For $0\leq t\leq t_0$,
\begin{equation}
\label{eqn:billiardFlow}G_k^t(x,\xi)=\begin{cases}\varphi_t(x,\xi)&\text{if }k_1=0\\
\exp_t(x,\xi)&\text{if }k_1=1
\end{cases}
\end{equation}
Then, we define $G_k^t$ inductively for $t>t_0$ by 
\begin{equation}
\label{eqn:billiardFlow2}
G_k^t(x,\xi)=G_{Sk}^{t-t_0}(G_k^{t_0}(x,\xi)).
\end{equation}
We call $G_k^t$ the billiard flow of type $k$. By \cite[Proposition 2.1]{Saf1}, $G_k^t$ is measure preserving.

\begin{remarks}
\item In \cite{Saf1}, geodesics could be of multiple types when total internal reflection occurred. However, in our situation, the metrics on either side of the boundary match, so there is no total internal reflection and geodesics are uniquely identified by their starting points and $k\in K$. 
\item In general, there exist situations where $G_k^t$ intersects the boundary infinitely many times in finite time. However, since we work in convex domains, we need not consider this situation. For a proof of this fact see the proof of Lemma \ref{lem:dynStrictlyConvex}. Note of course that the number of possible reflection in a given time $T$ grows as one approaches glancing points.
\end{remarks}
Now, for $k\in K $ and $T>0$, we define the set $\mc{O}_{T,k}\subset S^*\re^d$ to be the complement of the set of $(x,\xi)$ such that one can define the flow $G_k^t$ for $t\in[0,T]$. That is, $\mc{O}_{T,k}$ is the set for which the billiard flow of type $k$ is glancing in time $0\leq t\leq T.$ Last, define the set
\begin{equation}
\label{eqn:glancingSet}
\mc{O}_T=\bigcup_{k\in K}\mc{O}_{k,T}.
\end{equation}

The billiard ball map reduces the dynamics of $G_0^k$ to the boundary. We define the billiard ball map as in \cite{HaZel}. Let $(x,\xi')\in B^*\partial\Omega$ and $(x,\xi)=\pi_-^{-1}(x,\xi')\in \partial\Omega_-$ be the unique inward pointing covector with $\pi(x,\xi)=(x,\xi')$. Then, the billiard ball map $\beta:B^*\partial\Omega\to \overline{B^*\partial\Omega}$ maps $(x,\xi')$ to the projection onto $T^*\partial\Omega$ of the first intersection of the billiard flow with the boundary. That is,
\begin{equation}
\label{eqn:billiardMap}\beta:(x,\xi')\mapsto \pi(\exp_{t_0(x,\xi)}(x,\xi)).\end{equation}

\begin{remarks} 
\item Just like the billiard flow, the billiard ball map is not defined for $(x,\xi')\in \pi(\partial\Omega_0)=S^*\partial\Omega$. However, since we consider convex domains, $\beta:B^*\Omega\to B^* \Omega$ and $\beta^n$ is well defined on $B^*\partial\Omega$.
\item Figure \ref{fig:billiardBallmap} shows the process by which the billiard ball map is defined.
\end{remarks}

The billiard ball map is symplectic. This follows from the fact that the Euclidean distance function $|x-x'|$ is locally a generating function for $\beta$; that is, the graph of $\beta$ in a
neighborhood of $(x_0, \xi_0, y_0,\eta_0)$ is given by
\begin{equation}
\label{eqn:locBillRelation}
\{(x\,,\,-d_x|x-y|\,,\,y\,,\,d_y|x-y|\,)\,:(x,y)\in \partial\Omega\times \partial\Omega\}.
\end{equation}
 We denote the graph of $\beta$ by $C_b$. For strictly convex $\Omega$, $C_b$ is given globally by \eqref{eqn:locBillRelation}.
 
 We also write 
 $$\beta_E:=(x(\beta(x,\xi/\sqrt{E})),\sqrt{E}\xi(\beta(x,\xi/\sqrt{E}))):B_E^*\pO\to B_E^*\pO)$$
 where $B_E^*\pO$ is the coball bundle of radius $\sqrt{E}$.

\subsection{Dynamics in Strictly Convex Domains}
\label{sec:dynamicsConvex}
We are interested in the behavior of the billiard ball map, $\beta(q)$ when $|\xi'(q)|_g$ is close to 1. Our interest in this region comes from a desire to understand how the reflection coefficients $R$ from \eqref{eqn:reflect} behaves when a wave travels nearly tangent to a strictly convex boundary.

Fix $q=(x_0,\xi_0)\in B^*\pO$ so that $\partial \Omega$ is strictly convex near $x_0$ and $|\xi_0|_g^2$ is sufficiently close to 1. Let $\gamma:[0,\delta)\to \partial \Omega$ be the unique length minimizing geodesic connecting $x_0$ and $\pi_x(\beta(q)).$ The existence and uniqueness of such a geodesic is guaranteed for $|\xi_0|_g^2$ close enough to 1 by the strict convexity of $\pO$. Indeed, this follows from the fact that $l(q,\beta(q))\to 0$ as $|\xi_0|_g^2\to 1$ and the fact that the exponential map is a diffeomorphism for small times.  

Let $s\in [0,\delta)$ have $\gamma(s)=\pi_x(\beta(q)).$ We first examine how the normal component to $\partial \Omega$ changes under the billiard ball map. Let $\Delta_{\xi_d}$ denote the change in the normal component under $\beta$. Then
\begin{align*}
\Delta_{\xi_d}&=\frac{((\gamma(s)-\gamma(0))\cdot \nu(0)-(\gamma(0)-\gamma(s))\cdot\nu(s))}{|\gamma(s)-\gamma(0)|}\\
&=\frac{(\gamma(s)-\gamma(0))\cdot(\nu(0)+\nu(s))}{|\gamma(s)-\gamma(0)|}.
\end{align*}
Here $|\cdot|$ is the euclidean norm in $\re^d$ and $\nu$ is the inward pointing unit normal.

First, note that 
\begin{gather*} 
\gamma''(s)=\kappa(s)\nu(s),\qquad\quad \nu'(s)\cdot \gamma'(s)=-\kappa(s),\\
\gamma'(s)\cdot \nu(s)=0,\qquad \quad\|\gamma'(s)\|=\|\nu(s)\|=1
\end{gather*}
where $\kappa(s)$ is the curvature of the geodesic $\gamma$ as a curve in $\re^d$. 
Then, expanding in Taylor series gives
\begin{align}
\Delta_{\xi_d}\left[s+\O{}(s^2)\right]&=\left[\gamma'(0)s+\gamma''(0) \tfrac{s^2}{2} +\gamma^{(3)}(0)\tfrac{s^3}{6} +\O{}(s^4)\right]\cdot \left[2\nu(0)+\nu'(0)s+\nu''(0)\tfrac{s^2}{2}+\O{}(s^3)\right]\nonumber\\
\Delta_{\xi_d}\left[1+\O{}(s)\right]&=2\gamma'(0)\cdot \nu(0)+\left(\gamma'\cdot \nu\right)'(0)s+\left(2\gamma^{(3)}(0)\cdot \nu(0)+3(\gamma'\cdot \nu')'(0)\right)\tfrac{s^2}{6}+\O{}(s^3)\nonumber\\
\Delta_{\xi_d}&=\left[2(\kappa'(0)\nu(0)-\kappa(0)\nu'(0))\cdot\nu(0)-3\kappa'(0)\right]\tfrac{s^2}{6}+\O{}(s^3)\nn
\label{eqn:dynamicsApprox}\Delta_{\xi_d}&=(2\kappa'(0)-3\kappa'(0))\tfrac{s^2}{6}+\O{}(s^3)=-\kappa'(0)\tfrac{s^2}{6}+\O{}(s^3).
\end{align}
Next observe that
$$\sqrt{1-|\xi'(q)|_g^2}=\frac{\gamma(s)-\gamma(0)}{|\gamma(s)-\gamma(0)|}\cdot\nu(0) =\frac{\kappa(0)}{2}s+\O{}(s^2)$$
Now, using $\kappa(0)>c>0$ for $\Omega$ strictly convex this implies 
$$s=\frac{2\sqrt{1-|\xi'(q)|_g^2}}{\kappa(0)}+\O{}((1-|\xi'|_g^2))$$
and therefore, 
$$l(q,\beta(q))=|\gamma(s)-\gamma(0)|=s+\O{}(s^2)=\frac{2}{\kappa(0)}\sqrt{1-|\xi'|_g^2}+\O{}(1-|\xi'|_g^2).$$
Summarizing, we have
\begin{lemma}
\label{lem:dynStrictlyConvex}
Let $\Omega\subset \re^d$ be strictly convex. Then, for $q\in B^*\partial\Omega$ sufficiently close to $S^*\pO$
\begin{gather*} \sqrt{1-|\xi'(\beta(q))|^2_g}=\sqrt{1-|\xi'(q)|^2_g}+\O{}(1-|\xi'(q)|_g^2)\\
l(q,\beta(q))=\frac{2}{\kappa(0)}\sqrt{1-|\xi'|_g^2}+\O{}(1-|\xi'|_g^2). 
\end{gather*}
\end{lemma}
This implies that set of $\O{}(h^\e)$ near glancing points is stable under the billiard ball map. This also follows from the equivalence of glancing hypersurfaces \cite{MelroseGlanceSurf}.



\section[Friedlander Model -- layer potentials and operators]{Boundary layer operators and potentials in the non-nomogeneous Friedlander model}
\label{sec:Friedlander}
Our goal is to give microlocal descriptions of the boundary layer operators and potentials near a glancing point. We start by considering the non-homogeneous Friedlander model problem 
\begin{gather} 
((hD_{x_d})^2-\mu {x_d} +hD_{y_1})u=0,\quad\quad u(0,y)=f(y)\nonumber\\
u|_{x_d>0}\text{ outgoing},\quad\quad\quad \|u\|_{L^2((-\infty,0]\times \re^{d-1})}<\infty.\label{eqn:FriedBCs}
\end{gather} 
Then, let $\Fh(u)$ denote the semiclassical Fourier transform in $y$,
$$\Fh{u}(x_d,\eta):=\frac{1}{(2\pi h)^{d-1}}\int u(x_d,y)e^{-\frac{i}{h}\la y,\eta\ra}dy.$$
Rescaling $w=h^{-2/3}\mu^{1/3}{x_d}$ gives that 
$$(h^{2/3}\mu^{-1/3}(D_w^2-w+h^{-2/3}\mu^{-2/3}{\eta_1})\Fh(u)(w,\eta)=0\,,\,\quad \quad \Fh(u)(0,\eta)=\Fh(f)(\eta).$$
Hence, using \eqref{eqn:FriedBCs}
$$\Fh(u)({x_d},\eta)=\begin{cases}\frac{Ai(-h^{-2/3}\mu^{1/3}{x_d}+h^{-2/3}\mu^{-2/3}{\eta_1})}{Ai(h^{-2/3}\mu^{-2/3}{\eta_1})}\Fh(f)(\eta)&{x_d}<0\\
\frac{A_-(-h^{-2/3}\mu^{1/3}{x_d}+h^{-2/3}\mu^{-2/3}{\eta_1})}{A_-(h^{-2/3}\mu^{-2/3}{\eta_1})}\Fh(f)(\eta)&{x_d}>0
\end{cases}
.$$
So, the Dirichlet to Neumann map for the interior problem $({x_d}<0)$ is given by 
$$\Fh(N_1f)(\eta)=-h^{-2/3}\mu^{1/3}\frac{Ai'(h^{-2/3}\mu^{-2/3}{\eta_1})}{Ai(h^{-2/3}\mu^{-2/3}{\eta_1})}\Fh(f)(\eta)$$
and that for the exterior problem $(x_d>0)$ by 
$$\Fh( N_2f)(\eta)=h^{-2/3}\mu^{1/3}\frac{A_-'(h^{-2/3}\mu^{-2/3}{\eta_1})}{A_-(h^{-2/3}\mu^{-2/3}{\eta_1})}\Fh(f)(\eta).$$
\begin{remark}
Since the goal of this section is only to present a simple model where the calculations are exact, we ignore the poles in $N_1$. It is possible to find the single and double layer operators and potentials without using the Dirichlet to Neumann map $N_1$ (see \cite[Section 4.5]{Galk} see also \cite[Section 7.11]{Taylor} for a general introduction to layer potential methods), but it simplifies the presentation to do so here.
\end{remark}
So, letting $\Theta_h(\eta)=h^{-2/3}\mu^{-2/3}{\eta_1}$, the single layer operator is given by 
\begin{align*} 
\Fh(Gf)(\eta)&=\Fh((N_1+N_2)^{-1}f)(\eta)=h^{2/3}\mu^{-1/3}\frac{Ai(\Theta_h)A_-(\Theta_h)}{A_-'(\Theta_h)Ai(\Theta_h)-Ai'(\Theta_h)A_-(\Theta_h)}\Fh(f)(\eta)\\
&=h^{2/3}\mu^{-1/3}2\pi e^{\pi i/6}Ai(\Theta_h)A_-(\Theta_h)\Fh(f)(\eta)
\end{align*} 
and the double layer operator is given by 
\begin{align*}
\Fh (\Dl f)(\eta)&=\frac{1}{2}\Fh(f)(\eta)-\Fh(GN_2f)(\eta)\\
&=\left(\frac{1}{2}-2\pi e^{\pi i/6}Ai(\Theta_h)A_-'(\Theta_h)\right)\Fh(f)(\eta)
\end{align*} 
Therefore, since 
$$\gamma^+\S  =G\,\quad \quad \gamma^+\D =-\frac{1}{2}I+\Dl$$
and both solve the Friedlander model equation away from ${x_d}=0$,
\begin{align*}
\Fh(\S f)&=h^{2/3}\mu^{-1/3}2\pi e^{\pi i/6}Ai(-h^{-2/3}\mu^{1/3}{x_d}+\Theta_h)A_-(\Theta_h)\Fh(f)(\eta)\\
\Fh(\D f)&=-2\pi e^{\pi i/6}Ai(-h^{-2/3}\mu^{1/3}{x_d}+\Theta_h)A_-'(\Theta_h)\Fh(f)(\eta)
\end{align*} 
Now, consider the kernel of $\S ^*\S $,
\begin{align*}
\S ^*\S (x',y')&=\frac{4\pi^2\mu^{-2/3}h^{4/3}}{(2\pi h)^{2d-2}}\iint_{-\infty}^0 \overline{Ai(-h^{-2/3}\mu^{1/3}w_1+\Theta_h(\eta))}\overline{A_-(\Theta_h(\eta))}\\
&\quad \quad Ai(-h^{-2/3}\mu^{1/3}w_1+\Theta_h(\xi))
 A_-(\Theta_h(\xi))e^{\frac{i}{h}(\la x'-w',\eta\ra +\la w'-y',\xi\ra}dw_1d\xi dw'd\eta \\
 &=\frac{4\pi^2\mu^{-1}h^{2}}{(2\pi h)^{d-1}}\iint_{\Theta_h(\xi)}^\infty |Ai(s)|^2|A_-(\Theta_h(\xi))|^2e^{\frac{i}{h}\la x'-y',\xi\ra}d\xi\\
 &=\frac{h^2}{\mu}\frac{1}{(2\pi h)^{-d+1}}\int \Psi_{\S }(\Theta_h(\xi))e^{\frac{i}{h}\la x'-y',\xi\ra }dsd\xi
\end{align*} 
Similarly, 
\begin{align*} 
\S ^*\D (x',y')&=-\frac{h^{4/3}}{\mu^{2/3}}\frac{1}{(2\pi h)^{d-1}}\int \overline{\Psi_{\D \S }(\Theta_h(\xi))}e^{\frac{i}{h}\la x'-y',\xi\ra}d\xi\\
\D ^*\S (x',y')&=-\frac{h^{4/3}}{\mu^{2/3}}\frac{1}{(2\pi h)^{d-1}}\int \Psi_{\D \S }(\Theta_h(\xi))e^{\frac{i}{h}\la x'-y',\xi\ra}d\xi\\
\D ^*\D (x',y')&=\frac{h^{2/3}}{\mu^{1/3}}\frac{1}{(2\pi h)^{d-1}}\int \Psi_{\D }(\Theta_h(\xi))e^{\frac{i}{h}\la x'-y',\xi\ra}d\xi
\end{align*} 
where
\begin{equation}
\label{e:layerSymbDef}
\begin{aligned} 
\Psi_{\S }(x)&:=4\pi ^2\int_{x}^\infty |Ai(s)|^2|A_-(x)|^2ds&&=4\pi^2|A_-(x)|^2[(Ai'(x))^2-x(Ai(x))^2]\\
\Psi_{\D \S }(x)&:=4\pi ^2\int_{x}^\infty |Ai(s)|^2A_-(x)\overline{A_-'(x)}ds&&=4\pi ^2A_-(x)\overline{A_-'(x)}[(Ai'(x))^2-x(Ai(x))^2]\\
\Psi_{\D }(x)&:=4\pi ^2\int_{x}^\infty |Ai(s)|^2|A_-'(x)|^2ds&&=4\pi ^2|A_-'(x)|^2[(Ai'(x))^2-x(Ai(x))^2]
\end{aligned} 
\end{equation}
since
$$\int_{x}^\infty (Ai(s))^2ds=(Ai'(x))^2-x(Ai(x))^2.$$
Using the Wronskian we have that $\Psi_{\S }(\zeta_j)=1$ where $\zeta_j$ is a zero of the Airy function, i.e. $Ai(\zeta_j)=0$. Moreover, using asymptotics for the Airy function, as $x\to -\infty$, 
$$\Psi_{\S }(x)\sim  1\,,\quad \Psi_{\D }(x)\sim -x\,,\quad \Psi_{\D \S }\sim i(-x)^{1/2}.$$
\begin{figure}
\centering
\includegraphics[width=.75\textwidth]{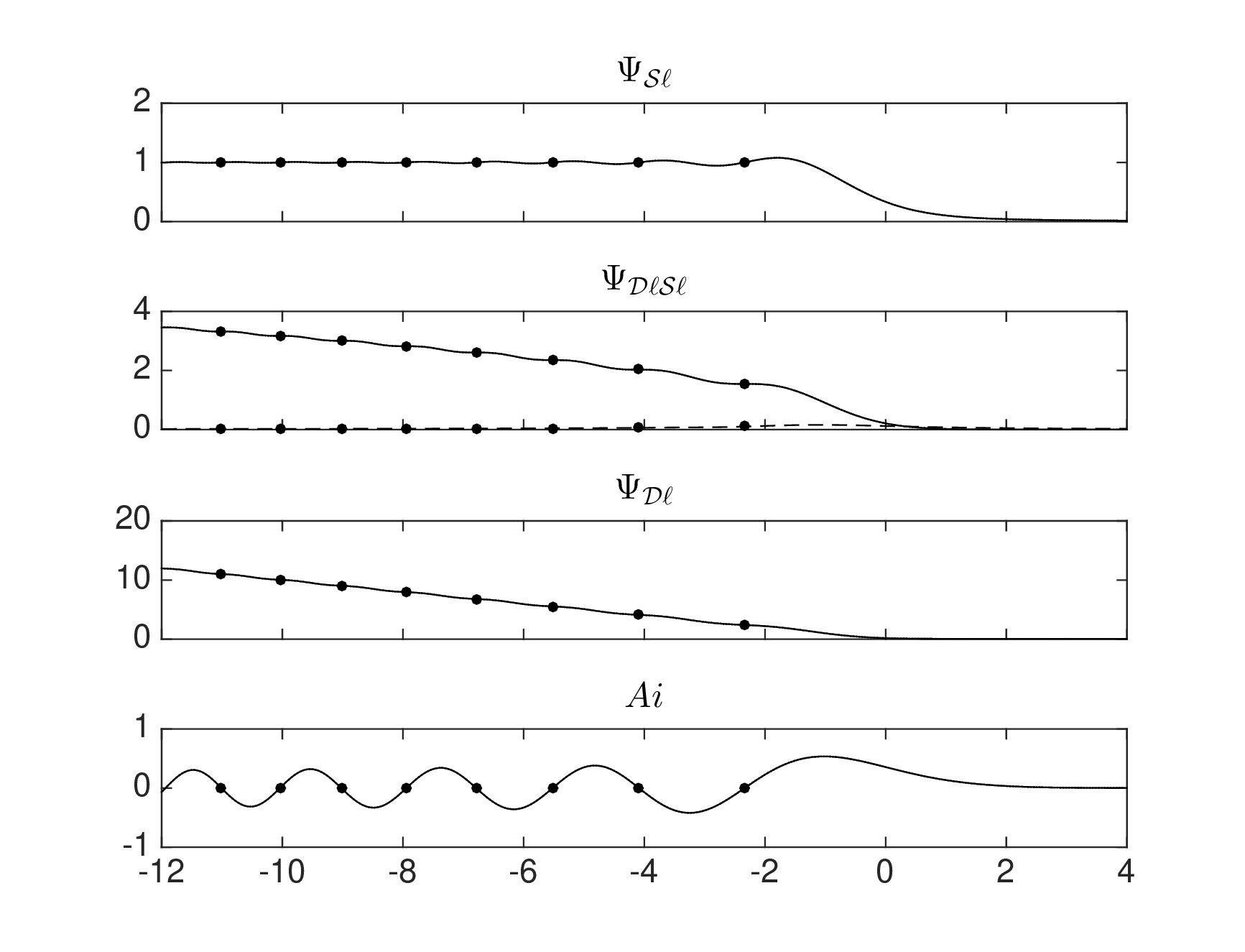}
\caption{We plot the symbols of $\Sl^*\Sl$, $\D^*\Sl$ and $\D^*\D$. From top to bottom, the graphs show $\Psi_{\Sl}$, $\Psi_{\D\Sl}$, $\Psi_{\D}$. The bottom graph shows $Ai$ for reference. In the graph of $\Psi_{\D\Sl}$, the imaginary part is shown in the solid line, and the real part in the dashed line. The black dots in each graph show $(\zeta_j,f(\zeta_j))$ where $\zeta_j$ are the zeros of $Ai(s)$ and $f$ is one of $\Psi_{\Sl}$, $\Psi_{\D\Sl}$, $\Psi{\D}$ or $Ai$ as described at the top of each graph.}
\end{figure}

\section{Analysis of the boundary layer operators and potentials near glancing}
\label{sec:Layer}
Our next task is to show that analogs of all of the formulas for the boundary layer operators and potentials from Section \ref{sec:Friedlander} hold in the general case.
 
\subsection{Preliminaries for the General Case}
In order to make an analysis similar to that for the model case, we use the microlocal models for $G$, $\Dl$, $\Sl$, and $\D$ developed in \cite[Section 4.5]{Galk} We recall the results here. The idea is to write a parametrix for the solution to the problem 
$$(-h^2\Delta-z^2)u=L^*\delta_{\partial\Omega}\otimes g_1+\delta_{\partial\Omega}\otimes g_2
$$
where $f_i$ are microlocalized near glancing and $\delta_{\partial\Omega}$ denotes the surface measure on $\partial\Omega$. The parametrix for the problem will be a sum of oscillatory integrals of the form
\begin{equation} 
\label{layer-e:parametrix} 
\begin{aligned}
H_1F&=(2\pi h)^{-d+1}\int (f_0Ai(h^{-2/3}\rho)+ih^{1/3}f_1Ai'(h^{-2/3}\rho))A_-(h^{-2/3}\Theta)e^{\frac{i}{h}\theta}\mc{F}_h(F)(\xi')d\xi'\\
H_2F&=(2\pi h)^{-d+1}\int (f_0Ai(h^{-2/3}\rho)+ih^{1/3}f_1Ai'(h^{-2/3}\rho))A_-'(h^{-2/3}\Theta)e^{\frac{i}{h}\theta}\mc{F}_h(F)(\xi')d\xi'.
\end{aligned}
\end{equation}
where $f_i$ solve certain transport equations and $\rho,\theta$ certain eikonal equations. The boundary values of $f_0$ and $f_1$ are determined by the limiting behavior of $\D g_1$ and $\Sl g_0$ at $\partial\Omega$.

Let $ z=1+i\mu$ with $|\mu|\leq Mh\log h^{-1}$. Then let 
$\e(h):=\max(h,|\mu|)$
Let $(x_0,\xi_0)\in S^*\pO$ and suppose that in coordinates $(x',x_d)$ near $x_0$, with $\pO=\{x_d=0\}$ and $x_d>0$ in $\Omega$,
$$-h^2\Delta= \sum_{ij} a^{ij}hD_{x_i}hD_{x_j}+h\Big(\sum_i b^ihD_{x_i}+c\Big).$$

Then there exist
$$\rho(x,\xi';h)=\rho_0+\sum_j\rho_j\e(h)^j,\quad\quad \theta(x,\xi';h)=\theta_0+\sum_j\theta_j\e(h)^j$$ 
solving the eikonal equations 
\begin{equation*}
\begin{cases}
z^2+\O{}(h^\infty)=\la ad\theta,d\theta\ra -\rho \la ad\rho,d\rho\ra \\
\O{}(h^\infty)=2\la d\theta,d\rho\ra
\end{cases}
\end{equation*}
on $\rho_0\leq 0$ and in Taylor series at $\rho_0=0$, $x_d=0$. Here, $\rho_0,\theta_0$ are real valued solving
\begin{equation*}
\begin{cases}
1=\la ad\theta_0,d\theta_0\ra -\rho \la ad\rho_0,d\rho_0\ra \\
0=2\la d\theta_0,d\rho_0\ra 
\end{cases}
\end{equation*}
on $\rho_0\leq 0$ and in Taylor series at $\rho_0=0$, $x_d=0$. 
We need a few additional properties of $\rho$ and $\theta$. In particular, 
\begin{equation}\label{eqn:propPhase}
\rho_0|_{\pO}=\xi_1,\quad\quad \partial_{x_d}\rho_0|_{\pO}>0,\quad\quad \partial^2_{x'\xi'}\theta_0|_{\pO}\neq 0\end{equation}
 and $\theta_{0b}:=\theta_0|_{\pO}$ has that 
\begin{equation}
\label{eqn:kappa}\kappa:(\partial_{\xi'}\theta_{0b}(x',\xi'),\xi')\mapsto (x',\partial_{x'}\theta_{0b}(x',\xi'))
\end{equation}
is a symplectomorphism reducing the billiard ball map for the Friedlander model case to that for $\Omega$. We also write $\theta_b=\theta|_{\pO}$. 
Next, let 
\begin{gather*}\Theta:=\rho|_{\pO}=\xi_1+i\e(h),\quad\quad \Theta_0:=\rho_0|_{\pO}=\xi_1.\end{gather*}

Finally, there exist
$$f_i\sim \sum_{j=0}^\infty f_{i,j}h^j,\quad\quad i=0,1$$
with $f_{0b}:=f_{0}|_{\pO}$ having $|f_{0b}|>c>0$ and $g_1|_{\pO}=0$ solving
\begin{equation}
\label{eqn:transEuc}
\left\{\begin{aligned}
\begin{gathered}2\la a d\theta_0,df_{0,n}\ra+2\rho_0\la ad\rho_0,df_{1,n}\ra  +\la b, df_{0,n}\ra \\
+\la ad\rho_0,d\rho_0\ra f_{1,n}-P_2\theta_0 f_{1,n}-\rho_0 (P_2\rho_0) f_{1,n}\end{gathered}
=F_{1,n}(\theta,\rho,f_{i,m<n},\mu)\\
{}\\
\begin{gathered}2\la ad\rho_0,df_{0,n}\ra-2\la a d\theta_0,df_{1,n}-\la b,df_{1,n}\ra \ra\\
-(P_2\rho_0)f_{0,n} +(P_2\theta_0)f_{1,n}\end{gathered}=F_{2,n}(\theta,\rho,f_{i,m<n},\mu).
\end{aligned}\right.
\end{equation}
on $\rho_0\leq 0$ and in Taylor series at $\rho_0=0$, $x_d=0$ so that for $H_i$ as in \eqref{layer-e:parametrix} $(-h^2\Delta-z^2)H_iF=\O{\Ph{-\infty}{}}(h^\infty)F$ whenever $F$ is supported $h^\e$ close to $\xi_1=0$. If $|\mu|\leq Ch$, then this also holds when $F$ is supported $\delta$ close to $\xi_1=0$ for $\delta$ small enough.

\subsubsection{Identification of $\partial_{x_d}\rho_0|_{x_d=0}$ and $|\partial_{y'}\theta_{0b}|_g^2$}
It will be useful to have the value of $\partial_{x_d}\rho_0|_{x_d=0}$ and $|\partial_{y'}\theta_{0b}|_g^2$. To obtain these, we simply write the eikonal equations in normal geodesic coordinates
Recall that in normal geodesic coordinates $(y',x_d)$ with $x_d>0$ in $\Omega$,
$$-h^2\Delta=(hD_{x_d})^2+R(y',hD_{y'})+2x_dQ(x_d,y',hD_{y'})+hF(x_d,y')hD_{x_d}$$
where 
\begin{gather*} R(y',D_{y'})=-\Delta_{\pO}=\bar{g}^{-1/2}\sum_{ij}D_{y_i}\bar{g}^{1/2}g^{ij}D_{y_j}\,,\quad \bar{g}=(\det (g^{ij}))^{-1/2}\\
Q(0,y',D_{y'})=\sum_{ij}D_{y_j}\bar{g}^{1/2}a_{ij}D_{y_i}
\end{gather*} 
where $Q(y',\xi')=\sum_{ij}a_{ij}(y')\xi_i\xi_j$ is the second fundamental form of $\pO$ lifted to $T^*\pO$, $g^{ij}=g^{ij}(y')$ is the metric on $T^*\pO$, and 
$R(y',\xi')=\sum_{ij}g^{ij}\xi_i\xi_j$ is the symbol of $-h^2\Delta_{\pO}$. 

Using the eikonal equations for $\rho_0$ and $\theta_0$ in these coordinates,
\begin{equation*}
\begin{cases}
\begin{aligned}
1&=(\partial_{x_d}\theta_0)^2+R(y',\partial_{y'}\theta_0)+2x_d(Q(y',\partial_{y'}\theta_0)+\O{}(x_d)) \\
&\quad \quad- \rho_0\left[(\partial_{x_d}\rho_0)^2+R(y',\partial_{y'}\rho_0)+2x_d(Q(y',\partial_{y'}\rho_0)+\O{}(x_d))\right]
\end{aligned}\\
0=2(\partial_{x_d}\theta_0\partial_{x_d}\rho_0+g^{ij}\partial_{y_i}\theta_0\partial_{y_j}\rho_0+2x_d(a_{ij}\partial_{y_i}\theta_0\partial_{y_j}\rho_0+\O{}(x_d)).
\end{cases}
\end{equation*}
Now, we know that $\rho_0|_{x_d=0}=\xi_1$ and $\partial_{x_d}\rho_0|_{x_d=0}>0$. So,
evaluation at $x_d=0$ shows
\begin{align*} 
1&=(\partial_{x_d}\theta_0)^2+R(y',\partial_{y'}\theta_0)-\xi_1(\partial_{x_d}\rho_0)^2=R(y',\partial_{y'}\theta_0)-\xi_1(\partial_{x_d}\rho_0)^2\\
0&=\partial_{x_d}\theta_0
\end{align*} 

Moreover, differentiating the first equation in $x_d$ and the second in $y'$ and evaluating at $x_d=0$ shows
\begin{align*} 
0&=2g^{ij}\partial^2_{x_dy_j}\theta_0\partial_{y_i}\theta_0+2Q(y',\partial_{y'}\theta_0)-(\partial_{x_d}\rho_0)^3-2\xi_1 \partial^2_{x_d}\rho_0\partial_{x_d}\rho_0\\
0&=2(\partial^2_{y'x_d}\theta_0\partial_{x_d}\rho_0)
\end{align*} 
Hence, 
\begin{gather*}
\label{eqn:rhoDer}(\partial_{x_d}\rho_0)^3|_{x_d=0}=2Q(y',\partial_{y'}\theta_0)-2(\xi_1\partial_{x_d}^2\rho_0\partial_{x_d}\rho_0)|_{x_d=0}=2Q(y;\partial_{y'}\theta_{0b})+\O{}(\xi_1)\\
\label{eqn:thetaDer}
R(y',\partial_{y'}\theta_{0b})=|\partial_{y'}\theta_{0b}|_g^2=1+\xi_1(\partial_{x_d}\rho_0)^2|_{x_d=0}
\end{gather*}
The implicit function theorem then implies that with $\xi'=\partial_{y'}\theta_{0b}$,
\begin{equation}
\label{eqn:xi1}\xi_1=\frac{|\xi'|_g^2-1}{(2Q(y,\xi'))^{2/3}}+\O{}((|\xi'|_g^2-1)^2),\quad\quad \partial_{x_d}\rho=2Q(y,\xi')+\O{}(|\xi'|_g^2-1).
\end{equation}

Now, in coordinates $(x,\xi)=\kappa^{-1}(y,\eta)$ where $\kappa$ is as in \eqref{eqn:kappa}, we have 
$$
\beta(x,\xi)=(x_1-2\sqrt{-\xi_1},x',\xi)
$$
since $\kappa$ reduces the Friedlander model to the billiard ball map for $\Omega$.
Let $\varphi_i$ be a partition of unity on $1-\e\leq |\xi'|_g\leq 1+ \e$ for some $\e>0$ small enough so that on $\supp \varphi_i$ $\kappa_i^{-1}$, with $\kappa_i$ given by \eqref{eqn:kappa}, is well defined. Let 
\begin{equation}
\label{eqn:approxInterp1}\Xi:=\sum_i\varphi_i \xi_1(\kappa_i^{-1}(x,\xi)).\end{equation}
Then we have the following lemma given the existence of an \emph{approximate interpolating Hamiltonian} for the billiard ball map. In particular, the lemma follows from the equivalence of glancing hypersurfaces \cite{MelroseGlanceSurf} (see \cite[Proposition 3.1]{PopovInterpolating} for a proof, see also \cite{MarviziMelrose})
\begin{lemma}
\label{lem:approxInterp}
Let $\Xi$ be as in \eqref{eqn:approxInterp1}. Then at $S^*\pO$, $\Xi=0$, $|d\Xi|>0$ and $\Xi<0$ in $B^*\pO$. Moreover,
\begin{gather*} \Xi\composed \beta(q)-\Xi(q)=\O{}((|\xi'|_g^2-1)^\infty),\\
\beta(q)-\exp(-2\sqrt{-\Xi}H_\Xi)(q)=\O{}((|\xi'|_g^2-1)^\infty),\\
\Xi(x',\xi')=\frac{|\xi'|_g^2-1}{(2Q(x',\xi'))^{2/3}}+\O{}((|\xi'|_g^2-1)^2).
\end{gather*}
\end{lemma}

\subsubsection{Microlocal description of the boundary layer potentials and operators}
We now recall the microlocal descriptions of the boundary layer potentials and operators near glancing from \cite[Section 4.5]{Galk}.
Let $\mc{A}i$, $\mc{A}i'$, $\mc{A}_-$, and $\mc{A}_-'$ denote the Fourier multiplier with multiplier $Ai(\Theta_h)$, $Ai'(\Theta_h)$, $A_-(\Theta_h)$, and $A_-'(\Theta_h)$ where for convenience, we define
 \begin{gather*} 
 \Theta_h:=h^{-2/3}\Theta,\quad\quad\Theta_{0h}:=h^{-2/3}\Theta_0,\quad\quad \rho_h:=h^{-2/3}\rho, \quad\quad \rho_{0h}:=h^{-2/3}\rho.\end{gather*}
 Next, let 
\begin{align*} 
Jf&:=(2\pi h)^{-d+1}\int f_{0b}e^{\frac{i}{h}(\theta_0+\la x'-y',\xi'\ra)}f(y')dy'd\xi',\\
JCf&:=(2\pi h)^{-d+1}\int f_{0b}(\partial_{x_d}\rho+ih\partial_{x_d}g_1)|_{x_d=0}e^{\frac{i}{h}(\theta_0+\la x'-y',\xi'\ra)}f(y')dy'd\xi',\\
JBf&:=(2\pi h)^{-d+1}\int f_{0b}\partial_{x_d}g_0|_{x_d=0}e^{\frac{i}{h}(\theta_0+\la x'-y',\xi'\ra)}f(y')dy'd\xi'.
\end{align*}
Then $J$ is an elliptic semiclassical Fourier integral operator quantizing the reduction of the Friedlander glancing pair to the glancing pair $\pO$, $S^*\re^d$ and it is not hard to check that $B,\,C\in \Ph{}{}(\partial\Omega)$ so that for any $\delta>0$,
$$\sigma(JCJ^{-1})=(2Q(x,\xi'))^{1/3}+\O{S_\delta}( h^{1-2\delta})$$
where $Q$ is the second fundamental form lifted to the cotangent bundle, $T^*\pO$. Thus $C$ is elliptic.
 
\begin{lemma}
\label{lem:layerNearGlance}
Suppose that $(x_0,\xi_0)\in S^*\pO$ and $\zeta\in \Cc(\re^d)$ have $\zeta\equiv 1$ on $[-1,1]$ with $\supp \zeta\subset [-2,2]$. Then there exists $\delta>0$ such that for any $M,\,\e>0$ if $|\Im z|\leq Mh\log h^{-1}$, 
\begin{align*} GXg&=h^{2/3}\omega^{-1}\chi J\mc{A}i\mc{A}_-C^{-1}J^{-1}Xg+\O{\Ph{-\infty}{}}(h^\infty)g\\
\Dl Xg&=\left(\frac{1}{2}\Id -\omega^{-1}\chi J(\mc{A}i\mc{A}_-'+h^{2/3}\mc{A}_-\mc{A}iC^{-1}B)J^{-1}\right)Xg+\O{\Ph{-\infty}{}}(h^\infty)g\\
(\Sl Xg)|_{\Omega}&=\omega^{-1}h^{2/3}A_{1,g}JC^{-1}J^{-1}Xg+\O{\mc{D}'\to C^\infty}(h^\infty)g\\
(\D Xg)|_{\Omega}&=-\omega^{-1}A_{2,g}Xg-h^{2/3}\omega^{-1}A_{1,g}JC^{-1}BJ^{-1}Xg+\O{\mc{D}'\to C^\infty}(h^\infty)g
\end{align*}
where $\omega=\frac{e^{-\pi i/6}}{2\pi}$,
\begin{gather*} 
\chi:=\zeta((3\delta)^{-1}|x-x_0|),\quad\quad A_{1,g}:=\chi H_1J^{-1},\quad\quad A_{2,g}:=\chi H_2J^{-1}\\
X:=\oph\left[\zeta\left(\delta^{-1}(|x-x_0|+|\xi'-\xi_0|_g)\right)\zeta\left(h^{-\e}\delta^{-1}||\xi'|_g-1|\right)\right].
\end{gather*}
If we only allow $|\Im z|\leq Mh$, then we can set $\e=0$.
\end{lemma}
A simple calculation shows that on $-Mh^{2/3}\leq\xi_1$, $AiA_-(\Theta_h)\in \Ph{-1/2,-1/2}{2/3}(\xi_1=0)$
and on $-Ch^\e\leq \xi_1\leq -Mh^{2/3}$, $AiA_-(\Theta_h)\in h^{-1/4+\e/4}\Ph{-1/2,-1/2}{1-\e/2}(\xi_1=0).$ Moreover, for $\xi_1\geq Mh^{2/3}$, 
$$2\pi e^{\pi i/6}AiA_-(\Theta_h)=\frac{h^{1/3}}{2\sqrt{\xi_1}}(1+\O{}(h(\xi_1)^{-3/2})).$$
So, using \eqref{eqn:xi1}
\begin{equation}
\label{eqn:Gsymb}
\begin{aligned}\sigma(Jh^{2/3}\omega^{-1}\chi \mc{A}i\mc{A}_-C^{-1}J^{-1}X)&=\frac{h}{2\sqrt{\xi_1(\kappa^{-1}(q))}}(1+\O{}(h(\xi_1)^{-3/2}))\frac{1}{\partial_{x_d}\rho(\kappa^{-1}(q))}\\
&\quad\quad\quad \zeta\left(\delta^{-1}(|x-x_0|+|\xi'-\xi_0|_g)\right)\zeta\left(h^{-\e}\delta^{-1}||\xi'|_g-1|\right)\\
&=\frac{h}{2\sqrt{|\xi'|_g^2-1}}(1+\O{}(h(|\xi'|_g^2-1)^{-3/2}))\\
&\quad\quad \quad\zeta\left(\delta^{-1}(|x-x_0|+|\xi'-\xi_0|_g)\right)\zeta\left(h^{-\e}\delta^{-1}||\xi'|_g-1|\right)
\end{aligned} 
\end{equation}

Finally, we recall the decomposition of the boundary layer operators away from glancing from \cite[Lemma 4.27]{Galk}. For a similar decomposition when $\Im z=0$ see \cite[Proposition 4.1]{HaZe}. 
\begin{lemma}
\label{lem:decompose}
Let $\Omega\subset \re^d$ be strictly convex with $\partial\Omega\in C^\infty$. Then for all $1/2>\e,\gamma >0$, and $z=E+\O{}(h^{1-\gamma})$ with $\Im z\geq -Ch\log h^{-1}$. Then\begin{gather*} 
G(z/h):=G_{\Delta}(z)+G_B(z)+G_g(z)+\O{\mc{D}'\to C^\infty}(h^\infty)\\
\Dl(z/h):=\Dl_{\Delta}(z)+\Dl_B(z)+\Dl_g(z)+\O{\mc{D}'\to C^\infty}(h^\infty)\\
\dDl(z/h):=\dDl_{\Delta}(z)+\dDl_B(z)+\dDl_g(z)+\O{\mc{D}'\to C^\infty}(h^\infty)
\end{gather*}
where $G_{\Delta}\in h^{1-\frac{\e}{2}}\Psi_{\e}^{-1}$, $\Dl_{\Delta}\in h^{1-2\e}\Ph{-1}{\e}$, $\dDl_{\Delta}\in h^{-1}\Ph{1}{\e}$, and  $G_B\in h^{1-\frac{\e}{2}}e^{(\Im z)_-d_{\Omega}/h}I^{comp}_{\delta}(C_b)$, $\Dl_B\in e^{(\Im z)_-d_{\Omega}/h}I^{\comp}_{\delta}(C_b)$,  and $\dDl_B\in h^{-1}e^{(\Im z)_-d_{\Omega}/h}I^{comp}_{\delta}(C_b)$ are FIOs associated to $\beta_E$ where $\delta=\max(\e,\gamma)$.
Moreover, 
\begin{gather*}{\MS}'((\cdot)_B)\subset \left\{\begin{gathered}(q,p)\in B_E^*\partial\Omega\times B_E^*\partial\Omega:\\\min(E-|\xi'(q)|_g,\,E-|\xi'(q)|_g\,,\,l(q,p))>ch^\e\end{gathered}\right\}\\
{\MS}'((\cdot)_g)\subset \left\{\begin{gathered}(q,p)\in T^*\partial\Omega\times T^*\partial\Omega:\\\max(|E-|\xi'(q)|_g|,\,|E-|\xi'(p)|_g|, l(q,p))<ch^\e\end{gathered}\right\}
\end{gather*}
\begin{gather*}
\sigma(G_{\Delta})=\frac{ih}{2\sqrt{E^2-|\xi'|_g^2}}\,,\quad \quad \sigma(\dDl_{\Delta})=\frac{ih^{-1}\sqrt{E^2-|\xi'|^2_g}}{2},\\
\sigma(G_Be^{\frac{\Im z}{h}\oph(l(q,\beta_E(q)))})=\frac{he^{\frac{i}{h}\Re\omega_0l(q,\beta_E(q))}}{2(E^2-|\xi'(\beta_E(q))|_g^2)^{1/4}(E^2-|\xi'(q)|_g^2)^{1/4}}dq^{1/2},\\
\sigma(\Dl_Be^{\frac{\Im z}{h}\oph(l(q,\beta_E(q)))})=\frac{-ie^{\frac{i}{h}\Re\omega_0l(q,\beta_E(q))}(E^2-|\xi'(q)|_g^2)^{1/4}}{2(E^2-|\xi'(\beta_E(q))|_g^2)^{1/4}}dq^{1/2},
\end{gather*}
\begin{gather*}
\begin{aligned}\sigma(\dDl_Be^{\frac{\Im z}{h}\oph(l(q,\beta_E(q)))})=&\\
&\!\!\!\!\!\!\!\!\!\!\!\!\!\!\!\!\!\!\!\!\!\!\!\!\!\!\!\!\!\!\!\!\!\!\!\!\!\!\!\!\!\!\!\!\!\!\!\!\!\!\!\!\!\!\!\frac{h^{-1}e^{\frac{i}{h}\Re\omega_0l(q,\beta_E(q))}(E^2-|\xi'(\beta_E(q))|_g^2)^{1/4}(E^2-|\xi'(q)|_g^2)^{1/4}}{2}dq^{1/2}.\end{aligned}
\end{gather*}
where we take $\sqrt{z}=\sqrt{|z|}e^{\frac{1}{2}\Arg(z)}$ for $-\pi/2<\Arg(z)<3\pi /2$.
\end{lemma}
\begin{remark}
The decomposition in \cite{HaZe} is slightly less precise than that in \cite{Galk} because the glancing pieces are microlocalized to a neighborhood of $S^*\partial\Omega\times S^*\partial\Omega$ rather than to a neighborhood of $S^*\partial\Omega\times S^*\partial\Omega\cap \Delta(T^*\partial\Omega)$ where $\Delta(T^*\partial\Omega)$ denotes the diagonal. 
\end{remark}

In particular, Lemma \ref{lem:decompose} together with \eqref{eqn:Gsymb} imply that there exists $M>0$ so that for $\chi=\chi(|\xi'|_g)\in \Ph{0,0}{2/3}(|\xi'|_g=1)$ with $\supp \chi\subset  \{|\xi'|_g\geq 1+Mh^{2/3}\}$, $G\oph(\chi)\in h^{2/3}\Ph{-1/2,-1/2}{2/3}(|\xi'|_g=1)$ with 
\begin{equation}
\label{eqn:Gsymb2}\sigma(G\oph(\chi))=\frac{h\chi(|\xi'|_g)}{2\sqrt{|\xi'|^2_g-1}}\left(1+\O{}(h(|\xi'|^2_g-1)^{-3/2})\right).
\end{equation}

\subsection{Analysis of $\Sl^*\Sl$, $\D^*\D$, and $\D^*\Sl$ near glancing}
Our next goal is to understand $\Sl^*\Sl$, $\D^*\D$, and $\D^*\Sl$ microlocally near glancing points. To do this, we will use the microlocal description of $\Sl$ and $\D$ from Lemma \ref{lem:layerNearGlance}. In particular, let $J_1$ be a microlocally unitary FIO quantizing $\kappa$ where $\kappa$ is as in \eqref{eqn:kappa}. Then we prove
\begin{lemma}
\label{lem:potentialGlance}
Fix $z=1+i\mu$ with $|\mu|\leq M h\log h^{-1}$. Then for any $\e>0$ and $\delta\leq 2/3$, for $\chi\in \Psi^{0,0}_{2/3}(|\xi'|_g=1)$ self adjoint with $\WFh(\chi) \subset \{||\xi'|_g -1|\leq h^\delta\}$, 
\begin{gather*} 
\chi\S^*\S\chi \in h^{2-\e}\Ph{0,0}{1-\delta/2}(\{|\xi'|_g=1\}),\quad \chi \D^*\S\chi ,\chi\S^*\D\chi \in h^{\frac{3}{2}-\frac{\delta}{4}-\e}\Ph{0,1/2}{1-\delta/2}(\{|\xi'|_g=1\}),\\
 \chi\D^*\D\chi\in h^{1-\frac{\delta}{2}-\e}\Ph{0,1}{1-\delta/2}(\{|\xi'|_g=1\}) 
\end{gather*}
Moreover, 
\begin{align*}
\sigma(J_1^*\chi\S ^*\S\chi J_1)&=\frac{h^2\Psi_{\S }(h^{-2/3}\Theta_0(\xi'))\chi^2(\kappa(x',\xi'))}{2Q(\kappa(x',\xi'))},\\
\sigma(J_1^*\chi\D ^*\S\chi J_1)&=\frac{h^{4/3}\Psi_{\D \S }(h^{-2/3}\Theta_0(\xi'))\chi^2(\kappa(x',\xi'))}{(2Q(\kappa(x',\xi')))^{2/3}},\\
\sigma(J_1^*\chi\S ^*\D\chi J_1)&=\frac{h^{4/3}\overline{\Psi_{\D \S }(h^{-2/3}\Theta_0(\xi'))}\chi^2(\kappa(x',\xi'))}{(2Q(\kappa(x',\xi')))^{2/3}},\\
\sigma(J_1^*\chi\D ^*\D\chi J_1)&=\frac{h^{2/3}\Psi_{\D }(h^{-2/3}\Theta_0(\xi'))\chi^2(\kappa(x',\xi'))}{(2Q(\kappa(x',\xi')))^{1/3}}.
\end{align*}
\end{lemma}

We prove this lemma using Lemma \ref{lem:layerNearGlance} to write a parametrix for $\S^*\S$. We then Taylor expand the Airy functions around their values at the boundary of $\Omega$ and estimate each of the terms. The higher order terms in the expansion will turn out to be lower order in $h$ and the symbols will be found by computing the first term. The operators $\D^*\S$, $\S^*\D$ and $\D^*\D$ are handled similarly.

\subsubsection{Estimates on the Remainder Terms}
We first give estimates on the size of terms that will be lower order. These terms arise from a Taylor expansion of the integrand when computing $\S^*\S$ using the mcirolocal model from Lemma \ref{lem:layerNearGlance}. In particular, consider an operator with kernel given by
\begin{align*} 
R_{ijklmno}&=(2\pi h)^{-2d+2}\iint_0^\infty b(w,x',y',\eta,\xi')h^{-2/3(j+k)}(\rho(w,\eta')-\rho(w,\xi'))^k(\Theta(\eta')-\Theta(\xi'))^jw_d^n\\
&\quad\quad Ai^{(l)}(\rho_h(w,\xi'))A_-^{(m)}(\Theta_h(\xi'))\overline{Ai^{(o)}(\rho_h(w,\xi'))A_-^{(i)}(\Theta_h(\xi'))}\\
&\quad\quad\quad e^{\frac{i}{h}(\theta(w,\xi')-\theta(w,\eta')-\theta_b(y',\xi')+\theta_b(x',\eta'))}dw_dd\xi'd\eta'dw'
\end{align*}
where $b\in S_\delta(\xi_1=0)$ is supported in $|\Theta(\xi')|,|\Theta(\eta')|\leq Ch^\delta$. First, observe that since $\partial_{x_d}\rho_0>0$ and for $t\gg 1$, 
$$Ai(t)\leq Ce^{-t^{3/2}},$$
 we may assume that $b$ is supported on $w_d<\e$ for any $\e>0$ by introducing an $\O{}(e^{-C/h})$ error. Next, notice that 
$$\theta(w,\xi')-\theta(w,\eta')=\theta_b(w',\xi')-\theta_b(w',\eta')+w_d^2\la \xi'-\eta', r(w,\xi',\eta')\ra.$$
So, 
$$\partial_{w'}\theta(w,\xi')-\partial_{w'}\theta(w,\eta')=(\partial^2_{x'\xi'}\theta_b(w',\eta')+w_d^2\partial_{w'}r)(\xi'-\eta')$$
and, using that $\partial^2_{x'\xi'}\theta_b\neq 0$, for $w_d$ small enough, the phase is stationary precisely at $\xi'=\eta'$. 

 We first change variables so that $W_d=h^{-2/3}\rho_0(w,\xi')$. Then, $w_d=h^{2/3}e(W_d,w',\xi')(W_d-h^{-2/3}\Theta_0(\xi'))$ where $e$ is elliptic. So, the kernel takes the form
\begin{align*} 
&R_{ijklmno}\\
&=(2\pi h)^{-2d+2}h^{2/3}\iint_{h^{-2/3}\Theta_0(\xi')}^\infty b_1(h^{2/3}W_d,w',x',y',\eta,\xi')h^{-2/3(j+k-n)}\\
&\quad\quad (\rho(w_d(W_d,w',\xi'),\eta')-h^{2/3}W_d-\e(h)\rho_1)^k(\Theta(\eta')-\Theta(\xi'))^j(W_d-h^{-2/3}\Theta_0(\xi'))^n\\
&\quad\quad Ai^{(l)}(W_d+h^{-2/3}\e(h)\rho_1(w,\xi'))A_-^{(m)}(\Theta_h(\xi'))\overline{Ai^{(o)}(W_d+h^{-2/3}\e(h)\rho_1(w,\xi'))A_-^{(i)}(\Theta_h(\xi'))}\\
&\quad\quad\quad e^{\frac{i}{h}(\theta_b(w',\xi')-\theta_b(w',\eta')-\theta_b(y',\xi')+\theta_b(x',\eta')+h^{4/3}(W_d-\Theta_0(\xi'))^2\la \xi'-\eta',r(W_d,w',x',\xi',\eta')\ra )}dW_dd\xi'd\eta'dw'
\end{align*}
Now, the integrand vanishes to order $|\xi'-\eta'|^{j+k}$, and the phase is stationary in $w'$ precisely at $\xi'=\eta'$. Hence, integrating by parts $j+k$ times in $w'$ and then applying stationary phase in the $w'$, $\eta'$ variables gives a finite sum of terms (possibly with additional positive powers of $h$) of the form
\begin{align*} 
&\frac{h^{2/3}}{(2\pi h)^{d-1}}\iint_{h^{-2/3}\Theta_0(\xi')}^\infty b_2(h^{2/3}W_d,x'+\O{}((W_d-h^{-2/3}\Theta_0(\xi'))^2h^{4/3}),x',y',\xi',\xi')\\
&\quad\quad h^{1/3(j+k+2(n-p-q))}\e(h)^{p+q}(W_d-h^{-2/3}\Theta_0(\xi'))^n Ai^{(l+p)}(W_d+h^{-2/3}\e(h)\rho_1)A_-^{(m)}(\Theta_h(\xi'))\\
&\quad\quad \overline{Ai^{(o+q)}(W_d+h^{-2/3}\e(h)\rho_1)A_-^{(i)}(\Theta_h(\xi'))}e^{\frac{i}{h}(\theta_b(x',\xi')-\theta_b(y',\xi'))}dW_dd\xi'.
\end{align*}
Note that we can apply stationary phase in the $w',\eta'$ variables since $\partial^2_{x'\xi'}\theta_{0b}\neq 0$. Next, change variables $\xi'\mapsto \Xi'(x',y',\xi')$ so that 
$$\theta_{0b}(x',\xi')-\theta_{0b}(y',\eta')=\la x'-y',\Xi'(x',y',\xi')\ra.$$
To  find such a change of variables, observe that 
$\Xi(x',x',\xi')=\partial_{x'}\theta_{0b}$ and hence $\partial_{\xi'}\Xi=\partial^2_{\xi'x'}\theta_{0b}\neq0$ so we can apply the implicit function theorem.
Then, integrating in $W_d$ and using the fact that on $\supp b_2$, $|\Theta(\xi')|\leq Ch^{\delta}$, we obtain 
$$R_{ijklmno}=(2\pi h)^{-d+1}\int b_3(x',y',\xi';h)e^{\frac{i}{h}\la x'-y',\xi'\ra }d\xi'$$
where, letting
$$r=\frac{2n+l+p+o+q+m+i-4}{2}+\frac{1}{4}(\delta_{o+q}^0+\delta_{l+p}^0+\delta_i^0+\delta_m^0),$$
$$b_3\in h^{2/3+\frac{1}{3}(j+k+p+q+2n)}(\log h^{-1})^{p+q}h^{\max(r,0)(\frac{1}{3}-\frac{\delta}{2})}S_{1-\delta/2}^{0,r}(\re^{d-1};\{\xi_1=0\})$$

Hence, the operator $R^{ijklmno}$ with kernel $R_{ijklmno}$ has for any $\e>0$,
\begin{gather*} 
R^{ijklmno}\in h^{1/3(j+k-l-m-i-o+2)+\delta/2(l+m+i+2n+o)-\e}\Ph{0,0}{1-\delta/2}(\{\xi_1=0\})\\
R^{ijklmno}\in h^{2/3+\frac{1}{3}(j+k+2n)}h^{\max(r,0)(\frac{1}{3}-\frac{\delta}{2})}\Psi_{1-\delta/2}^{0,r}(\re^{d-1};\{\xi_1=0\}).
\end{gather*}

\subsubsection{The Principal Part}
By the analysis above, we see that when microlocalized near glancing points $\S ^*\S $, $\D ^*\D $, and $\D ^*\S $ are pseudodifferential in a second microlocal class. We just need to compute the principal symbol of these operators. The symbols will turn out to be $\Psi_{\S }$, $\Psi_{\D }$, and $\Psi_{\D \S },$ respectively. 

First, using the principle of stationary phase, we compute that  
\begin{equation*} 
\begin{gathered} 
J^{-1}f=(2\pi h)^{-d+1}\int b_0(y',\xi')e^{\frac{i}{h}(\la x',\xi'\ra-\theta_b(y',\xi'))}f(y')dy'd\xi'\,,\\
C^{-1}J^{-1}f=(2\pi h)^{-d+1}\int b_1(y',\xi')e^{\frac{i}{h}(\la x',\xi'\ra-\theta_b(y',\xi'))}f(y')dy'd\xi'.
\end{gathered}
\end{equation*} 
where 
$$b_0=\frac{|\det\partial^2_{x'\xi'}\theta_b(y',\xi')|}{g_{0b}(y',\xi')}+\O{S}(h)\quad \text{ and }\quad b_1=\frac{b_0(y',\xi')}{\partial_{x_d}\rho(y',\xi')}+\O{S}(h).$$

Denote the kernels of $\S ^*\S $, $\D ^*\D $, and $\D ^*\S $, respectively by $K_{\S }$, $K_{\D }$, and $K_{\D \S }$ respectively. 
We explicitly consider $\S ^*\S $ and we record the end result for the others. The kernel of $\S $ is given by 
\begin{align*} 
\S (x,y)&=\frac{2\pi e^{\pi i/6}h^{2/3}}{(2\pi h)^{d-1}}\int \left(g_0(x,\xi')Ai(\rho_h(x,\xi'))+ih^{1/3}g_1(x,\xi')Ai'(\rho_h(x,\xi'))\right)\\
&\quad\quad A_-(\Theta_h(\xi'))b_1(y',\xi')e^{\frac{i}{h}(\theta(x,\xi')-\theta_b(y',\xi'))}d\xi'
\end{align*}
The kernel of $\S^*\S$ is given by
\begin{align*}
K_{\S }&= \frac{4\pi ^2h^{4/3}}{(2\pi h)^{2d-2}}\iint_0^\infty\left(g_0(w,\xi')Ai(\rho_h(w,\xi'))+ih^{1/3}g_1(w,\xi')Ai'(\rho_h(w,\xi'))\right)\\
&\quad \quad \left(\overline{g_0(w,\eta')Ai(\rho_h(w,\eta'))}-ih^{1/3}\overline{g_1(w,\eta')Ai'(\rho_h(w,\eta'))}\right)\\
&\quad\quad A_-(\Theta_h(\xi'))\overline{A_-(\Theta_h(\eta'))}b_1(y',\xi')\overline{b_1(x',\eta')}e^{\frac{i}{h}(\theta_b(x',\eta')-\theta(w,\eta')+\theta(w,\xi')-\theta_b(y',\xi'))}dw_ddw'd\xi'd\eta'
\end{align*}  
Taylor expanding the Airy functions around $\rho_h(w,\xi')$ and $\Theta_h(\xi')$ produces lower order terms of the form $h^{4/3}R_{0jkjk0}$, $(j,k)\neq (0,0)$, $h^{5/3}R_{0jk(j+1)k10}$, $h^{5/3}R_{0jkjk11}$ and $h^2R_{0jk(j+1)k21}$.  In particular, $\S ^*\S =A+\O{\Ph{}{1-\delta/2}(\{\xi_1=0\})}(h^{2+\delta/2-\e})$ where $A$ has kernel
\begin{align*}
A(x,y)&= \frac{4\pi ^2h^{4/3}}{(2\pi h)^{2d-2}}\iint_0^\infty g_0(w,\xi')Ai(\rho_h(w,\xi'))\overline{g_0(w,\eta')Ai(\rho_h(w,\xi'))}A_-(\Theta_h(\xi'))\overline{A_-(\Theta_h(\xi'))}\\
&\quad\quad b_1(y',\xi')\overline{b_1(x',\eta')}e^{\frac{i}{h}(\theta_b(x',\eta')-\theta(w,\eta')+\theta(w,\xi')-\theta_b(y',\xi'))}dw_ddw'd\xi'd\eta'
\end{align*} 
Then, changing variables $W_d\mapsto h^{-2/3}\rho_0(w,\xi')$ and performing stationary phase as in the analysis of $R_{jklmno}$ gives
\begin{align*}
A(x,y)&= \frac{4\pi ^2h^{2}}{(2\pi h)^{d-1}}\iint_{h^{-2/3}\Theta_0(\xi')}^\infty \frac{a_0(x',\xi')\overline{a_0(x',\xi')}+\O{S_{1-\delta/2}}(h^{1/3+\delta/2-\e})}{|\det\partial^2_{x'\xi'}\theta(x'\xi')|\partial_{x_d}\rho(x',\xi')}|Ai(W_d+h^{-2/3}\e(h)\rho_1)|^2\\
&\quad\quad |A_-(\Theta_h(\xi'))|^2b_1(y',\xi')\overline{b_1(x',\xi')}e^{\frac{i}{h}(\theta_b(x',\xi')-\theta_b(y',\xi'))}dW_ddw'd\xi'd\eta'
\end{align*} 
Using that the phase is stationary at $x'=y'$ to integrate by parts in $\xi'$ when terms of size $|x'-y'|$ appear, that for any $\e>0$,
\begin{equation*}Ai(W_d+h^{-2/3}\e(h)\rho_1)=\begin{cases} Ai(W_d)+\O{S}(h^{-2/3}\e(h)\la W_d\ra^{1/4})&W_d\leq C\\
Ai(W_d)+\O{S}(h^{-2/3}\e(h)\la W_d\ra^{1/4}e^{-2/3W_d^{3/2}})&W_d\geq C,\end{cases}\end{equation*}
 and using the definition of $\Psi_{\S }$ gives for any $\e>0$,
\begin{align*}
A(x,y)&= \frac{h^{2}}{(2\pi h)^{d-1}}\int \frac{|a_0(x',\xi')|^2|b_1(x',\xi')|^2\Psi_{\S }(\Theta_{0h}(\xi'))+\O{S_{1-\delta/2}}(h^{\delta/2-\e})}{|\det\partial^2_{x'\xi'}\theta_{0b}(x',\xi')|\partial_{x_d}\rho(x',\xi')}\\
&\quad\quad e^{\frac{i}{h}(\theta_b(x',\xi')-\theta_b(y',\xi'))}d\xi'
\end{align*} 
Now, let $J_1$ be a microlocally unitary semiclassical FIO quantizing $\kappa$ i.e. 
$$J_1f=(2\pi h)^{-d+1}\int c(x',\xi')e^{\frac{i}{h}(\theta_{0b}(x',\xi')-\la y',\xi'\ra)}d\xi'$$
where 
$$c=|\det \partial^2_{x'\xi'}\theta_{0b}(x',\xi')|^{1/2}+\O{}(h).$$
Applying stationary phase gives
\begin{align*}
&J_1^*AJ_1(x,y)\\
&= \frac{h^{2}}{(2\pi h)^{d-1}}\int \left.\frac{\bar{c}(w',\xi')c(z',\xi')|a_0(w',\xi')|^2|b_1(w',\xi')|^2\Psi_{\S }(\Theta_{0h}(\xi'))+\O{S_{1-\delta/2}}(h^{\delta/2-\e})}{|\det\partial^2_{x'\xi'}\theta_{0b}(w',\xi')|^2\partial_{x_d}\rho(w',\xi')|\det\partial^2_{x'\xi'}\theta_{0b}(z',\xi')|}\right|_{\substack{y'=\partial_{\xi'}\theta_{0b}(z',\xi')\\
x'=\partial_{\xi'}\theta_{0b}(w',\xi')}}\\
&\quad\quad e^{\frac{i}{h}(\la x'-y',\xi'\ra+\theta_{1b}(w',\xi')-\theta_{1b}(z',\xi'))}d\xi'
\end{align*} 
Again, using integration by parts on terms that are $\O{}(|x'-y'|)$, we can assume that $x'=y'$ in the amplitude and hence have 
\begin{align*}
&J_1^*AJ_1(x,y)\\
&= \frac{h^{2}}{(2\pi h)^{d-1}}\int \left.\frac{\bar{c}(w',\xi')c(w',\xi')|a_0(w',\xi')|^2|b_1(w',\xi')|^2\Psi_{\S }(\Theta_{0h}(\xi'))+\O{S_{1-\delta/2}}(h^{\delta/2-\e})}{|\det\partial^2_{x'\xi'}\theta_{0b}(w',\xi')|^2\partial_{x_d}\rho(w',\xi')|\det\partial^2_{x'\xi'}\theta_{0b}(w',\xi')|}\right|_{
x'=\partial_{\xi'}\theta_{0b}(w',\xi')}\\
&\quad\quad e^{\frac{i}{h}\la x'-y',\xi'\ra}d\xi'
\end{align*} 
So, plugging in the definition of $c$ and $b_1$, we have 
\begin{align}
\label{layer-e:singleLayer}
J_1^*\S ^*\S J_1&= \frac{h^{2}}{(2\pi h)^{d-1}}\int \left.\frac{\Psi_{\S }(\Theta_{0h}(\xi'))+\O{S_{1-\delta/2}}(h^{\delta/2-\e})}{|\partial_{x_d}\rho(w',\xi')|^2\partial_{x_d}\rho(w',\xi')}\right|_{
x'=\partial_{\xi'}\theta_{0b}(w',\xi')}e^{\frac{i}{h}\la x'-y',\xi'\ra}d\xi'
\end{align} 
Similar computations give 
\begin{equation}
\label{layer-e:layerFormulae}
\begin{aligned} 
J_1^*\D ^*\S J_1&=\frac{h^{4/3}}{(2\pi h)^{d-1}}\int \left.\frac{\Psi_{\D \S }(\Theta_{0h}(\xi'))+\O{S^{0,1/2}_{1-\delta/2}}(h^{\frac{1}{6}+\delta/4-\e})}{(\partial_{x_d}\rho(w',\xi'))^2}\right|_{
x'=\partial_{\xi'}\theta_{0b}(w',\xi')}e^{\frac{i}{h}\la x'-y',\xi'\ra}d\xi'\\
J_1^*\S ^*\D J_1&=\frac{h^{4/3}}{(2\pi h)^{d-1}}\int \left.\frac{\overline{\Psi_{\D \S }(\Theta_{0h}(\xi'))}+\O{S^{0,1/2}_{1-\delta/2}}(h^{\frac{1}{6}+\delta/4-\e})}{(\overline{\partial_{x_d}\rho(w',\xi'))^2}}\right|_{
x'=\partial_{\xi'}\theta_{0b}(w',\xi')}e^{\frac{i}{h}\la x'-y',\xi'\ra}d\xi'\\
J_1^*\D ^*\D J_1&=\frac{h^{2/3}}{(2\pi h)^{d-1}}\int \left.\frac{\Psi_{\D }(\Theta_{0h}(\xi'))+\O{S^{0,1}_{1-\delta/2}}(h^{\frac{1}{3}-\e})}{\partial_{x_d}\rho(w',\xi')}\right|_{
x'=\partial_{\xi'}\theta_{0b}(w',\xi')}e^{\frac{i}{h}\la x'-y',\xi'\ra}d\xi'.
\end{aligned}
\end{equation} 
Hence, all of the above operators are second microlocal pseudodifferential operators with respect to the glancing surface $\{|\xi'|_g=1\}.$ 

Plugging \eqref{eqn:xi1} into \eqref{layer-e:singleLayer} and \eqref{layer-e:layerFormulae} gives that 
\begin{align*}
\sigma(J_1^*\S ^*\S J_1)&=\frac{h^2\Psi_{\S }(h^{-2/3}\Theta_0(\xi'))}{2Q(\kappa(x',\xi'))}&
\sigma(J_1^*\D ^*\S J_1)&=\frac{h^{4/3}\Psi_{\D \S }(h^{-2/3}\Theta_0(\xi'))}{(2Q(\kappa(x',\xi')))^{2/3}}\\
\sigma(J_1^*\S ^*\D J_1)&=\frac{h^{4/3}\overline{\Psi_{\D \S }(h^{-2/3}\Theta_0(\xi'))}}{(2Q(\kappa(x',\xi')))^{2/3}}&
\sigma(J_1^*\D ^*\D J_1)&=\frac{h^{2/3}\Psi_{\D }(h^{-2/3}\Theta_0(\xi'))}{(2Q(\kappa(x',\xi')))^{1/3}}
\end{align*}
where $\kappa$ is as in \eqref{eqn:kappa}.

\section{Preliminary analysis of the generalized boundary damped equation}
\label{sec:genBStable}
We examine problems of the form 
\begin{gather}
\label{p1-e:main1}
\begin{cases}(-h^2\Delta -z^2)u=w&\text{ in } \Omega\\
h\partial_\nu u +Bu =hv&\text{ on }\partial\Omega\\
u|_{\pO}=\psi
\end{cases}\\
 \label{eqn:zrange} 
 z\in [1-ch,1+ch]+i[-Mh\log h^{-1},Mh\log h^{-1}].
 \end{gather}

 We then assume that
 $B=hN_2(z/h)+hV(z)$, with $V$ analytic for $z$ as in \eqref{eqn:zrange},
$V\in
h^{\alpha}(\Ph{0,m}{2/3}\{|\xi'|_g=E'\}\cup \Ph{0,m}{2/3}\{|\xi'|_g=1\})
$ for some $\alpha\geq -1$ and $m\in \re$. 
 
 Furthermore, suppose that for some $\delta>0$, $M,M_1>0$, and $0<\e<\frac{1}{2}$ 
\begin{equation} 
\label{resFree-e:ellipticAssumption} 
\begin{gathered} V\text{ is elliptic},\text{ on }\left||\xi'|_g-1\right|<\delta,\\
\begin{aligned}
\left|1+\frac{h\sigma(V)}{2\sqrt{|\xi'|_g^2-1}}\right|&\geq \delta \left(\left\la \frac{h^{1+\alpha}}{\sqrt{|\xi'|_g^2-1}}\right\ra+\la \xi'\ra^{m-1}\right)\quad\quad&&|\xi'|_g>1+Mh^{2/3}\\
\left|1+\frac{ih\sigma(V)}{2\sqrt{1-|\xi'|_g^2}}\right|&\geq \delta \left\la \frac{h^{1+\alpha}}{\sqrt{1-|\xi'|_g^2}}\right\ra\quad\quad&&|\xi'|_g\leq 1-h^{\e}
\end{aligned}\\
\log \left(1+\frac{h\sigma(V)}{\sqrt{|\xi'|_g^2-1}}\right)\text{ exists and is smooth on }T^*\pO\setminus\{|\xi'|_g\leq M_1\}
\end{gathered}
\end{equation}

The problem \eqref{p1-e:main1} is a highly generalized version of a standard boundary damped equation which was studied in the seminal work of Bardos--Lebeau--Rauch \cite{BardLebeau} see also \cite{KochTataruStable}.  In order to study this problem from the spectral point of view, we must see that the inverse operator is meromorphic with finite rank poles. This is similar to the analysis in the case of the standard damped wave equation (see for example \cite[Chapter 5]{EZB} and references therein).

 \subsection{Meromorphy of the Resolvent}
 For $s>-1/2$, let 
 $$\mc{P}(z):=\begin{pmatrix}-h^2\Delta -z^2\\\gamma \partial_\nu+h^{-1}B(z)\gamma  \end{pmatrix}:H^{s+2}(\Omega)\to H^s(\Omega)\oplus H^{s+1/2-\max(m-1,0)}(\pO).$$ 
 We will show that $\mc{P}(z)^{-1}$ is a meromorphic family of operators with finite rank poles. Our analysis is similar in spirit to that for potential and black box scattering see for example \cite[Chapters 2,3,4]{ZwScat}.

 Then, when $(I+VG)^{-1}:H^s(\pO)\to H^{s+\max(m-1,0)}(\pO)$ exists,
 \begin{align*}(\mc{P}^{-1})^t&=\begin{pmatrix}[I-\Sl(I+VG)^{-1}(\gamma \partial_\nu +h^{-1}B\gamma)]h^{-2}1_{\Omega}R_0(z/h)1_{\Omega}\\
 \Sl(I+VG)^{-1}\end{pmatrix}\\
 &:H^{s}(\Omega)\oplus H^{s+1/2-\max(m-1,0)}(\pO)\to H^{s+2}(\Omega).
 \end{align*}
 To check that this is the inverse, we simply apply the jumps formulas from for example \cite[Lemma 4.1 and Proposition 4.1.1]{Galk}. For the Sobolev mapping properties of $1_{\Omega}R_01_{\Omega}$, $\Sl$, $\D$, see for example \cite[Theorems 9, 10]{Epstein}.
  Now, 
$$(I+VG)^{-1}=I-V(I+GV)^{-1}G,\quad\quad (I+GV)^{-1}=I-G(I+VG)^{-1}V$$
therefore, $I+GV$ is invertible if and only if $I+VG$ is invertible.
 Thus, to check that $\mc{P}^{-1}$ has a meromorphic continuation from $\Im z>0$, it is enough to check that for $(I+GV)^{-1}$. To see this, we first show that $I+GV$ is a holomorphic family of Fredholm operators with index 0 on the domain of $R_0$. The condition \eqref{resFree-e:ellipticAssumption} and Lemma \ref{lem:decompose} imply that for $M$ sufficiently large and $0\leq \chi_0\in \Cc(\re)$ with $\chi_0\equiv 1$ on $|x|\leq M$ and $\supp \chi_0\subset \{|x|\leq M+1\},$ $(I+GV)(1-\chi(|hD'|_g))\in \Ph{\max(m-1),0)}{}(\pO)$ is elliptic on $|\xi'|_g\geq M+1$ with symbol
 $$f:=\sigma((I+GV)(1-\chi_0(|hD'|_g)))=\left(1+\frac{h\sigma(V)}{2\sqrt{|\xi'|_g^2-1}}\right)(1-\chi_0(|\xi'|_g)).$$ Then, for $k=1,2$, let $0\leq \chi_k\in \Cc(\re)$ with $\chi_k\equiv 1 $ on $|x|\leq M+1$ and $\supp \chi_k\subset \supp\chi_{k+1}$ with $\supp \chi_2\subset \{|x|\leq M+2\}$. 
Then, by assumption, $\log \frac{f}{|f|}$ is well defined  on $\supp \chi_2(|\xi'|_g)$ and hence for $K>0$ large enough
 $$q=f+K\chi_2(|\xi'|_g)\left(\frac{f}{|f|}\right)^{1-\chi_1(|\xi'|_g)}\in S^{m-1}\quad\quad\text{ has }|q|\geq c  \la \xi'\ra^{m-1}.$$
 
Now, $\oph(q):\Hh^{s+\max(m-1,0)}(\pO)\to H^s_h(\pO)$ is invertible for $h$ small enough and 
\begin{align*} 
\oph(q)(I+GV)&=I+A_1&(I+GV)\oph(q)&=I+A_2\\
\oph(q)(I+VG)&=I+A_3&(I+VG)\oph(q)&=I+A_4
\end{align*}
with $A_i:\Hh^s(\pO)\to \Hh^{s-1}(\pO)$.
Therefore, both $I+GV$ and $I+VG$ are Fredholm with index 0. The analysis below will show that there exists $z_0$ with $\Im z>0$ so that $I+GV$ is injective. Therefore, $(I+GV)^{-1}$ exists at $z_0$ and by the analytic Fredholm Theorem has a meromorphic continuation to $\mathbb{C}$ when $d$ is odd and to the logarithmic cover of $\mathbb{C}\setminus\{0\}$ when $d$ is even. 
 
Write
\begin{equation} 
\label{p1-e:boundaryMain1} 
(I+VG)\varphi=v.
\end{equation}
 Note that if $\varphi$ has \eqref{p1-e:boundaryMain1}, then $u=\Sl\varphi$ 
solves \eqref{p1-e:main} with $w=0$ and $\psi=G\varphi$. That is, 

\begin{equation}
\label{p1-e:main}
\begin{cases}(-h^2\Delta -z^2)u=0&\text{ in } \Omega\\
h\partial_\nu u +Bu =hv&\text{ on }\partial\Omega\\
u|_{\pO}=\psi
\end{cases}
\end{equation}
Similarly, if 
\begin{equation}\label{p1-e:boundaryMain} (I+GV)\psi=Gv,\end{equation}
then 
$$u=-\Sl V\psi+\Sl v$$
solves \eqref{p1-e:main}. 
Now, suppose that $u$ solves \eqref{p1-e:main}. Then
$$u=h^{-1}\Sl h\partial_\nu u-\D u|_{\pO}=-h^{-1}\Sl Bu|_{\pO}-\D u|_{\pO}+\Sl v=-\Sl Vu|_{\pO}+\Sl v$$
where we have used that in $\Omega$,
$$\Sl N_2+\D =0\quad\quad\text{ and hence } \quad\quad (h^{-1}\Sl B+\D)=\Sl V.$$
Therefore, taking $x\to \pO$ gives
$$u|_{\pO}=-GVu|_{\pO}+Gv\quad\quad\imply\quad  (I+GV)u|_{\pO}=Gv.$$
That is, $\psi:=u|_{\pO}$ solves \eqref{p1-e:boundaryMain}. Finally, if $\psi$ solves \eqref{p1-e:boundaryMain}, then $\varphi:=v-V\psi$ solves \eqref{p1-e:boundaryMain1}.
\begin{lemma}
\label{lem:equivalence}
The following are equivalent
\begin{enumerate}
\item $u$ solves \eqref{p1-e:main} 
\item $u=\Sl (v-Vu|_{\pO})$
\item $u|_{\pO}=\psi$ solves \eqref{p1-e:boundaryMain}
\item $v-Vu|_{\pO}=\varphi$ solves \eqref{p1-e:boundaryMain1}.
\end{enumerate}
\end{lemma}

Note also that since $I+VG$ is Fredholm with index 0, it is not invertible if and only if there exists a nonzero solution $\psi$ to $(I+VG)\psi=0.$ 
Hence, together with Lemma \ref{lem:equivalence} we have proved the following 
\begin{lemma}
\label{lem:meromorphy}
The operator $\mc{P}^{-1}$ is meromorphic on the domain of $R_0(\lambda)$ and the following are equivalent
\begin{enumerate}
\item $\mc{P}^{-1}(z)$ has a pole at $z_0$.
\item There exists a nonzero solution $\psi$ to $(I+G(z_0)V(z_0)\psi=0.$
\item There exists a nonzero solution $\varphi$ to $(I+V(z_0)G(z_0)\varphi=0.$ 
\item There exists a nonzero solution $u$ to \eqref{p1-e:main} with $v=0$.
\end{enumerate}
\end{lemma}

\section{Microlocal analysis of the generalized boundary damped wave equation}
 We now proceed to study the poles of $\mc{P}(z)^{-1}$. It is convenient to study \eqref{p1-e:boundaryMain} because then the solution to \eqref{p1-e:main} has $u|_{\pO}=\psi$. From now on, we do so without comment.
 
 \subsection{Brief outline of the computations}
 
 The analysis in the next few sections proceeds as follows. We first study the elliptic region where there is no propagation and hence the analysis is relatively simple. Then, we study the hyperbolic region where standard propagation occurs. In this case, we use the decomposition of $G$ (Lemma \ref{lem:decompose}) to rewrite \eqref{p1-e:boundaryMain} in terms of the reflectivity operator, $R$ from \eqref{eqn:reflect} and transition operator $T$ from \eqref{eqn:transition}. We use the symbolic calculus of FIO's to show that this new operator has a microlocal inverse on the hyperbolic set. However, we must show that this inverse preserves the hyperbolic set up to a small remainder. This is done in Lemma \ref{p1-l:parametrixMS}. 
 
 Putting these two regions together leaves the glancing region to be analyzed. Here, we apply the microlocal models of $G$ and $\S$ near glancing from Lemmas \ref{lem:layerNearGlance} and \ref{lem:potentialGlance}. We start by using \eqref{p1-e:boundaryMain} together with the model for $G$ near glancing to further localize $\psi$ near certain 'almost glancing hypersurfaces'. Using that $\S V\psi$ solves \eqref{p1-e:main} with $v=0$, we obtain estimates on $\Im z$ from the description of $\Sl^*\Sl$ near glancing.
 
\subsection{Elliptic Region}
\label{sec:elliptic}
Fix $0<\e<\frac{1}{2}$ and $0<c_1<c_2<c$. We first estimate solutions to \eqref{p1-e:boundaryMain} in the elliptic region $\mc{E}:=\{|\xi'|_g\geq 1+ ch^{\e}\}$.

Let $\chi_1\in S_{\e}(|\xi'|_g=1)$ have $\chi_1 \equiv 1 $ on $|\xi'|_g\geq \{1+c_2h^{\e}\}$ and $\supp \chi_1\subset \{|\xi'|_g\geq 1+c_1h^{\e}\}$. Also, let $\chi_2\in S_{\e}(|\xi'|_g=1)$ have $\supp \chi_2\subset \{|\xi'|_g\geq 1+c_2 h^{\e}\}$ and $ \chi_2\equiv 1$ on $\{|\xi'|_g\geq 1+c h^{\e}\}$. Let $X_1=\oph(\chi_1)$ and $X_2=\oph(\chi_2).$

Let $\psi$ solve \eqref{p1-e:boundaryMain}. Then, we have 
\m (I+GV)X_1\psi= [GV,X_1]\psi+X_1Gv.\,\,\m 
Now, by Lemma \ref{lem:decompose}, $GV X_1=G_{\Delta}VX_1+\O{\Ph{-\infty}{}}(h^\infty)$ where $G_{\Delta}\in h^{2/3}\Ph{-1/2,-1}{2/3}(|\xi'|_g=1).$ 
 By our assumptions on $V$ and Lemma \ref{semi-l:elliptic}, there exists 
$$A\in h^{\max(-2/3-\alpha,0)}\Ph{1/2,\min(0,1-m)}{2/3}(|\xi'|_g=1)\cup \Ph{0,\min(0,1-m)}{2/3}(|\xi'|_g=E')$$
 so that 
$A(I+G_\Delta V)=X_2$ and $\MS(A)\subset \{\chi_1\equiv 1\}$. 
So, 
$$X_2\psi=A[G_{\Delta}V,X_1]\psi+AX_1Gv+\O{\Ph{-\infty}{}}(h^\infty)(\psi+v)$$
and hence, 
\begin{align*} 
\|X_2\psi\|_{\Hh^m}&\leq C(\|A[G_{\Delta}V,X_1]\psi\|_{L^2}+\|AX_1G_\Delta v\|_{\Hh^m}+\O{}(h^\infty)(\|\psi\|_{\Hh^{-N}}+\|v\|_{\Hh^{-N}})\\
&\leq C(h^{1-\e/2}\|v\|_{L^2}+\O{}(h^\infty)\|\psi\|).
\end{align*}

Summarizing,
\begin{lemma}
\label{p1-l:ellipticEstimate}
For all $0<\e<1/2$, $c>0$, and $N>0$, there exists $h_0=h_0(\e,c)>0$ such that for $0<h<h_0$, $\chi\in S^{0,0}_{\e}(|\xi'|_g=1)$ with $\supp \chi\subset \{|\xi'|_g\geq 1+c h^{\e}\}$, and $\psi$ solving \eqref{p1-e:boundaryMain}
$$\|\oph(\chi)\psi\|_{\Hh^m}\leq C(h^{1-\e/2}\|v\|_{L^2}+\O{}(h^\infty)\|\psi\|_{\Hh^{-N}}).$$
\end{lemma}

\subsection{Hyperbolic Region}

\label{sec:appDynamics}
Recall from Lemma \ref{lem:decompose} that \m G=G_\Delta+G_B+G_g+\O{L^2\to C^\infty}(h^\infty).\m 

First suppose that $\MS(X)\subset \{|\xi'|_g\leq 1-ch^\e\}$ for some $0<\e<1/2$. 
Then, suppose that 
\m (I+GV)X\psi=f\,\,\m 
and let $G_{\Delta}^{-1/2}$ be a microlocal inverse for $G_{\Delta}^{1/2}$ on 
$$\mc{H}:=\{|\xi'|_g\leq 1-r_{\mc{H}}h^{\e}\}$$
where $r_{\mc{H}}\ll c$
Then
\begin{align*}(I+GV)X_1\psi&=(I+(G_\Delta+G_B)V)X_1\psi +\O{}(h^\infty)\psi \\
&=(I+G_\Delta^{1/2}(I+G_{\Delta}^{-1/2}G_BG_{\Delta}^{-1/2})G_{\Delta}^{1/2}V)X_1\psi+\O{}(h^\infty)\psi=f.
\end{align*}
Thus, $f$ is microlocalized on $\mc{H}$ and, following the formal algebra in \cite[Section 2]{Zaletel} multiplying by $G_\Delta^{1/2}V$, we have 
$$G_{\Delta}^{1/2}VX_1\psi=-G_\Delta^{1/2}VG_\Delta^{1/2}(I+G_{\Delta}^{-1/2}G_BG_{\Delta}^{-1/2})G_{\Delta}^{1/2}VX_1\psi+\O{}(h^\infty)\psi+G_{\Delta}^{1/2}Vf.$$

\begin{remark} By Lemma \ref{lem:dynStrictlyConvex}, a microlocal inverse on $\mc{H}$ will be a microlocal inverse on $\MS(G_BX_1)$.
\end{remark}

Writing $\varphi=G_{\Delta}^{1/2}VX_1 \psi$ and $T=G_{\Delta}^{-1/2}G_BG_{\Delta}^{-1/2}$, we have 
$$(I+G_{\Delta}^{1/2}VG_{\Delta}^{1/2})\varphi=-G_\Delta^{1/2}VG_{\Delta}^{1/2}T\varphi+\O{}(h^\infty)\psi+G_{\Delta}^{1/2}Vf.$$
Hence, letting 
\m R:=-(I+G_{\Delta}^{1/2}VG_{\Delta}^{1/2})^{-1}G_{\Delta}^{1/2}VG_{\Delta}^{1/2},\,\,\m 
we have 
\m \varphi=RT\varphi+\O{}(h^\infty)\psi-RG_{\Delta}^{-1/2}f.\,\,\m 
Here, $T$ is an FIO associated to the billiard map such that
$$\sigma(\exp\left(\frac{\Im z}{h}\oph (l(q,\beta(q)))\right)T)(\beta(q),q)=\exp\left(\frac{i\Re \omega_0l(\beta(q),q)}{h}\right)e^{-i \pi/4} dq^{1/2}\in S_{\e}$$
and $R\in \Ph{}{\e}\cup\Ph{0,0}{2/3}(|\xi'|_g=E') $ is as in \eqref{eqn:reflect}.

Thus by the wavefront set calculus we have for $N>0$ independent of $h$,
\begin{equation}\label{eqn:nonGlance}
(I-(RT)^N)\varphi=\O{}(h^\infty)\psi-\sum_{m=0}^{N-1}(RT)^mRG_{\Delta}^{-1/2}f\end{equation}
and by Egorov's theorem (Lemma \ref{lem:FIOcomp}), we have 
\begin{equation}
\label{eqn:appEgorov}
(RT)_N:=((RT)^*)^N(RT)^N=\oph(a_N)+\O{\Ph{-\infty}{}}(h^{\infty})
\end{equation}
where $a_N\in S_{\e}\cup S_{2/3}^{0,0}(|\xi'|_g=E')$. Moreover, with $\delta=\max(2\e,2/3)$ for $u$ with $\MS(u)\subset\mc{H}$, by the  Sharp G$\mathring{\text{a}}$rding inequality, Lemma \ref{semi-l:garding}, and Lemma \ref{lem:upperBound},
\begin{gather*} 
 \inf_{\mc{H}} \left(|\tilde{\sigma}((RT)_N)(q)|+\O{}(h^{I_{(RT)_N}(q)+1-\delta})\right)\|u\|_{L^2}\leq \|(R T)^Nu\|_{L^2}^2\\
\|(RT)^Nu\|^2\leq \sup_{\mc{H}} \left(|\tilde{\sigma}((RT)_N)(q)|+\O{}(h^{I_{(RT)_N}(q)+1-\delta})\right)\|u\|_{L^2}.
\end{gather*} 
Let 
$$\varkappa_1:=1-\sqrt{\sup_{\mc{H}}\tilde{\sigma}((R T)_N)}\quad\quad \varkappa_2:=\sqrt{\inf_{\mc{H}}\tilde{\sigma}((R T)_N)}-1.$$
Finally, let $\varkappa=\max(\varkappa_1,\varkappa_2)$. Then, we have 

\begin{lemma}
\label{p1-l:parametrixMS}
Suppose that $\varkappa>h^{\gamma_1}$ where $\gamma_1<\min(1/2-\e,1/6).$ Let $c>r_{\mc{H}}$ and $g\in L^2$ have $\MS(g)\subset \{1-Ch^{\e}\leq |\xi'|_g\leq 1-ch^{\e}\}. $ If 
$$(I-(RT)^N)u=g,$$
then for any $\delta>0$,
 $$\MS(u)\subset \{1-(C+\delta)h^{\e}\leq |\xi'|_g\leq 1-(c-\delta)h^{\e}\}.$$ 
 In particular, there exists an operator $A$ with $\|A\|_{L^2\to L^2}\leq 2\varkappa^{-1}$, 
$$A(I-(R T)^N)=I\text{ microlocally on }\mc{H}$$
and if $\MS(g)\subset \{1-Ch^{\e}\leq |\xi'|_g\leq 1-ch^{\e}\}$, then 
$$\MS(Ag)\subset \{1-(C+\delta)h^{\e}\leq |\xi'|_g\leq 1-(c-\delta)h^{\e}\}.$$
\end{lemma}

\begin{proof}
In the case that $\varkappa_2>h^{\gamma_1}$, we write 
$$(I-(R T)^N)=-(R T)^N(I-(R T)^{-N})$$
microlocally on $\mc{H}$ and invert by Neumann series to see that for any $g$, $(I-(RT)^N)u=g$ has a unique solution modulo $h^\infty$ with $\|u\|\leq \varkappa^{-1}\|g\|$. On the other hand, if $\varkappa_1>h^{\gamma_1}$, $\|(R T)^N\|\leq 1-\varkappa_1$, and we have that for any $g$, $(I-(R T)^N)u=g$ has a unique solution with $\|u\|\leq \varkappa_1^{-1}\|g\|.$ 

We will consider the case of $\varkappa_1>h^{\gamma_1}$, the case of $\varkappa_2<h^{\gamma_1}$ being similar with $(R T)^N$ replace by $(R T)^{-N}$. Inversion by Neumann series already shows that we can solve 
$(I-(R T)^N)u_1=g$ with $\|u_1\|\leq \varkappa ^{-1} \|g\|.$ To complete the proof of the lemma, we need to show that this inverse has the required microsupport property. For this, we need a fine almost invariance result near the glancing set. In particular, by Lemma \ref{lem:approxInterp}, that there exists an approximate first integral $\Xi(x,\xi)\in C^\infty(\overline{B^*\pO})$ so that $\Xi=0$, $|d\Xi|>0$ on $S^*\pO$, $\Xi<0$ in $B^*\pO)$ and
\begin{equation}
\label{eqn:approxInterp}\Xi(\beta(q))-\Xi(q)=r(q)
\end{equation}
with $r(q)\in C^\infty(B^*\pO)$ vanishing to infinite order at $S^*\pO$. (See also \cite{PopovInterpolating,MarviziMelrose,popovNear}) In particular, we have that in neighborhood of $S^*\pO$, 
$$\Xi(x',\xi')=e(x',\xi)(|\xi'|_g^2-1)$$
with $e>c>0$. 

For $k\geq1$, let $\chi_k=\chi_k(\zeta)$ with $\chi_{k+1}\equiv 1$ on $\supp \chi_k$ and $\chi_1\equiv 1$ on $\MS(g)$ so that 
$$\supp \chi_k\subset \{1-(C+\delta)h^\e\leq |\xi'|_g\leq 1-(c-\delta)h^\e\}.$$ 
Let $X_k=\oph(\chi_k)$. Finally, let $\chi_\infty\in S_\e$ with $\chi_\infty\equiv 1$ on $\bigcup\limits_k\supp \chi_k$ and 
$$ \supp \chi_\infty \subset\{1-(C+2\delta)h^\e\leq |\xi'|_g\leq 1-(c-2\delta)h^\e\}. $$
Then \eqref{eqn:approxInterp} implies that
$$|\chi_k(\beta(q))-\chi_k(q)|=\O{}(h^\infty).$$

Suppose that $u$ is the unique solution of 
$$(I-(RT)^N)u=g.$$ 
We will show that $u$ is microlocalized as described in the lemma.
Letting $u_1=u$, we have
$$(I-(R T)^N)X_1u_1=g+\O{}(h^\infty)g+[X_1,(R T)^N]X_\infty u_1=:g+g_1.$$
Let $\delta=\max(2\e,2/3)$. Then 
$$[X_1,T]=T(T^{-1}X_1T-X_1)=Th^{1-\delta }B$$
with $B\in \Ph{}{\e}$. 
In fact, 
\begin{equation}
\label{eqn:commuteT} 
T^{-1}X_1T=\oph(\chi_1(\beta(q))+\O{\Ph{}{\e}}(h^{1-2\e}).
\end{equation}
Hence, since $X_\infty u$ is microlocalized $h^\e$ close to glancing, 
$$\MS([X_1,(R T)^N]X_\infty u_1) \subset \{\chi_2\equiv 1\}$$
and 
$g_1:=[X_1,(R T)^N]X_\infty u_1$ has 
$$\|g_1\|\leq Ch^{1-\delta}\varkappa^{-1}\|g\|_{L^2}.$$
Now, let $u_2$ have
 \begin{gather*} (I-(R T)^N)u_2=-g_1,\quad\quad \|u_2\|\leq \varkappa^{-1}\|g_1\|\leq Ch^{1-\delta}\varkappa^{-2}\|g\|
 \end{gather*}
 So,
$$(I-(R T)^N)(X_1u+u_2)=g+\O{}(h^\infty)g.$$
Continuing in this way, let 
$$(I-(R T)^N)u_k=-g_{k-1}\,,\quad g_{k-1}=[X_{k-1},(RT)^N]X_\infty u_{k-1}.$$
Then, 
$$\|u_k\|\leq  \varkappa^{-2k}(h^{k(1-\delta)})\|g\|_{L^2}.$$
Moreover, letting $\tilde{u}\sim \sum_k X_ku_k$, we have $X_\infty \tilde{u}=\tilde{u}+\O{}(h^\infty)\tilde{u}$ and  
$$(I-(R T)^N)\tilde{u}=g+\O{}(h^\infty)g$$
which implies $\tilde{u}-u=\O{}(h^\infty)$ and hence that $(I-(RT)^N)$ has a microlocal inverse, $A$, with the properties claimed in the lemma.
\end{proof}

We now suppose that $\psi$ solves \eqref{p1-e:boundaryMain} and use \eqref{eqn:nonGlance} to obtain estimates on $\psi$. Let $\chi_k\in S_\e$ with $\chi_k\equiv 1$ on $\{|\xi'|_g\leq 1-2kch^\e\}$ and $\supp \chi_k\subset \{|\xi'|_g\leq 1-(2k-1)ch^\e\}. $ 
Then
$$(I+GV)X_1\psi=-[X_1,GV]\psi+X_1Gv=:\psi_1+\tilde{v}$$
where $\MS(\psi_1)\subset \mc{H}\cap\{|\xi'|_g\geq 1-3c/2h^\e\}.$ 
Then with $\varphi=G_{\Delta}^{1/2}VX_1\psi$, 
$$(I-(RT)^N)\varphi=\O{}(h^\infty)\psi -\sum_{m=0}^{N-1}(RT)^mRG_{\Delta}^{-1/2}(\psi_1+\tilde{v})$$
and hence by Lemma \ref{p1-l:parametrixMS}, when $\varkappa\geq h^{\gamma_1}$ for $\gamma_1<\min(1/2-\e,1/6)$, 
$$\varphi=\O{}(h^\infty)\psi-\sum_{m=0}^{N_1}A(RT)^mRG_{\Delta}^{-1/2}(\psi_1+\tilde{v})$$
and, using the microsupport statement from Lemma \ref{p1-l:parametrixMS}, 
$$X_2\varphi=-\sum_{m=0}^{N-1}A(RT)^mRG_{\Delta}^{-1/2}\tilde{v}+\O{\Ph{-\infty}{}}(h^\infty)(\psi+v).$$

Hence, 
\begin{align*} 
\|X_2\varphi\|_{L^2}&\leq \varkappa^{-1}\left\|\sum_{m=0}^{N-1}(RT)^mRG_{\Delta}^{-1/2}X_1Gv\right\|+\O{}(h^\infty)(\|\psi\|+\|v\|)\\
&\leq C\varkappa^{-1}e^{N\Do(\Im z)_-/h}h^{1/2-\e/2}\|v\|+\O{}(h^\infty)\|\psi\|
.\end{align*}
Then, since $\varphi=G_{\Delta}^{1/2}VX_1\psi$, $VX_1\psi=G_{\Delta}^{-1/2}\varphi+\O{}(h^\infty)\psi$ and 
\begin{align*} X_3\psi&=-X_3GV\psi +X_3Gv=-X_3GVX_1\psi+X_3Gv+\O{}(h^\infty)\psi\\
&=-X_3GG_{\Delta}^{-1/2}\varphi+X_3Gv+\O{}(h^\infty)\psi=-X_3GG_{\Delta}^{-1/2}X_2\varphi+X_3Gv+\O{}(h^\infty)\psi.
\end{align*}
Hence,
\begin{align*} 
\|X_3\psi\|& \leq \|X_3GG_{\Delta}^{-1/2}X_2\varphi\| +\|X_3Gv\|+\O{}(h^\infty)\|\psi\|\\
&\leq C(\varkappa^{-1}h^{1-\e}e^{(N+1)\Do(\Im z)_-/h}\|v\|+\O{}(h^\infty)\|\psi\|  )
\end{align*} 

Next, we examine when $\varkappa\geq c h^{\gamma_1}$. 
If this is not the case, then
$$\liminf_{h\to 0} \frac{\inf \left||\tilde\sigma((RT)_N)(q)| -1\right|}{h^{\gamma_1}}=0.\,\,\m 
So, let 
$$|\tilde{\sigma}(R T)_N(q)|=e^{e(q)}.$$
Taking logs and renormalizing we have
$$\frac{2\Im z}{h}Nl_N(q)-\frac{2\Im z}{h}Nl_N(q) +\log |\tilde{\sigma}((RT)_N)(q)|=e(q).$$
This implies
\begin{align*}
-\frac{\Im z}{h}&=-l_N^{-1}(q)\left[\frac{\Im z}{h}l_N(q)+\frac{1}{2N}\log |\tilde{\sigma}((RT)_N)(q)|+e(q)\right]\\
&=-l_N^{-1}(q)(r_N(q)+e(q)).
\end{align*}
where $r_N$ as in \eqref{eqn:defineAverageReflection}.
Thus, if $\varkappa\leq ch^{\gamma_1}$, for any $c>0$,
\begin{equation}
\nonumber
\inf_{\mc{H}}-l_N^{-1}(r_N+ch^{\gamma_1})\leq -\frac{\Im z}{h}\leq \sup_{\mc{H}}-l_N^{-1}(r_N-ch^{\gamma_1}).
\end{equation}

Now, writing 
\m RT=\left[R\exp\left(-\frac{\Im z}{h}\oph (l(q),\beta(q))\right)\right]\left[\exp\left(\frac{\Im z}{h}\oph (l(q),\beta(q))\right)T\right]\,\,\m
and applying Lemma \ref{lem:EgorovSheaf} shows that
\begin{multline*}r_N(q):=\tilde{\sigma}((RT)_N)(q)\\
=\exp\left(-\frac{2\Im z}{h}\sum\limits_{n=0}^{N-1}l(\beta^n(q),\beta^{n+1}(q))\right)\prod\limits_{i=1}^N\left(|\tilde{\sigma}(R)(\beta^{i}(q))|^2+\O{}(h^{I_{R}(\beta^i(q))+1-2\e})\right).\end{multline*}

Summarizing the discussion, we have
\begin{lemma}
\label{lem:dynamicalRestriction}
Let $0<\e<1/2$, $\gamma_1<\min(1/2-\e,1/6)$, $c>0$, $M>0$ and suppose that $\chi\equiv 1$ on $\{|\xi'|_g\leq 1-Ch^\e\}$ and $\supp \chi \subset \{|\xi'|\leq 1-ch^\e\}$. Suppose further that $\psi$ solves \eqref{p1-e:boundaryMain}. Then there exists $h_0>0$ small enough so that if $0<h<h_0$ and 
\begin{equation}
\label{eqn:prelimResFree}
-\frac{\Im z}{h}<\inf_{\mc{H}}-l_N^{-1}(r_N+ch^{\gamma_1})
\text{  or  }
-\frac{\Im z}{h}>\sup_{\mc{H}}-l_N^{-1}(r_N-ch^{\gamma_1}),
\end{equation}
where $l_N$ and $r_N$ are as in \eqref{eqn:defineLength} and \eqref{eqn:defineAverageReflection} respectively, then
\begin{equation}
\label{eqn:dynamicsRestrict}
\|\oph(\chi)\psi\|_{L^2}\leq C(h^{1-\e-\gamma_1}e^{(N+1)\Do(\Im z)_-/h}\|v\|_{L^2}+\O{}(h^\infty)\|\psi\|_{\Hh^{-M}})
\end{equation}
\end{lemma}

\subsection{Glancing Region}
Let $\chi\in S_\e(|\xi'|_g=1) $ with $\chi \equiv 1 $ on $\{||\xi'|_g-1|\leq ch^\e\}$ and $\supp \chi \subset \{||\xi'|_g-1|\leq Ch^\e\}$. Then
$$(I+GV)\oph(\chi)\psi=[GV,\oph(\chi)]\psi+\oph(\chi)Gv.$$
Let $\varphi_i$ be a partition of unity on $S^*\partial\Omega$.  We then use the microlocal model for $G$ near glancing. 
$$\sum_i(I+h^{2/3}J_{i}\omega^{-1} \mc{A}_-\mc{A}iC^{-1}J_i^{-1}V)\varphi_i\oph(\chi)\psi=\oph(\chi)Gv+[GV,\oph(\chi)]\psi+\O{}(h^\infty)(\psi).$$

First, observe that if $\alpha>-2/3$, then our model shows that $(I+GV)$ is an elliptic pseudodifferential operator on $\supp \chi$ and hence
\begin{lemma}
\label{lem:glanceSmallAlpha}
Suppose that $\alpha>-2/3$. Then under the assumptions of Lemma \ref{lem:dynamicalRestriction}, there exists $N>0$ so that
$$\|\psi\|_{L^2}\leq Ch^{-N}\|v\|_{L^2}.$$
\end{lemma}

Throughout the rest of our analysis near glancing, it will be convenient to use $\Xi$ from Lemma \ref{lem:approxInterp}. Then 
$$\Xi(x',\xi'):=(|\xi'|_g^2-1)(2Q(x',\xi'))^{-2/3}+\O{}((|\xi'|_g^2-1)^2).$$
Moreover, $\xi_1(\kappa^{-1}(x',\xi'))=\Xi(x',\xi')+\O{}((|\xi'|_g^2-1)^\infty)$ 
where $\kappa$ is the symplectomorphism \eqref{eqn:kappa} reducing the billiard ball map for the Friedlander model to that for $\Omega$ near $(x',\xi')\in S^*\pO$. 
In particular, notice that if $\chi\in S_\e^{0,0}(\xi_1=1)$ with $\supp \chi\subset\{ah^\e_1\leq 1-|\xi'|_g^2\leq bh^{\e}\}$, then
\begin{gather*} 
\sigma(J_i\oph(\chi(\Xi))J_i^{-1})=\chi(\xi_1)\\
\MS(_i\oph(\chi(\Xi))J_i^{-1})\subset \{ah^\e_1\leq \xi_1\leq bh^{\e_2}\}.
\end{gather*}

Now, the assumption that on $|\xi'|_g-1>Mh^{2/3}$, 
$$\left|1+\frac{h\sigma(V)}{2\sqrt{|\xi'|^2_g-1}}\right|\geq \delta \left \la \frac{h^{1+\alpha}}{\sqrt{|\xi'|_g^2-1}}\right\ra$$
(see \eqref{resFree-e:ellipticAssumption})
together with Lemma \ref{semi-l:elliptic} and \eqref{eqn:Gsymb2} imply that $I+GV$ is microlocally invertible on $|\xi'|_g\geq 1+Mh^{2/3}$. 

When $\alpha<-2/3$, we can localize further. In particular,  fix $M_1>0$. Then since $V$ is elliptic and $\alpha<-2/3$, $I+GV$ is an elliptic pseudodifferential operator when for some $\delta>0$ and all $1\leq j\leq M_1$,
$$|h^{-2/3}\Xi(x',\xi')+\zeta_j|\geq \delta,\quad\quad h^{-2/3}\Xi(x',\xi')+\zeta_{M_1+1}\geq \delta $$

So, there exists $C>0$ such that, letting $\chi_2\in S_{2/3}(|\xi'|_g=0)$ have $\supp \chi_2\subset |\xi_1|\leq Ch^\e$ and 
\begin{equation}
\label{eqn:chi2}
\chi_2\equiv 1\text{ on }\begin{cases} |\xi_1|\leq CMh^{2/3}&\alpha=-2/3\\
|\xi_1h^{-2/3}+\zeta_j|\leq \delta,\,C h^\e \leq \xi_1\leq h^{2/3}\zeta_{M_1}+\delta h^{2/3} &\alpha<-2/3,\end{cases}
\end{equation}
we have
\begin{lemma}
\label{lem:ellipticGlance}
Let $\chi_2$ be as in \eqref{eqn:chi2}. Then 
\begin{align*} 
\|(1-\oph(\chi_2(\Xi)))\oph(\chi_1)\psi\|&\leq Ch^{-1/3+\e/2-\alpha}(\|\oph(\chi)Gv\|+\|[GV,\oph(\chi)]\psi\|+\O{}(h^\infty)\|\psi\|
\end{align*} 
and hence, under the assumptions of Lemma \ref{lem:dynamicalRestriction}, 
\begin{align*} 
\|(1-\oph(\chi_2(\Xi)))\oph(\chi_1)\psi\|&\leq Ch^{-1/3+\e/2}(h^{2/3}+h^{1-\e-\gamma_1}e^{(N+1)\Do(\Im z)_-/h})\|v\|+\O{}(h^\infty)\|\psi\|
\end{align*} 
\end{lemma}

\subsubsection{Flux formula}

With $\chi_2$ as in \eqref{eqn:chi2}, define
$$\psi_{ng}:=(1-\oph(\chi_2(\Xi))\oph(\chi_1))\psi.$$
and $\psi_g:=\psi-\psi_{ng}$.

By an integration by parts, we have for a solution $u$ to \eqref{p1-e:main},
\begin{equation} 
\label{p1-e:greens}\left(\frac{2\Re z\Im z}{h}\|u\|_{L^2}^2-\Im \la B\psi,\psi\ra\right)= -\Im\la h v,\psi\ra .
\end{equation}
On the other hand, 
\begin{equation}
\label{p1-e:greens2}
u=h^{-1}\S h\partial_\nu u-\D u=-(h^{-1}\S B +\D )\psi+\S v=-\S V\psi +.\S v
\end{equation}
Since we already have estimates for $\psi_{ng}$, we write 
$$u=(-\S V \psi_g)+(\S (v -V \psi_{ng}))=:u_{g}+u_{ng}.$$

Now, \cite[Theorem 1.1]{GalkSLO} together with an application of the Phragm\'en Lindel\"of principle implies 
\begin{align*} 
\|\S (v- V\psi_{ng})\|=\|u_{ng}\|&\leq h^{5/6}e^{\Do(\Im z)_-/h}(\|v\|+h^{\alpha}\|\psi_{ng}\|_{\Hh^m})\\
\|\S V\psi_g\|=\|u_g\|&\leq Ch^{5/6+\alpha}e^{\Do(\Im z)_-/h}\|\psi_g\|
\end{align*} 
Then, 
\begin{align*} \|u\|^2-\|u_{g}\|^2&=2\Re \la u_{g},u_{ng}\ra +\|u_{ng}\|^2\\
&\leq \delta\|u_{g}\|^2+(1+2\delta^{-1})\|u_{ng}\|^2\\
&\leq C\delta h^{5/3+2\alpha}e^{2\Do (\Im z)_-/h}\|\psi_g\|^2+(1+2\delta^{-1})\|u_{ng}\|^2
\end{align*}
\begin{align*}
\left|\la B\psi,\psi\ra -\la B\psi_g,\psi_g\ra\right| &=|\la B\psi_g,\psi_{ng}\ra +\la B\psi_{ng},\psi_g\ra +\la B\psi_{ng},\psi_{ng}\ra|\\
&\leq C(\delta\|\psi_g\|^2+C(1+\delta^{-1}))\|\psi_{ng}\|_{\Hh^m}^2.
\end{align*} 
Now, rewrite \eqref{p1-e:greens} as 
$$\frac{2\Re z\Im z}{h}\|u_{g}\|^2-\Im \la B\psi_g,\psi_g\ra =\Im \la hv,\psi\ra +\frac{2\Re z\Im z}{h}(\|u_{g}\|^2-\|u\|^2)+\Im (\la B\psi_g,\psi_g\ra -\la B\psi,\psi\ra).$$
Plugging our estimates in together gives 
\begin{equation}
\label{eqn:est1}
\begin{aligned} 
\left|\frac{2\Re z\Im z}{h}\|u_{g}\|^2+\Im \la- hN_2\psi_g,\psi_g\ra +\la -h\Im V\psi_g,\psi_g\ra \right|\!\!\!\!\!\!\!\!\!\!\!\!\!\!\!\!\!\!\!\!\!\!\!\!\!\!\!\!\!\!\!\!\!\!\!\!\!\!\!\!\!\!\!\!\!\!\!\!\!\!\!\!\!\!\!\!\!\!\!\!\!\!\!\!\!\!\!\!\!\!\!\!\!\!\!\!\!\!\!\!\!\!\!\!\!\!\!\!\!\!\!\!\!\!\!\!\!\!\!\!\!\!\!\!\!\!\!\!\!\!\!\!\!\!\!\!\!\!\!\!\!\!\!\!\!\!\!\!\!&\\
&\leq Ch(\delta_1^{-1}\|v\|^2+\delta_1\|\psi\|^2)+ C|\Im z|h^{-1}(\delta_2h^{5/3+2\alpha}e^{2\Do(\Im z)_-/h}\|\psi_g\|^2+(1+\delta_2^{-1})\|u_{ng}\|^2)\\
&\quad\quad+C(\delta_3\|\psi_g\|^2+(1+\delta_3^{-1})\|\psi_{ng}\|_{\Hh^m}^2)\\
&\leq C(\delta_1h+|\Im z|h^{2/3+2\alpha}e^{2\Do(\Im z)_-/h}\delta_2 +\delta_3)\|\psi_g\|_{L^2}^2\\
&\quad\quad +C(h\delta_1 +\delta_3^{-1}+(1+\delta_2^{-1})|\Im z|h^{2/3+2\alpha}e^{2\Do(\Im z)_-/h} )\|\psi_{ng}\|_{\Hh^m}^2 \\
&\quad\quad +C(h\delta_1^{-1}+\delta_2^{-1}|\Im z|h^{2/3+2\alpha}e^{2\Do(\Im z)_-/h})\|v\|_{L^2}^2
\end{aligned}
\end{equation}
In particular, we have
\begin{lemma}
For all $\gamma_1\in \re$, $c>0$, there exists $C>0$ so that if
\begin{equation}
\label{eqn:estpsi}\left|\frac{2\Re z\Im z}{h}\|u_{g}\|^2+\Im \la- hN_2\psi_g,\psi_g\ra +\la -h\Im V\psi_g,\psi_g\ra \right|\geq ch^{\gamma_1}\|\psi_g\|^2.
\end{equation}
then
\begin{align*}\|\psi_g\|^2&\leq C(h^{\gamma_1}+h^{-\gamma_1}+(1+|\Im z|h^{2/3+2\alpha-\gamma_1}e^{2\Do(\Im z)_-/h})|\Im z|h^{2/3+2\alpha}e^{2\Do(\Im z)_-/h})\|\psi_{ng}\|_{\Hh^m}^2\\
&\quad\quad +C(h^{2-\gamma_1}+|\Im z|^2h^{4/3+2\alpha-\gamma_1}e^{4\Do(\Im z)_-/h})\|v\|_{L^2}^2
\end{align*}
\end{lemma}

\subsubsection{Estimates on the glancing set}
We now obtain estimates of the form \eqref{eqn:estpsi} using the description of the single and double layer potentials from section \ref{sec:Layer}. First, observe that
\begin{align*} 
\|u_{g}\|_{L^2(\Omega)}^2&=\left\la \mc{B}\psi_g,\psi_g\right\ra_{L^2(\pO)} 
\end{align*} 
where by Lemma \ref{lem:potentialGlance}
$$\mc{B}:=V^*\S ^*\S V \in h^{2+2\alpha}\Ph{}{1-\e/2}(|\xi'|_g=1)$$
is elliptic and has symbol given by
$$\sigma(\mc{B})=\frac{|\sigma(hV)|^2}{2Q}\left(\Psi_{\S }(\alpha_{0h})\composed \kappa^{-1}\right).$$

Take $\e,\e_1>0$ small enough and let 
\begin{equation}
\label{eqn:Ldef}
\mc{L}_\alpha:=\begin{cases} 
\{||\xi'|_g-1|\leq h^\e\,,\,|\Xi+h^{2/3}\zeta_j|<\e_1 h^{2/3}\text{ or }\Xi\leq -M_1h^{2/3}\}&\alpha<-2/3\\
\{||\xi|'_g-1|\leq CMh^{2/3}&\alpha\geq-2/3
\end{cases}
\end{equation}
where $C$ and $M$ are as in \eqref{eqn:chi2}. 
  
 Now, define
$$ 
\frac{2\Re z \Im z}{h}\|u_{g}\|^2+\Im \la -hN_2\psi_g,\psi_g\ra +\la -h\Im V\psi_g,\psi_g\ra =\\
\left\la A\psi_g,\psi_g\right\ra$$
where
$$A:=\frac{2\Re  z\Im z}{h}\mc{B}-\Im (hN_2+hV_2).$$
 Then, applying the Sharp G$\mathring{\text{a}}$rding inequality (see Lemma \ref{semi-l:garding}) along with bounds on the norm of pseudodifferential operators (see Lemma \ref{lem:upperBound}), we obtain
\begin{multline} 
\label{eqn:ImzlowerBound}
\inf_{\mc{L}_\alpha}\left(\frac{2\Re z\Im z}{h}\frac{|\sigma(hV)|^2}{2Q}\Psi_{\S}(h^{-2/3}\Xi)(1+\O{}(h^{\e/2}))\right.\\
\left.-h(\Im \sigma(N_2)+\sigma(\Im V))-ch^{1/3+\e/2}-ch^{4/3+\alpha}\right)\|\psi_g\|^2\leq \la A\psi_g,\psi_g\ra 
\end{multline}
\begin{multline}
\label{eqn:ImzupperBound}
\la A\psi_g,\psi_g\ra \leq 
\sup_{\mc{L}_\alpha}\left(\frac{2\Re z\Im z}{h}\frac{|\sigma(hV)|^2}{2Q}\Psi_{\S}(h^{-2/3}\Xi)(1+\O{}(h^{\e/2}))\right.\\
\left.-h(\Im \sigma(N_2)+ \sigma(\Im V))+ch^{1/3+\e/2}+ch^{4/3+\alpha}\right)\|\psi_g\|^2
\end{multline} 
Notice that for all $\delta>0$, there exists $M_1$ large enough and $\e_1$ small enough so that 
$$1-\delta \leq \Psi_{\Sl}(h^{-2/3}\Xi)\leq 1+\delta,\quad\quad(x,\xi)\in \mc{L}_{\alpha}\quad (\alpha<-2/3).$$ 
So, we have 
\begin{lemma}
\label{lem:glance}
For all $\delta>0$ there exists $h_0>0$, $N,M>0$, $C,\,c>0$ such that for $0<h<h_0$ if $\pm \Im z\geq 0$ and one of the following holds
\begin{equation} 
\label{eqn:cond1}
\begin{gathered}\frac{-\Im z}{h}\leq \inf_{\mc{L}_\alpha}-\frac{h(\Im \sigma(N_2)+\sigma(\Im V)+c(h^{1/3+\alpha}+h^{-1/3+\e/2}))Q}{|\sigma(hV)|^2\Psi_{\Sl}(h^{-2/3}\Xi)}(1\pm\delta)\\
\frac{-\Im z}{h}\geq \sup_{\mc{L}_\alpha}-\frac{h(\Im \sigma(N_2)+\sigma(\Im V)-c(h^{1/3+\alpha}+h^{-1/3+\e/2}))Q}{|\sigma(hV)|^2\Psi_{\Sl}(h^{-2/3}\Xi)}(1\mp \delta)
\end{gathered}
\end{equation}
then
\begin{equation}
\label{eqn:est1}
 \|\psi_g\|_{L^2}\leq Ch^{-N}(\|v\|_{L^2}+\|\psi_{ng}\|_{\Hh^m})+\O{}(h^\infty)\|\psi\|_{\Hh^{-M}}.
\end{equation}
If $\alpha<-2/3$, we can replace the conditions \eqref{eqn:cond1} with
\begin{equation*} 
\begin{gathered}
\frac{-\Im z}{h}\leq \inf_{\mc{L}_\alpha}-\frac{h(\Im \sigma(N_2)+\sigma(\Im V)+c(h^{1/3+\alpha}+h^{-1/3+\e/2})Q}{|\sigma(hV)|^2}(1\pm \delta)\\
\frac{-\Im z}{h}\geq \sup_{\mc{L}_\alpha}-\frac{h(\Im \sigma(N_2)+\sigma(\Im V)-c(h^{1/3+\alpha}+h^{-1/3+\e/2}))Q}{|\sigma(hV)|^2}(1\mp \delta).
\end{gathered}
\end{equation*}
\end{lemma}

\subsection{Further localization away from the real axis when $\alpha<-2/3$}
\label{sec:furtherLoc}
We now focus our attention on the region $|\Im z|\geq ch^N$ for some $N>0$ and $\alpha<-2/3$. In this region, we are able to decompose $\psi=u|_{\pO}$ into pieces, $\psi_j$, concentrating at $\Xi=\zeta_jh^{2/3}$, that still have 
$$(I+GV)\psi_j=Gv_j$$
with the norm of $v_j$ controlled by the norm of $v$.

 We again use the representation of $G$ near glancing. With $\chi$ and $\varphi_i$ as above
$$(I+\sum_ih^{2/3}J_i\omega^{-1}\mc{A}_i\mc{A}iC^{-1}J_i^{-1}V\varphi_i)\oph(\chi)\psi=\oph(\chi)Gv+[GV,\oph(\chi)]\psi+\O{}(h^\infty)\psi.$$
Fix $\e_1>0$ small enough and let $\chi_j\equiv 1$ on $|\xi_1+\zeta_j|\leq \e_1 h^{2/3}$ with $\supp \chi_j\subset |\xi_1+\zeta_j|\leq 2\e_1 h^{2/3}$ and let $L_j=\oph(\chi_j(\Xi)).$ Then
\begin{multline*} \sum_i(I+h^{2/3}J_i\omega^{-1}\mc{A}_i\mc{A}iC^{-1}J_i^{-1}V)\varphi_iL_j\oph(\chi)\psi=\\
L_j\oph(\chi)Gv+L_j[GV,\oph(\chi)]\psi+[GV,L_j]\oph(\chi)\psi+\O{}(h^\infty)\psi
\end{multline*} 
Now, $[GV,L_j]$ is a pseudodifferential operator with support on the complement of $\mc{L}_\alpha$. Therefore by Lemma \ref{lem:ellipticGlance} there exists $M>0$ so that
$$\|[GV,L_j]\oph(\chi)\psi\|\leq h^{-M}\|v\|+\O{}(h^\infty)\|\psi\|.$$
So,
$$(I+GV)L_j\oph(\chi)\psi=w$$
with 
$$\|w\|\leq h^{-M}\|v\|+\O{}(h^\infty)\|\psi\|.$$
Now, 
$G^{-1}=N_1+N_2$ and since $|\Im z|\leq Mh\log h^{-1}$, 
$$\|h(N_1+N_2)\|_{\Hh^1\to L^2}\leq \frac{C}{|\Im z|}.$$
Hence, using that $|\Im z|\geq ch^{N}$, we have 
$$(I+GV)L_j\oph(\chi)\psi=GG^{-1}w=G(N_1+N_2)w=:Gv_j$$
so that for some $M>0$,
$$\|v_j\|\leq h^{-M}\|v\|+\O{}(h^\infty)\|\psi\|.$$
So, formulas \eqref{p1-e:greens} and \eqref{p1-e:greens2} hold with $\psi$ replaced by $L_j\oph(\chi)\psi$ and $v$ replaced by $v_j$. Let  $\psi_j=L_j\oph(\chi)\psi$,
$$\mc{L}_j:=\{|\Xi(x',\xi)+h^{2/3}\zeta_j|<2\e_1 h^{2/3},$$
and $u_j$ be the solution to 
$$\begin{cases} (-h^2\Delta-z^2)u_j=0&\text{in }\Omega\\
(h\partial_\nu+B)u_j=v_j&\text{on }\pO\\
u_j|_{\pO}=\psi_j
\end{cases}
$$
Next, fix $\delta>0$ and take $\e_1$ small enough, $\pm\Im z\geq 0$. Then following the arguments above,
\begin{multline} 
\label{eqn:ImzlowerBoundLoc}
\inf_{\mc{L}_j}\left(\frac{2 \Im z}{h(1\pm\delta)}\frac{|\sigma(hV)|^2}{2Q}-h(\Im \sigma(N_2)+\sigma(\Im V))-ch^{1/3+\e/2}-ch^{4/3+\alpha}\right)\|\psi_j\|^2\\\leq \la A\psi_j,\psi_j\ra 
\end{multline}
\begin{multline}
\label{eqn:ImzupperBoundLoc}
\la A\psi_j,\psi_j\ra \leq \\
\sup_{\mc{L}_j}\left(\frac{2\Im z}{h(1\mp \delta)}\frac{|\sigma(hV)|^2}{2Q}-h(\Im \sigma(N_2)+\sigma(\Im V))+ch^{1/3+\e/2}+ch^{4/3+\alpha}\right)\|\psi_j\|^2
\end{multline} 
and
\begin{equation}
\label{eqn:locEst}
\left|\left\la A\psi_j,\psi_j\right\ra \right|\leq C(\delta^{-1}\|v_j\|+\delta\|\psi_j\|).
\end{equation}
In particular, using that
$$\sigma(hN_2)=(2hQ)^{1/3}\frac{A_-'(h^{-2/3}\Xi)}{A_-(h^{-2/3}\Xi)},$$
we have
\begin{lemma}
\label{lem:glancej}
Suppose that $\pm \Im z\geq ch^M$, $\alpha<-2/3$. Fix $j>0$. Then there exist $h_0>0$, $N,\,C>0$ such that if one of the following holds 
\begin{gather*} 
\frac{-\Im z}{h}\leq \inf_{\mc{L}_j}-\frac{h(h^{-2/3}(2Q)^{1/3}\Im \frac{A'_-(-\zeta_j)}{A_-(-\zeta_j)} +\sigma(\Im V)+ch^{1/3+\alpha})Q}{|\sigma(hV)|^2}(1\pm\delta)\\
\frac{-\Im z}{h}\geq \sup_{\mc{L}_j}-\frac{h(h^{-2/3}(2Q)^{1/3}\Im \frac{A'_-(-\zeta_j)}{A_-(-\zeta_j)} +\sigma(\Im V)-ch^{1/3+\alpha})Q}{|\sigma(hV)|^2}(1\mp\delta)
\end{gather*}
then
\begin{equation*}
 \|\psi_j\|\leq Ch^{-N}\|v\|+\O{}(h^\infty)\|\psi\|.
\end{equation*}
\end{lemma}

With these estimates in hand, for any $M>0$, let
\begin{equation}
\label{eqn:Ldef2}
\mc{L}'_{M}:=\{-2h^\e\leq \Xi\leq  (-\zeta_{M+1}+2\e)h^{2/3}\}.
\end{equation}
and let $\chi_2'=\chi_2'(\xi_1)\in S_{2/3}$ have $\chi_2\equiv 1$ on 
$$\{-h^\e\leq \xi_1\leq  (-\zeta_{M+1}+\e)h^{2/3}\}$$ and 
$\supp\chi_2\subset\mc{L}'_{M}.$
Then define
$$\psi'_{g}=\oph(\chi_2(\Xi))\oph(\chi_1)\psi$$ 
and
$\psi'_{ng} =\psi-\psi_{g}'.$
Then \eqref{eqn:ImzlowerBound} and \eqref{eqn:ImzupperBound} still hold with $\mc{L}$ replaced by $\mc{L}'_{M}$ and we have 

\begin{lemma}
\label{lem:glanceprime}
For all $\delta>0$ there exists $h_0>0$, $N,M>0$, $C>0$ such that for $0<h<h_0$ if $\pm\Im z\geq 0$ and
one of the following holds 
\begin{equation*}
\begin{gathered}\frac{-\Im z}{h}\leq \inf_{\mc{L}'_{M}}-\frac{h(\Im \sigma(N_2)+\sigma(\Im V)+ch^{1/3+\alpha})Q}{|\sigma(hV)|^2}(1\pm\delta)\\
\frac{-\Im z}{h}\geq \sup_{\mc{L}'_M}-\frac{h(\Im \sigma(N_2)+\sigma(\Im V)-ch^{1/3+\alpha})Q}{|\sigma(hV)|^2}(1\mp\delta)
\end{gathered}
\end{equation*}
then
\begin{equation*}
 \|\psi'_g\|_{L^2}\leq Ch^{-N}(\|v\|_{L^2}+\|\psi'_{ng}\|_{\Hh^m})+\O{}(h^\infty)\|\psi\|_{\Hh^{-N}}.
\end{equation*}
\end{lemma}

So, combining Lemmas \ref{p1-l:ellipticEstimate}, \ref{lem:dynamicalRestriction},  \ref{lem:glanceSmallAlpha}, \ref{lem:ellipticGlance}, \ref{lem:glance}, \ref{lem:glancej}, and \ref{lem:glanceprime} gives
\begin{theorem}
\label{thm:fullGenerality}
Let $\psi $ be a solution to \eqref{p1-e:boundaryMain}. Fix $\delta>0$, $0<\e<1/2$, $\gamma_1<\min(\frac{1}{2}-\e,\frac{1}{6})$, $M_1, M_2>0$. Then there exists $h_0>0$ and $N>0$ such that for $0<h<h_0$ if
\begin{equation*}
-\frac{\Im z}{h}<\inf_{\mc{H}}-l_N^{-1}(r_N+ch^{\gamma_1})
\text{  or  }
-\frac{\Im z}{h}>\sup_{\mc{H}}-l_N^{-1}(r_N-ch^{\gamma_1}),
\end{equation*}
$\pm\Im z\geq 0$ and one of the following holds
\begin{equation} 
\label{eqn:cond12}
\begin{gathered}\frac{-\Im z}{h}\leq \inf_{\mc{L}_\alpha}-\frac{h(\Im \sigma(N_2)+\sigma(\Im V)+c(h^{1/3+\alpha}+h^{-1/3+\e/2}))Q}{|\sigma(hV)|^2\Psi_{\Sl}(h^{-2/3}\Xi)}(1\pm\delta)\\
\frac{-\Im z}{h}\geq \sup_{\mc{L}_\alpha}-\frac{h(\Im \sigma(N_2)+\sigma(\Im V)-c(h^{1/3+\alpha}+h^{-1/3+\e/2}))Q}{|\sigma(hV)|^2\Psi_{\Sl}(h^{-2/3}\Xi)}(1\mp \delta)
\end{gathered}
\end{equation}
then 
\begin{equation}
\label{eqn:bvEstimate}\|\psi\|_{L^2}\leq Ch^{-N}\|v\|_{L^2}.
\end{equation}
and $\mc{P}(z)$ is invertible.
Moreover, if $\alpha<-2/3$ then \eqref{eqn:cond12} can be replaced by 
\begin{equation}
\label{eqn:cond13} 
\begin{gathered}
\frac{-\Im z}{h}\leq \inf_{\mc{L}_\alpha}-\frac{h(\Im \sigma(N_2)+\sigma(\Im V)+c(h^{1/3+\alpha}+h^{-1/3+\e/2})Q}{|\sigma(hV)|^2}(1\pm \delta)\\
\frac{-\Im z}{h}\geq \sup_{\mc{L}_\alpha}-\frac{h(\Im \sigma(N_2)+\sigma(\Im V)-c(h^{1/3+\alpha}+h^{-1/3+\e/2}))Q}{|\sigma(hV)|^2}(1\mp \delta).
\end{gathered}
\end{equation}

Finally, if $\pm \Im z\geq ch^{M_1}$ and $\alpha<-2/3$, then \eqref{eqn:bvEstimate} holds and $\mc{P}(z)$ is invertible if 
\begin{equation*}
\begin{gathered}\frac{-\Im z}{h}\leq \inf_{\mc{L}'_{M_2}}-\frac{h(\Im \sigma(N_2)+\sigma(\Im V)+ch^{1/3+\alpha})Q}{|\sigma(hV)|^2}(1\pm\delta)\\
\frac{-\Im z}{h}\geq \sup_{\mc{L}'_{M_2}}-\frac{h(\Im \sigma(N_2)+\sigma(\Im V)-ch^{1/3+\alpha})Q}{|\sigma(hV)|^2}(1\mp\delta)
\end{gathered}
\end{equation*}
and one of the following holds for $1\leq j\leq M_2$
\begin{gather*} 
\frac{-\Im z}{h}\leq \inf_{\mc{L}_j}-\frac{h(h^{-2/3}(2Q)^{1/3}\Im \frac{A'_-(-\zeta_j)}{A_-(-\zeta_j)} +\sigma(\Im V)+ch^{1/3+\alpha})Q}{|\sigma(hV)|^2}(1\pm\delta)\\
\frac{-\Im z}{h}\geq \sup_{\mc{L}_j}-\frac{h(h^{-2/3}(2Q)^{1/3}\Im \frac{A'_-(-\zeta_j)}{A_-(-\zeta_j)} +\sigma(\Im V)-ch^{1/3+\alpha})Q}{|\sigma(hV)|^2}(1\mp\delta)
\end{gather*}
\end{theorem}
\noindent In particular, this implies Theorem \ref{thm:mainGeneral}.

 \section{Application to transparent obstacles}
 \label{sec:appTrans}
 In the case of transparent obstacles, we want to consider \eqref{intro-e:transparent}, repeated here for the reader's convenience,
 \begin{equation*}
 \begin{cases}
(-c^2\Delta -\lambda^2)u_1=0&\text{ in }\Omega\\
(-\Delta -\lambda^2) u_2=0&\text{ in }\re^d\setminus\overline{\Omega}\\
u_1=u_2&\text{ on }\pO\\
\partial_\nu u_1-\aleph\partial_\nu u_2=0&\text{ on } \pO\\
u_2\text{ is $\lambda$-outgoing}
\end{cases}
\end{equation*}
Thus, writing $\lambda=cz/h$, in the language of \eqref{p1-e:main},
 $$B=hN_2(z/h)+\aleph h\No(c z/h)-hN_2(z/h)$$
 where $\No$ is the outgoing Dirichlet to Neumann map for the exterior problem (see section \ref{sec:mainThm})
 
 Thus,
 $V=\aleph\No(cz/h)-N_2(z/h)$ has 
 $$V\in h^{-2/3}(\Ph{1/2,1}{2/3}(|\xi'|_g=c)\cup \Ph{1/2,1}{2/3}(|\xi'|_g=1))\subset h^{-1}(\Ph{0,1}{2/3}(|\xi'|_g=c)\cup \Ph{0,1}{2/3}(|\xi'|_g=1)).$$

 In order to fit the transparent obstacle problem into the framework of Theorem \ref{thm:fullGenerality} with $\alpha=-1$, we only need to check that $V$ is elliptic near $|\xi'|_g=1$ and that 
 $1+\frac{h\sigma(V)}{2\sqrt{|\xi'|_g^2-1}}$
  has the required properties. We start by calculating the symbols of $\mc{B}$, $B$, and $V$. Let $\Xi^E$ be the function given by Lemma \ref{lem:approxInterp} when we replace 1 by $E$ in the eikonal equation for $\rho_0$ and $\theta_0$. 
  \begin{gather*}
  g_E(x,\xi'):=(2Q(x,\xi'))^{1/3}\frac{A_-'(h^{-2/3}\Xi^E)}{A_-(h^{-2/3}\Xi^E)}
 \end{gather*}
 Then,
 \begin{gather*}\sigma(B)= \sigma(h\aleph \No(cz/h))=\begin{cases} -i\aleph \sqrt{c^2-|\xi'|_g^2}&|\xi'|_g\leq c-h^\e\\
 \aleph h^{1/3}g_{c}(x,\xi')&||\xi'|_g-c|\leq h^\e\\
 \aleph \sqrt{|\xi'|_g^2-c^2}&|\xi'|_g\geq c+h^\e
 \end{cases}
 \end{gather*}
 \begin{gather*}
 \sigma(hV)=\begin{cases} i(\sqrt{1-|\xi'|_g^2}-\aleph \sqrt{c^2-|\xi'|_g^2})&|\xi'|_g\leq \min(1,c)-h^\e\\
 i\sqrt{1-|\xi'|_g^2}+\aleph \sqrt{|\xi'|^2_g-c^2}& c+h^\e\leq |\xi'|_g\leq 1-h^\e\\
 -i\aleph\sqrt{c^2-|\xi'|_g^2}-\sqrt{|\xi'|^2_g-1}& 1+h^\e\leq |\xi'|_g\leq c-h^\e\\
 \aleph \sqrt{|\xi'|_g^2-c^2}-\sqrt{|\xi'|_g^2-1}&|\xi'|_g\geq \max(1,c)+h^\e\\
 h^{1/3}\aleph g_{c}+i\sqrt{1-|\xi'|_g^2}&||\xi'|_g-c|\leq h^\e,\,|\xi'|_g\leq 1-h^\e\\
 h^{1/3}\aleph g_{c}-\sqrt{|\xi'|_g^2-1}&||\xi'|_g-c|\leq h^\e,\,|\xi'|_g\geq 1+h^\e\\
-i\aleph\sqrt{c^2-|\xi'|_g^2}- h^{1/3} g_{1}&||\xi'|_g-1|\leq h^\e,\,|\xi'|_g\leq c-h^\e\\
\aleph\sqrt{|\xi'|_g^2-c^2} - h^{1/3} g_{1}&||\xi'|_g-1|\leq h^\e,\,|\xi'|_g\geq c+h^\e\\
 \end{cases}\\
 {}\\
 \sigma(\mc{B})=\frac{\aleph^2|c^2-|\xi'|_g^2|}{2Q}\Psi_{\S}(h^{-2/3}\Xi)(1+\o{}(1)),\quad\quad ||\xi'|_g-1|\leq h^\e
 \end{gather*}
 
 Now, we compute 
  $$1+\frac{h\sigma(V)}{2\sqrt{|\xi'|_g^2-1}}=\begin{cases} \frac{1}{2}+\frac{\aleph}{2}\frac{\sqrt{c^2-|\xi'|_g^2}}{\sqrt{1-|\xi'|_g^2}}&|\xi'|_g\leq \min(1,c)-h^\e\\
\frac{1}{2}+i\frac{\aleph}{2}\frac{\sqrt{|\xi'|_g^2-c^2}}{\sqrt{1-|\xi'|_g^2}}&c+h^\e\leq |\xi'|_g\leq 1-h^\e\\
\frac{1}{2}-i\frac{\aleph}{2}\frac{\sqrt{c^2-|\xi'|_g^2}}{\sqrt{|\xi'|_g^2-1}}&1+Mh^{2/3}\leq |\xi'|_g\leq c-h^\e\\
\frac{1}{2}+\frac{\aleph}{2}\frac{\sqrt{|\xi'|_g^2-c^2}}{\sqrt{|\xi'|_g^2-1}}&\max(c+h^\e,1+Mh^{2/3})\leq |\xi'|_g\\
\frac{1}{2}+i\frac{\aleph}{2}\frac{h^{1/3}g_c}{\sqrt{1-|\xi'|_g^2}}&|c-|\xi'|_g|\leq h^\e,\, |\xi'|_g\leq 1-h^\e\\
\frac{1}{2}+\frac{\aleph}{2}\frac{h^{1/3}g_c}{\sqrt{|\xi'|_g^2-1}}&|c-|\xi'|_g|\leq h^\e,\, |\xi'|_g\geq 1+Mh^{2/3}
\end{cases}$$
Thus, we can see that $V$ is elliptic near $|\xi'|_g=1$ and the transparent obstacle problem fits into the framework of Theorem \ref{thm:fullGenerality}.

In order to finish the proof of Theorem \ref{thm:mainTransparent}, we just need to check a few symbolic properties. First, notice $V=\aleph \No (cz/h)- N_2(z/h)$. Thus,
 $$\sigma(N_2(z/h)+V)=\aleph\sigma(N_2(cz/h))=-ih^{-1}\aleph \sqrt{c^2-|\xi'|_g^2}$$ 
 where we take $\sqrt{-1}=i$. Putting this in \eqref{eqn:cond1} gives that \eqref{eqn:est1} holds when $c>1$ and
\begin{equation*}
\begin{gathered}
\frac{-\Im z}{h}\leq \inf_{|\xi'(q)|_g=1}-\frac{Q}{\aleph\sqrt{c^2-1}}(1\pm \delta)\quad\text{ or }\quad
\frac{-\Im z}{h}\geq \sup_{|\xi'(q)|_g=1}-\frac{Q}{\aleph\sqrt{c^2-1}}(1\mp \delta).
\end{gathered}
\end{equation*}
or when $c<1$ and
\begin{equation*}
\begin{gathered}
\frac{-\Im z}{h}\geq  \delta.
\end{gathered}
\end{equation*}

Next, observe that
$$\sigma(R)=\begin{cases}
\frac{-\sqrt{1-|\xi'|_g^2}+\aleph\sqrt{c^2-|\xi'|_g^2}}{\sqrt{1-|\xi'|_g^2}+\aleph\sqrt{c^2-|\xi'|_g^2}}&|\xi'|_g\leq \min(1,c)-h^\e\\
\frac{i\aleph h^{1/3}g_c-\sqrt{1-|\xi'|^2_g}}{\sqrt{1-|\xi'|_g^2}+i\aleph h^{1/3}g_c}&|c-|\xi'|_g|\leq h^\e\,,\,|\xi'|_g\leq 1-h^\e\\
\frac{-\sqrt{1-|\xi'|_g^2}+i\aleph\sqrt{|\xi'|_g^2-c^2}}{\sqrt{1-|\xi'|_g^2}+i\aleph\sqrt{|\xi'|_g^2-c^2}}&c+h^\e\leq |\xi'|_g\leq 1-h^\e
\end{cases}.
$$
The following geometric lemma completes the proof of Theorem \ref{thm:mainTransparent}.
\begin{lemma}
\label{lem:transparentToGlance}
Fix $N>0$ and let $(x_0,\xi_0)\in S^*\pO$ and suppose that $\{(x_n,\xi_n)\}\subset B^*\pO$ has $(x_n,\xi_n)\to (x_0,\xi_0)$. Then
$$l_N^{-1}r_N\to \begin{cases}\frac{Q(x_0,\xi_0)}{\aleph \sqrt{c^2-1}}&c>1\\
0&c<1
\end{cases}.$$
\end{lemma}
\begin{proof}
The conclusion for $c<1$ is clear since for $|\xi'|_g>c$, $\log |\sigma(R)|^2=0$. So, we need only consider the case $c>1$. 
First, write 
$$|\sigma(R)|^2(x,\xi')=1-\frac{4\sqrt{1-|\xi'|_g^2}}{\sqrt{1-|\xi'|^2_g}+\aleph\sqrt{c^2-|\xi'|^2}}+\O{}(1-|\xi'|^2_g).$$
So, 
\begin{equation}
\label{eqn:logr}\log |\sigma(R)|^2(x,\xi')=-\frac{4\sqrt{1-|\xi'|_g^2}}{\sqrt{1-|\xi'|^2_g}+\aleph\sqrt{c^2-|\xi'|_g^2}}+\O{}(1-|\xi'|^2_g).
\end{equation}
Now, by Lemma \ref{lem:dynStrictlyConvex} 
\begin{gather*} 
\sqrt{1-|\xi'(\beta(q))|_g^2}=\sqrt{1-|\xi'(q)|_g^2}+\O{}(1-|\xi'|_g^2),\quad\quad l(q,\beta(q))=\frac{2}{\kappa(0)}\sqrt{1-|\xi'|_g^2}+\O{}(1-|\xi'|_g^2)
\end{gather*}
where $\kappa(s)$ is the curvature of the unique length minimizing geodesic, $\gamma$ in $\pO$ connecting $\pi_x(q)$ and $\pi_x(\beta(q))$ at the point $\gamma(s)$.
Thus, we have that for $q$ sufficiently close to glancing,
\begin{align*} 
\frac{\log |\sigma(R)(\beta(q))|^2}{2l(q,\beta(q))}
&=-\frac{\kappa(0)}{\aleph\sqrt{c^2-|\xi'|^2_g}}+\O{}(\sqrt{1-|\xi'|^2_g}).
\end{align*}
Moreover, since $\sqrt{1-|\xi'(\beta(q))|_g^2}=\sqrt{1-|\xi'|_g^2}+\O{}(1-|\xi'|_g^2)$ and 
$\kappa(s)=\kappa(0)+\O{}(s)=\kappa(0)+\O{}(\sqrt{1-|\xi'|_g^2})$, 
we have
$$\frac{r_N}{l_N}=-\frac{\kappa(0)}{\aleph\sqrt{c^2-|\xi'|^2_g}}+\O{}(1-|\xi'|_g^2).$$

All that remains to prove is that $\kappa(0)=Q(x,\xi')+\o{}(1)$ as $|\xi'|_g\to 1$. This follows from the fact that the curvature of the geodesic on $\pO$ passing through $x$ in the direction $\xi'$ is $Q(x,\xi')$ together with the fact that 
$$\gamma'(0)-\frac{\xi'}{|\xi'|_g}=\o{}(1).$$
To see this we simply use the fact that a billiards trajectory approaches a geodesic as $|\xi'|_g\to 1$ (see for example \cite{Pet}).
\end{proof}
Together, this discussion proves Theorem \ref{thm:mainTransparent}.

 \section{Application to $\delta$ potentials}
 \label{sec:appDelta}
For the application to $\delta$ potentials, we consider 
$$(-h^2\Delta +h^2V\otimes \delta_{\pO}-z^2)u=0,\quad\quad u\text{ is $z/h$ outgoing}.$$
It is shown in \cite{GS} that this is equivalent to $u=u_1\oplus u_2$ where $u_1=u1_{\Omega}$ and $u_2=u1_{\re^d\setminus\overline{\Omega}}$ solving
\begin{equation}
\label{eqn:deltaTransmit}
\begin{cases}
(-h^2\Delta-z^2)u=0&\text{in }\re^d\setminus\pO\\
u_1-u_2=0&\text{on }\pO\\
h\partial_\nu u_1-h\partial_\nu u_2+hVu_1=0&\text{on }\pO\\
u_2\text{ is $z/h$-outgoing}
\end{cases}
\end{equation}
In this case, $V=V$ (indeed this is the motivation for our notation). For our purposes, we will assume that $V\in h^{\alpha}\Ph{1}{}$ is self adjoint and hence, $\Im V=0$. Moreover, we assume that $\alpha\geq -1$ and $\sigma(V)\geq ch^\alpha$ on $|\xi'|_g=1$ and for any $\delta>0$, there exists $c>0$ so that $\frac{h\sigma(V)}{2\sqrt{|\xi'|^2-1}}> -1+c$ on $|\xi'|_g\geq 1+\delta$. This clearly implies all of the assumptions \eqref{resFree-e:ellipticAssumption}. Theorem \ref{thm:fullGenerality} then yields Theorem \ref{thm:deltaMain} as a Corollary.

\section{Application to boundary stabilization}
\label{sec:appStable}
The application to the boundary stabilization problem \eqref{eqn:stationaryBoundStable} is similar to that for the transmission problem. In particular, note that 
$$1+\frac{i\sigma(hV)}{2\sqrt{1-|\xi'|_g^2}}=-\frac{1}{2}+\frac{a}{2\sqrt{1-|\xi'|_g^2}}$$
and the fact that $a\geq a_0>0$ implies the ellipticity of $V$. Finally, an argument identical to that in Lemma \ref{lem:transparentToGlance} together with Theorem \ref{thm:fullGenerality} gives Theorem \ref{thm:boundaryStable}.

\section{Optimality for the transparent obstacle problem on the circle}
\label{sec:optimalTrans}
For the optimality of Theorem \ref{thm:deltaMain}, see \cite{GalkCircle}. We now show that Theorem \ref{thm:mainTransparent} is optimal in the case of the unit disk in $\re^2$. 
In this case, \eqref{intro-e:transparent} reads
$$\begin{cases}
(-c^2\Delta-\lambda^2)u_1=0&\text{in }B(0,1)\\
(-\Delta-\lambda^2)u_2=0&\text{in }\re^d\setminus\overline{B(0,1)}\\
u_1=u_2&\text{on }|x|=1\\
\partial_ru_1-\aleph\partial_ru_2=0&\text{on }|x|=1\\
u_2\text{ is $\lambda-$outgoing}
\end{cases}.$$
We now expand $u_i$ in Fourier series, writing 
$$u_i(r,\theta)=\sum_{n}u_{i,n}(r)e^{in\theta}.$$
Then, 
\begin{align*} (-c^2\partial_r^2-\frac{c^2}{r}\partial_r+\frac{c^2n^2}{r}-\lambda^2)u_{1,n}(r)&=0&(-\partial_r^2-\frac{1}{r}\partial_r+\frac{n^2}{r}-\lambda^2)u_{2,n}(r)&=0
\end{align*}
Multiplying by $r^2$ and rescaling by $x_1=\lambda c^{-1}r$ for $u_{1,n}$ and $x_2=\lambda r$ for $u_2$, we see that $u_{i,n}(x_i)$ solves Bessel's equation. Together with the outgoing condition for $u_2$ and the fact that $u_1$ is in $L^2$, this implies that 
$$u_{1,n}=K_nJ_n(\lambda c^{-1}r),\quad\quad u_{2,n}=C_nH_n^{(1)}(\lambda r).$$ 
Then, the boundary conditions imply that either $K_n=C_n=0$ or $C_n\neq 0$ and 
$$\frac{K_n}{C_n}=\frac{H_n^{(1)}(\lambda)}{J_n(c^{-1}\lambda)},\quad \quad K_nc^{-1}\lambda J_n'(\lambda c^{-1})-C_n\aleph \lambda {H_n^{(1)}}'(\lambda)=0.$$ 
Rewriting this (and assuming $\lambda\neq 0$) we have 
\begin{equation}\label{eqn:boundaryCircle}
f(\lambda):=c^{-1}J_n'(c^{-1}\lambda)H_n^{(1)}(\lambda)-\aleph {H_n^{(1)}}'(\lambda)J_n(c^{-1}\lambda)=0
\end{equation}

Throughout this section we will refer to microlocalization of the Fourier modes $e^{in\theta}$. 
Notice that for a Fourier mode $u_n=(u_{1,n}(r)\oplus u_{2,n})e^{in\theta}$, the component of the frequency tangent to $\partial B(0,1)$ is given by $n$ and the rest of the osciallitions are normal to the boundary. Naively taking the Fourier transform, we see that if $(-\Delta-\lambda^2)u=0$, then the Fourier support of $u$ is contained in $|\xi|^2=\lambda^2$. Therefore, since $|\Im \lambda |\ll |\Re \lambda|$ the total frequency of the mode is given by $|\Re \lambda|$ and the fraction of frequency tangent to the boundary is given by $n/\Re \lambda$. This can be reinterpreted in terms of the semiclassical wavefront set (with $\Re \lambda= h^{-1}$) of the mode as saying that 
$$\WFh(u_n|_{\pO})\subset\{|\xi'|_g=hn\}.$$
For this reason, we refer to modes with $n\ll |\Re \lambda|$ as normal to the boundary, those with $\e|\Re \lambda|<n<(c^{-1}-\e)|\Re \lambda|$ transverse, and $(c^{-1}-\e)|\Re \lambda|<n$ glancing.
\subsection{Asymptotics of Bessel and Hankel functions}
We collect here some properties of the Airy and Bessel functions that are used in the analysis for the unit disk. These formulae can be found in, for example \cite[Chapter 9,10]{NIST}.

Recall that the Bessel of order $n$ functions are solutions to 
$$z^2y''+zy'+(z^2-n^2)y=0.$$
We consider the two independent solutions $H_n^{(1)}(z)$ and $J_n(z)$.  

\begin{figure}
\centering
\begin{subfigure}[b]{.45\textwidth}
\includegraphics[width=\textwidth]{cin-2alpha-1.eps}
\caption{$\aleph=1$ $c=2$}
\end{subfigure}
\begin{subfigure}[b]{.45\textwidth}
\includegraphics[width=\textwidth]{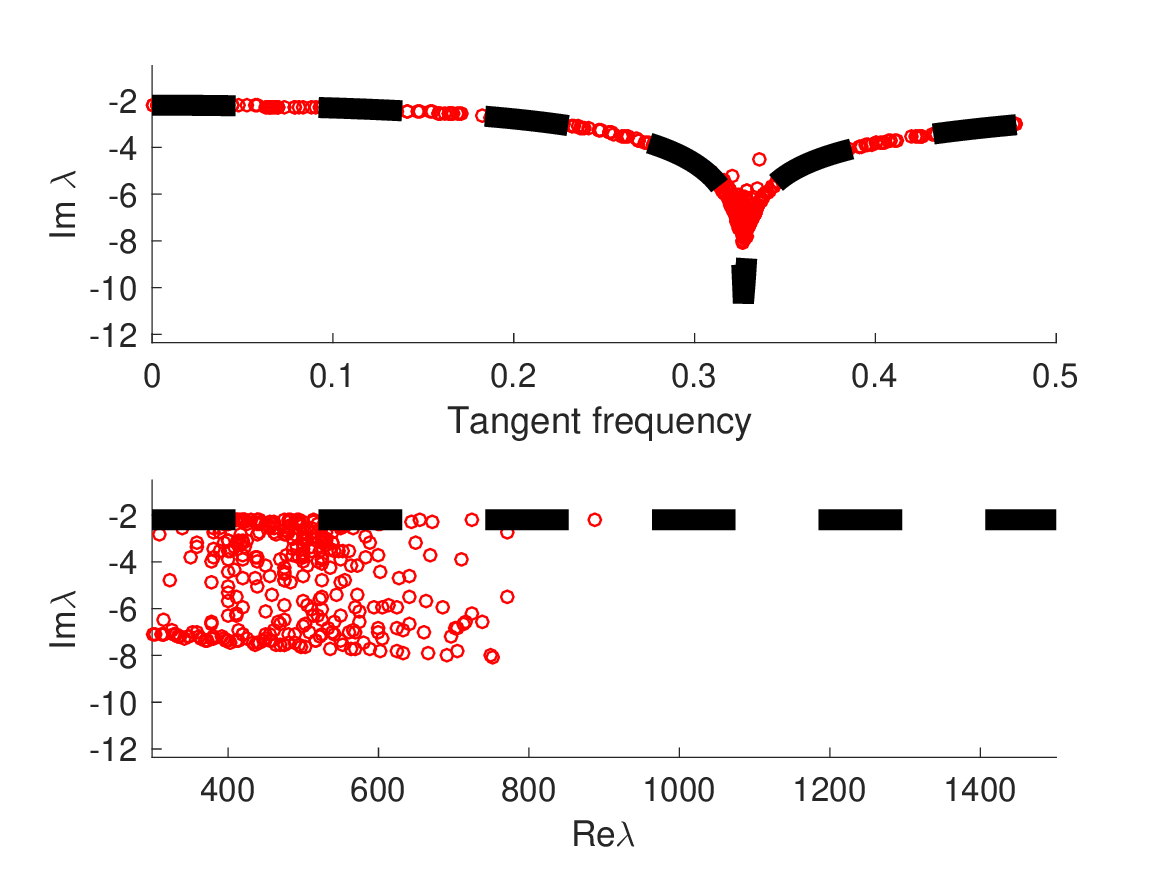}
\caption{$\aleph=0.4$ $c=2$}
\end{subfigure}
\begin{subfigure}[b]{.45\textwidth}
\includegraphics[width=\textwidth]{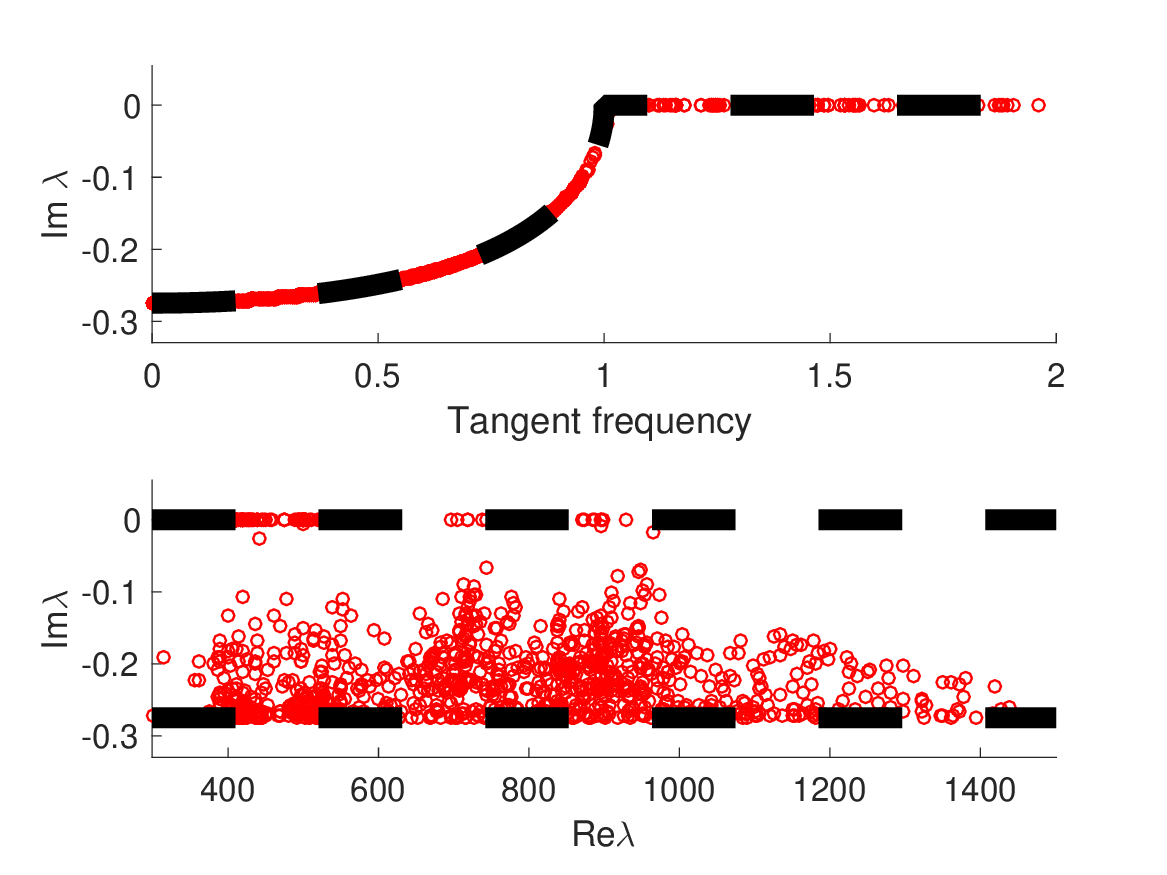}
\caption{$\aleph=1$ $c=0.5$}
\end{subfigure}
\begin{subfigure}[b]{.45\textwidth}
\includegraphics[width=\textwidth]{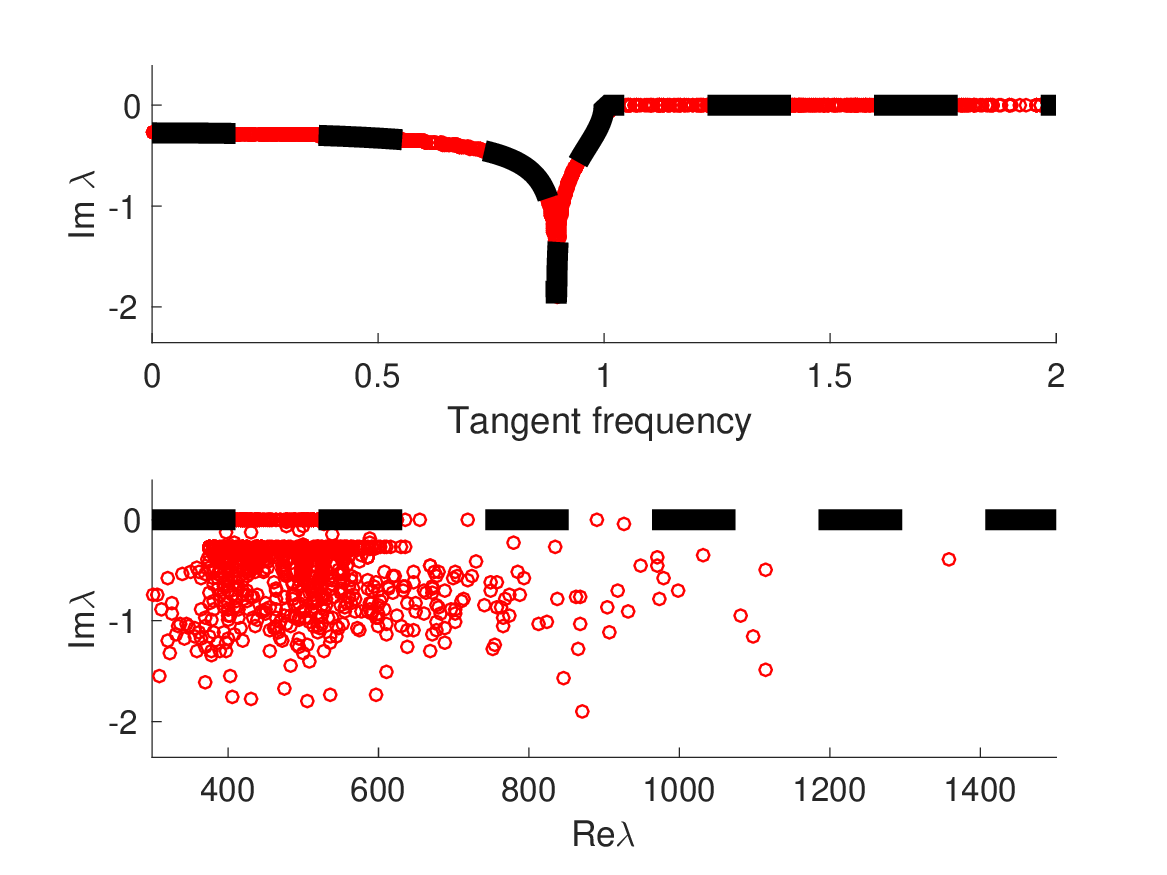}
\caption{$\aleph=4$ $c=0.5$}
\end{subfigure}
\caption{We show numerically computed resonances for the transparent obstacle problem with various $c$ and $\aleph$ when $\Omega=B(0,1)\subset \re^2$. In this case, we expand the solutions to \eqref{intro-e:transparent} as $u_i(r,\theta)=\sum_{n}u_{i,n}(r)e^{in\theta}$ and solve for some of the resonances with $\Re \lambda \sim 800$. In the lower graphs of each of the four subfigures, the blue circles show $\Im \lambda$ vs. $\Re \lambda$. The red lines show the upper and lower bounds for $\Im \lambda$ when $\aleph$ corresponds to TE waves and the upper bounds on $\Im\lambda$ when $\aleph$ corresponds to TM waves from Theorem \ref{thm:mainTransparent}. Notice that by orthogonality of $e^{in\theta}$ and $e^{im\theta}$ for $m\neq n$, the pair $(u_{1,n}e^{in\theta}, u_{2,n}e^{in\theta})$ satisfies \eqref{intro-e:transparent}. In the top graph of each subfigure, the blue circles show $\Im \lambda$ vs. $n/\Re \lambda$ for such pairs.  That is, we plot $\Im \lambda$ vs. the scaled tangent frequency of the resonance state. The red curve shows a plot of $c\frac{r_1}{2l_1}(c\xi)$, the decay rate predicted for a billiards trajectory traveling with scaled tangent frequency $c\xi$. The large spikes in the top graphs occur at the Brewster angle when $\aleph$ corresponds to TM waves. \label{fig:circleLots}}
\end{figure}

We now record some asymptotic properties of Bessel functions. Consider $n$ fixed and $z\to \infty$ 
\begin{align}
J_n(z)&=\left(\frac{1}{2\pi z}\right)^{1/2}\left( e^{i(z-\frac{n}{2}\pi -\frac{1}{4}\pi)}+e^{-i(z-\frac{n}{2}\pi -\frac{1}{4}\pi )}+\O{}(|z|^{-1}e^{|\Im z|})\right)\nonumber\\
H_n^{(1)}(z)&=\left(\frac{2}{\pi z}\right)^{1/2}\left(e^{i(z-\frac{n}{2}\pi -\frac{1}{4}\pi )}+\O{}(|z|^{-1}e^{|\Im z|})\right)\nonumber\\
J_n'(z)&=i\left(\frac{1}{2\pi z}\right)^{1/2}\left( e^{i(z-\frac{n}{2}\pi -\frac{1}{4}\pi)}-e^{-i(z-\frac{n}{2}\pi -\frac{1}{4}\pi )}+\O{}(|z|^{-1}e^{|\Im z|})\right)\nonumber\\
{H_n^{(1)}}'(z)&=i\left(\frac{2}{\pi z}\right)^{1/2}\left(e^{i(z-\frac{n}{2}\pi -\frac{1}{4}\pi )}+\O{}(|z|^{-1}e^{|\Im z|})\right)\nonumber\\
J_n'(c^{-1}z)H_n^{(1)}(z)&=\frac{i\sqrt{c}}{\pi z}\left(e^{i((c^{-1}+1)z-n\pi -\frac{1}{2}\pi )}-e^{-i(c^{-1}-1)z)}+\O{}(|z|^{-1}e^{(c^{-1}+1)|\Im z|}) \right)\label{eqn:asympFixedn1}\\
J_n(c^{-1}z){H_n^{(1)}}'(z)&=\frac{i\sqrt{c}}{\pi z}\left(e^{i((c^{-1}+1)z-n\pi -\frac{1}{2}\pi)}+e^{-i(c^{-1}-1)z}+\O{}(|z|^{-1}e^{(c^{-1}+1)|\Im z|})\right)\label{eqn:asympFixedn2}
\end{align}

Next, we record asymptotics that are uniform in $n$ and $z$ as $n\to \infty$. Let $\zeta=\zeta(z)$ be the unique smooth solution on $0<z<\infty$ to 
\begin{equation}\label{eqn:diffEqzeta}\left(\frac{d\zeta}{dz}\right)^2=\frac{1-z^2}{\zeta z^2}\end{equation} 
 with 
 $$\lim_{z\to 0}\zeta= \infty,\quad\quad\lim_{z\to 1}\zeta= 0,\quad\quad \lim_{z\to \infty}\zeta=-\infty.$$ 
 
  Then
 \begin{align}
 \frac{2}{3}(-\zeta)^{3/2}&=\sqrt{z^2-1}-\asec(z)&1<z<\infty \label{eqn:zetaAsymptotic}\\
 \frac{2}{3}(\zeta)^{3/2}&=\log \left(\frac{1+\sqrt{1-z^2}}{z}\right)-\sqrt{1-z^2}&0<z<1\nonumber\\
\frac{1-z^2}{\zeta z^2}&\to \sqrt[3]{2}&z\to 0\label{eqn:zetaLimit}
 \end{align}
 Let $$Ai(s)=\frac{1}{2\pi}\int _{-\infty }^{\infty }e^{i(\frac{1}{3}t^3+st)}dt$$ 
 for $s\in \re$
be the Airy function solving 
$Ai''(z)-zAi(z)=0.$
Then, $A_-(z)=Ai(e^{2\pi i/3}z)$ is another solution of the Airy equation. 

For $z$ fixed as $n\to \infty$
\begin{align}
J_n(nz)&= \left(\frac{4\zeta}{1-z^2}\right)^{1/4}\left(\frac{Ai(n^{2/3}\zeta)}{n^{1/3}}+\O{}(Ei(5/3,7/3))\right)\nonumber\\
H_n^{(1)}(nz)&=2e^{-\pi i/3}\left(\frac{4\zeta}{1-z^2}\right)^{1/4}\left(\frac{A_-(n^{2/3}\zeta)}{n^{1/3}}+\O{}(E_-(5/3,7/3))\right)\nonumber\\
J_n'(nz)&= -\frac{2}{z}\left(\frac{1-z^2}{4\zeta}\right)^{1/4}\left(\frac{Ai'(n^{2/3}\zeta)}{n^{2/3}}+\O{}(Ei(8/3,4/3))\right)\nonumber
\end{align}
\begin{align}
{H_n^{(1)}}'(nz)&=\frac{4e^{2\pi i/3}}{z}\left(\frac{1-z^2}{4\zeta}\right)^{1/4}\left(\frac{A'_-(n^{2/3}\zeta)}{n^{2/3}}+\O{}(E_-(8/3,4/3))\right)\nonumber\\
\end{align}
\begin{align}
J_n'(c^{-1}nz)H_n^{(1)}(nz)&=\frac{4e^{2\pi i/3}c}{z}\left(\frac{(1-c^{-2}z^2)\zeta(z)}{\zeta(c^{-1}z)(1-z^2)}\right)^{1/4}\left(\frac{Ai'(n^{2/3}\zeta(c^{-1}z))}{n^{2/3}}+\O{}(Ei(8/3,4/3)(c^{-1}z))\right)\nonumber\\
&\quad\quad\quad\left(\frac{A_-(n^{2/3}\zeta)}{n^{1/3}}+\O{}(E_-(5/3,7/3)(z))\right)\nonumber
\end{align}
\begin{align}
J_n(c^{-1}nz){H_n^{(1)}}'(nz)&=\frac{4e^{2\pi i/3}}{z}\left(\frac{(1-z^2)\zeta(c^{-1}z)}{\zeta(z)(1-c^{-2}z^2)}\right)^{1/4}\left(\frac{Ai(n^{2/3}\zeta(c^{-1}z))}{n^{1/3}}+\O{}(Ei(5/3,7/3)(c^{-1}z)\right)\nonumber\\
&\quad\quad\quad\left(\frac{A'_-(n^{2/3}\zeta(z))}{n^{2/3}}+\O{}(E_-(8/3,4/3)(z))\right)\nonumber
\end{align}
where 
\begin{gather*} E_-(s,t)=|A_-'(n^{2/3}\zeta)|n^{-s}+|A_-(n^{2/3}\zeta)|n^{-t}\\
Ei(s,t)=|Ai'(n^{2/3}\zeta)|n^{-s}+|Ai(n^{-2/3}\zeta)|n^{-t}
\end{gather*}

We now record some facts about the Airy functions $Ai$ and $A_-$. 
For $s\in \re$,
$$Ai(s)=e^{-\pi i/3}A_-(s)+e^{\pi i/3}\overline{A_-(s)}$$ 
and hence
\begin{equation} \label{eqn:ImAMinus} 
\Im (e^{-5\pi i/6}A_-(s))=-\frac{Ai(s)}{2}
\end{equation}

Next, we record asymptotics for Airy functions as $z\to \infty$ in the sector $|\Arg z|<\pi/3-\delta$. Many of these asymptotic formulae hold in larger regions, but we restrict our attention to this sector. Let $\eta=2/3z^{3/2}$ where we take principal branch of the square root. Then
\begin{align*}
A_-(z)&=\frac{e^{-\pi i/6}e^{\eta}}{2\sqrt{\pi}z^{1/4}}(1+\O{}(|z|^{-3/2}))&A_-(-z)&=\frac{e^{\pi i/12}e^{i\eta}}{2\sqrt{\pi }z^{1/4}}\nonumber\\
A_-'(z)&=\frac{e^{-\pi i/6} z^{1/4}e^{\eta}}{2\sqrt{\pi }}(1+\O{}(|z|^{-3/2}))&A_-'(-z)&= \frac{e^{-5\pi i/12}z^{1/4}e^{i\eta}}{2\sqrt{\pi}}\nonumber
\end{align*}
\vspace{-0.6cm}
\begin{equation}
\label{eqn:airyAsymps}
\begin{aligned}
Ai(z)&=\frac{z^{-1/4}e^{-\eta}}{2\sqrt{\pi}}(1+\O{}(|z|^{-3/2})) &Ai(-z)&=\frac{z^{-1/4}}{2\sqrt{\pi }}\left(e^{i\eta-i\pi /4}+e^{-i\eta +i\pi /4}+\O{}(|z|^{-3/2}e^{|\Im \eta|})\right) \\
Ai'(z)&=-\frac{z^{1/4}e^{-\eta}}{2\sqrt{\pi }}(1+\O{}(|z|^{-3/2}))&  Ai'(-z)&=\frac{z^{1/4}}{2i\sqrt{\pi }}\left(e^{i\eta-i\pi /4}-e^{-i\eta +i\pi /4}+\O{}(|z|^{-3/2}e^{|\Im \eta|})\right) 
\end{aligned}
\end{equation}

\subsection{Resonances normal to the boundary (fixed $n$)}
First, we fix $n\geq 0$ and examine solutions with $\Re \lambda \to \infty$. We assume that $\aleph \neq c^{-1}$.
Consider \eqref{eqn:boundaryCircle} and apply the asymptotics \eqref{eqn:asympFixedn1} and \eqref{eqn:asympFixedn2} with $\Im \lambda \leq 0$
$$(c^{-1}-\aleph)e^{i((c^{-1}+1)\lambda-n\pi -\frac{1}{2}\pi)}-(c^{-1}+\aleph)e^{-i(c^{-1}-1)\lambda}+\O{}(|z|^{-1}e^{(c^{-1}+1)|\Im z|})=0.$$
So, ignoring the error term for now, we have 
$$\frac{1-\aleph c}{1+\aleph c }e^{i(2c^{-1}\lambda_0-n\pi -\frac{1}{2}\pi)}=1.$$
So, 
\begin{align*}c^{-1}\Im \lambda_0&=\frac{1}{2}\log\left|\frac{1-\aleph c}{1+\aleph c}\right|.&c^{-1}\Re \lambda_0&=\frac{2-\sgn (1-\aleph c)+2n+4k}{4}\pi 
\end{align*}
Taking $\lambda_0$ as above, we have $f(\lambda_0)=\O{}(|\Re \lambda_0|^{-1})$, $|f'(\lambda_0)|\geq c$, and $|f''(\lambda)|\leq C$. for $|\lambda-\lambda_0|<\delta$ for some $\delta>0$.
We now recall Newton's method (see for example \cite[Lemma 4.1]{GalkCircle}
\begin{lemma}
\label{lem:Newton}
Suppose that $z_0\in\mathbb{C}$. Let 
$\Omega:=\{z\in\complex: |z-z_0|\leq \e\}$
and suppose $f:\Omega\to \mathbb{C}$ is analytic. Suppose that 
$$|f(z_0)|\leq a\,,\quad |\partial_zf(z_0)|\geq b\,,\quad \sup_{z\in \Omega}|\partial_z^2f(z)|\leq d.$$
Then if 
\begin{equation}
\label{eqn:condition}
a+d\e^2<\e b<c<1
\end{equation}
there is a unique solution $z$ to $f(z)=0$ in $\Omega$.
\end{lemma}
Using this, we have that there exists a unique solution $\lambda_1$ to $f(\lambda_1)=0$ with $|\lambda_1-\lambda_0|=\O{}(|\Re \lambda_0|^{-1}).$

\subsection{Resonances with non-zero tangent frequency ($\e\Re \lambda\leq n\leq (\min(1,c)-\e)\Re \lambda$)}
In this case, we write 
$$f(\lambda):=c^{-1}J_n'\left(n\frac{c^{-1}\lambda}{n}\right)H_n^{(1)}(\lambda)-\aleph {H_n^{(1)}}'(\lambda)J_n\left(n \frac{c^{-1}\lambda}{n}\right)=0.$$
Write $z=\frac{\lambda}{n}$. Then taking $n\leq (c-\e)\Re \lambda$ and ignoring error terms, $f(\lambda_0)=0$ implies 
\begin{align}\left[\left(\frac{1-c^{-2}z_0^2}{1-z_0^2}\right)^{1/2}-\aleph\right]e^{\frac{4in}{3}((-\zeta(c^{-1}z_0))^{3/2}-i\pi/2}
&=\left[\left(\frac{1-c^{-2}z_0^2}{1-z_0^2}\right)^{1/2}+ \aleph \right]\nonumber\\
-i\frac{\sqrt{c^{-2}z_0^2-1}-\aleph \sqrt{z_0^2-1}}{\sqrt{c^{-2}z_0^2-1}+\aleph\sqrt{z_0^2-1}}&=e^{-\frac{4in}{3}((-\zeta(c^{-1}z_0))^{3/2}}\label{eqn:noError}
\end{align}

Fix $\max(c,1)+\delta<r<\infty$ with $\delta <c^2$ so that 
$$\sqrt{c^{-2}r^2-1}-\aleph \sqrt{r^2-1}\neq 0.$$ 
Let 
$$g(s,n,k):=\sqrt{c^{-2}s^2-1}-\asec(c^{-1}s)+\frac{4k-\sgn(\sqrt{c^{-2}s^2-1}-\aleph \sqrt{s^2-1})}{4n}\pi.$$
Then, fix $q\in \mathbb{Z}_+$ $p\in \mathbb{Z}$ and let $n=qm$ and $k=pm$ so that 
$$g(s,qm,pm)=\sqrt{c^{-2}s^2-1}-\asec(c^{-1}s)+\frac{p}{q}\pi-\frac{\sgn(\sqrt{c^{-2}s^2-1}-\aleph \sqrt{s^2-1})}{4mq}\pi.$$
Then, for any $\e>0$ small enough, there exists $p_\e,q_\e$ so that 
\begin{gather*}
|g(r,qm,pm)|<\e+\O{}(m^{-1})\\
\partial_sg(r,q_\e m,p_\e m)=\frac{\sqrt{c^{-2}r^2-1}}{r}\geq C\sqrt{\delta},\quad\quad
\partial^2_sg(r,q_\e m,p_\e m)=-\frac{r^{-3}}{\sqrt{c^{-2}-r^{-2}}}\leq \frac{C}{\sqrt{\delta}}
\end{gather*}
Therefore, taking $\e$ small enough and $m$ large enugh (depending on $r-c$), there is a solution $r_m$ to $g(r_m,q_\e m,p_\e m)=0$ with $|r-r_m|<C\e$. 

With this $r_m$, let 
$$\lambda_0=mqr_m+i\frac{r_m}{2\sqrt{c^{-2}r_m^2-1}}\log \left|\frac{\sqrt{c^{-2}r_m^2-1}-\aleph \sqrt{r_m^2-1}}{\sqrt{c^{-2}r_m^2-1}+\aleph\sqrt{r_m^2-1}}\right|$$
and $z_0=\lambda_0/mq$.
Let 
$$H(z,n)=\exp\left(-\frac{4in}{3}(-\zeta(c^{-1}z))^{3/2}\right)+i\frac{\sqrt{c^{-2}z^2-1}-\aleph \sqrt{z^2-1}}{\sqrt{c^{-2}z^2-1}+\aleph\sqrt{z^2-1}}.$$
Then, accounting for the errors omitted to obtain \eqref{eqn:noError} there is a function $a(z,n)=\O{}(n^{-2/3})$, analytic in $z$ so that $nz$ is a resonance if and only if 
$$H(z,n)=a(z,n)\left[1+\exp\left(-\frac{4in}{3}(-\zeta(c^{-1}z))^{3/2}\right)\right]$$
Now, using \eqref{eqn:zetaAsymptotic}
\begin{align*} 
-\frac{4imq}{3}(-\zeta(c^{-1}z_0))^{3/2}&=-2imq(\sqrt{c^{-2}z_0^2-1}-\asec(c^{-1}z_0))\\
&=-2imq(\sqrt{c^{-2}r_m^2-1}-\asec(c^{-1}r_m)+i\frac{\sqrt{c^{-2}r_m^2-1}}{r_m}\Im z_0+\O{}((\Im z_0)^2))\\
&=i(2mpi-\frac{\sgn(\sqrt{c^{-2}r_m^2-1}-\aleph \sqrt{r_m^2-1})}{2})\pi\\
&\quad\quad+\log\left|\frac{\sqrt{c^{-2}r_m^2-1}-\aleph \sqrt{r_m^2-1}}{\sqrt{c^{-2}r_m^2-1}+\aleph \sqrt{r_m^2-1}}\right|+\O{}((mq)^{-1})
\end{align*}
So, 
\begin{multline*} \exp\left(-\frac{4imq}{3}(-\zeta(c^{-1}z_0))^{3/2}\right)=\\-\sgn(\sqrt{c^{-2}r_m^2-1}-\aleph \sqrt{c^{-2}r_m^2-1})i\left|\frac{\sqrt{c^{-2}r_m^2-1}-\aleph \sqrt{r_m^2-1}}{\sqrt{c^{-2}r_m^2-1}+\aleph \sqrt{r_m^2-1}}\right|(1+\O{}((mq)^{-1})).
\end{multline*}
So, $H(z_0,mq)=\O{}((mq)^{-1}).$
Moreover,  $|z_0-z|\leq 1$, 
$$|\partial_zH(z,mq)|\geq c mq.$$
Hence, by the implicit function theorem, there exists a resonance $z_1$ with 
\begin{align*} z_1&=z_0+\O{}\left(\frac{\sup_{|z-z_0|\leq 1} \left|a(z,mq)\left[1+\exp\left(-\frac{4imq}{3}(-\zeta(c^{-1}z))^{3/2}\right)\right]\right|}{\inf_{|z-z_0|\leq 1}|\partial_zH(z,mq)|}\right)\\
&=z_0+\O{}((mq)^{-5/3})
\end{align*}
Thus, there is a resonance, $\lambda_1$ with 
$$\lambda_1=mqr_m+i\frac{r_m}{2\sqrt{c^{-2}r_m^2-1}}\log \left|\frac{\sqrt{c^{-2}r_m^2-1}-\aleph \sqrt{r_m^2-1}}{\sqrt{c^{-2}r_m^2-1}+\aleph\sqrt{r_m^2-1}}\right|+\O{}((mq)^{-2/3}).$$

Now, notice that if $|\xi'|_g^{-1}c=r$, then on $B(0,1)$, $l((x,\xi'),\beta(x,\xi'))=2\sqrt{1-r^{-2}c^2}.$
So, 
\begin{align*} l_N^{-1}r_N(x,\xi')&=\frac{1}{4\sqrt{1-r^{-2}c^{2}}}\log \left|\frac{\sqrt{1-r^{-2}c^2}-\aleph\sqrt{c^2-r^{-2}c^2}}{\sqrt{1-r^{-2}c^2}+\aleph \sqrt{c^2-c^2r^{-2}}}\right|^2\\
&=\frac{c^{-1}r}{2\sqrt{c^{-2}r^2-1}}\log \left|\frac{\sqrt{c^{-2}r^2-1}-\aleph\sqrt{r^2-1}}{\sqrt{c^{-2}r^2-1}+\aleph \sqrt{r^2-1}}\right|.
\end{align*}
Now, by construction for any $r$ with $\max(1,c)<r<\infty$ so that $\sqrt{c^{-2}r^2-1}-\aleph \sqrt{r^2-1}\neq 0$ and $\delta$ small enough, we have $|r-r_m|<\delta$ so, taking $m$ large enough,
$$|c^{-1}\Im \lambda -l_N^{-1}r_N(x,\xi')|\leq C\delta.$$
This shows that Theorem \ref{thm:mainTransparent} is sharp. Moreover, when $c<1$, \cite{popovNear} shows that there are sequences of resonances converging to the real axis that have $n\approx c^{-1}\Re \lambda$. 

\begin{remark}
Notice also that 
\begin{equation*}\frac{mq}{c^{-1}\Re \lambda}=cr_m^{-1}=|\xi'|_g.
\end{equation*}
Thus, since \eqref{eqn:boundaryCircle} with parameter $n$ corresponds to a resonant state with $u|_{\pO}=Ae^{in\theta}$, the semiclassical tangent frequency of the resonance state is $cn/\Re \lambda$ when we take $\Re z\sim c$. Plugging this into $cl_N^{-1}r_N(x,\xi')$ gives the decay rate of the resonance state. See also Figures \ref{fig:circle} and \ref{fig:circleLots} for numerically computed resonances in this case.
\end{remark}

\appendix
\section{List of notation}
For the convenience of the reader, we include a list of some of the notation used in this paper.
\begin{multicols*}{2}
\begin{itemize}[leftmargin=*]
\item[-] $\Omega\,$: strictly convex domain with smooth boundary -- Section \ref{sec:Omega}
\item[-] $l(q_1,q_2)\,$: chord length -- \eqref{eqn:defineLength}
\item[-] $l_N(q)\,$: average chord length --  \eqref{eqn:defineLength}
\item[-] $|\xi'|_g\,$ metric induced on $T^*\partial\Omega$ -- Section \ref{sec:Omega}
\item[-] $\beta:B^*\partial\Omega\to B^*\partial\Omega$: the billiard ball map -- Section \ref{sec:billiard}
\item[-] $\Ph{m}{\delta}(M)\,$: semiclassical pseudifferential operator classes -- Section \ref{sec:semiclassicalPreliminaries}
\item[-] $S^m_\delta(T^*M)\,$: symbol classes -- \eqref{eqn:defSymbol}
\item[-] $\sigma:\Ph{m}{\delta}(M)\to S^m_{\delta}(T^*M)\,$: the symbol map -- \eqref{eqn:princSymbol}
\item[-] $Ai$, $A_i$, $\Phi_-$, $\zeta_i\,$: Airy related functions -- Section \ref{sec:potential}, \eqref{e:Airy1} 
\item[-] $Q(x',\xi')\in C^\infty(T^*\partial\Omega)\,$: the symbol of the second fundamental form -- Section \ref{sec:potential}
\item[-] $N_2(z/h)\,$: the outgoing Dirichlet to Neumann Map -- Section \ref{sec:mainThm}
\item[-] $G(z/h)\,$: the single layer operator -- Section \ref{sec:mainThm}
\item[-] $G_B$, $G_\Delta\,$: decomposition of $G$ -- Lemma \ref{lem:decompose}
\item[-] $\Ph{k_1,k_2}{\delta}(M;\Sigma)$, $S^{k_1,k_2}_{\delta}(M;\Sigma)\,$: second microlocal operators and symbols -- Section \ref{sec:secondMicrolocal}
\item[-] $R\,$: the reflection operator -- \eqref{eqn:reflect}
\item[-] $T\,:$ the transition operator -- \eqref{eqn:transition}
\item[-] $\oph \,:$ quantization operator -- Section \ref{sec:semiclassicalPreliminaries}
\item[-] $r_N\,:$ the average reflectivity -- \eqref{eqn:defineAverageReflection}
\item[-] $\tilde{\sigma}\,$: the compressed shymbol -- Section \ref{sec:shymbol}
\item[-] $I_A(q)\,$: the order of $A$ at $q$ -- Section \ref{sec:shymbol}
\item[-] $H_h^m\,:$ semiclassical Sobolev spaces -- \eqref{e:semiSob}
\item[-] $\S$, $\D$, respectively the single and double layer operators -- \eqref{e:layerDef}
\item[-] $\O{}(\cdot)$ and $\o{}(\cdot)$ --\eqref{e:oDef}
\item[-] $\WFh$ the semiclassical wavefront set -- Definition \ref{d:wavefront}
\item[-] $\Psi_{\S}$, $\Psi_{\D\S}$, $\Psi_{\D}$ symbols of layer potentials -- \eqref{e:layerSymbDef}
\end{itemize}
\end{multicols*}

\bibliographystyle{alpha} 
\bibliography{references.bib}

\end{document}